\pdfoutput=1
\documentclass[11pt,twoside,reqno]{amsart}
\usepackage{amsfonts,amssymb,amsmath,amsthm,amsxtra}
\usepackage{mathtools}
\usepackage{bm}
\usepackage{xparse}
\usepackage{mathtools}
\usepackage{dsfont}
\usepackage{bbm}
\usepackage{mathrsfs}
\usepackage{tikz-cd}
\usepackage{enumitem}
\usepackage{comment}

\mathtoolsset{showonlyrefs,showmanualtags}

 \usepackage[top=1.5in, bottom=1.5in, left=1.23in, right=1.23in]{geometry}
\usepackage[all]{xy}
\usepackage{amsmath}
\usepackage{amssymb}
\usepackage{amsthm}
\usepackage{amscd}

\usepackage{scalerel,stackengine}
\stackMath
\newcommand\reallywidehat[1]{%
\savestack{\tmpbox}{\stretchto{%
  \scaleto{%
    \scalerel*[\widthof{\ensuremath{#1}}]{\kern-.7pt\bigwedge\kern-.7pt}%
    {\rule[-\textheight/2]{1.5ex}{\textheight}}
  }{\textheight}%
}{0.9ex}}%
\stackon[2.25pt]{#1}{\tmpbox}%
}

\newcommand{\E}{\mathbb{E}}
\newcommand{\F}{\mathcal{F}}
\newcommand{\A}{{\mathbb{A}}}
\newcommand{\Schwartz}{{\bm S}}

\newcommand{\V}{\bm{V}}
\renewcommand{\mod}{{\ \mathrm{mod}\ }}

\newcommand{\PP}{\mathbb  P}

\newcommand{\I}{\mathbb I}
\newcommand{\RR}{\mathbb R}
\newcommand{\R}{\RR}
\newcommand{\TT}{\mathbb T}
\newcommand{\NN}{{\mathbb N}}
\newcommand{\N}{\NN}
\newcommand{\Z}{\mathbb Z}
\newcommand{\QQ}{{\mathbb Q}}
\newcommand{\Q}{\QQ}
\newcommand{\C}{{\mathbb C}}
\newcommand{\D}{{\mathbb D}}
\newcommand{\G}{{\mathbb G}}

\newcommand{\eps}{\varepsilon}

\newcommand{\naive}{\mathrm{naive}}
\newcommand{\lrg}{{>}}
\newcommand{\sml}{{\leq}}
\newcommand{\T}{{\mathrm{T}}}
\newcommand{\B}{{\mathrm{B}}}

\newcommand{\Sample}{{\mathcal{S}}}
\newcommand{\Proj}{{\mathcal{P}}}

\newcommand{\nn}{\mathrm{n}}
\newcommand{\ind}[1]{\mathds{1}_{{#1}}}

\numberwithin{equation}{section}

\newtheorem{theorem}[equation]{Theorem}
\newtheorem{definition}[equation]{Definition}
\newtheorem{proposition}[equation]{Proposition}

\newtheorem{corollary}[equation]{Corollary}
\newtheorem{lemma}[equation]{Lemma}

\theoremstyle{definition}
\newtheorem{remark}[equation]{Remark}
\newtheorem{example}[equation]{Example}

\DeclareMathOperator{\Height}{h}

\DeclareMathOperator{\Log}{Log}

\usepackage{hyperref} 
\hypersetup{
    colorlinks=true,       
    linkcolor=blue,          
    citecolor=magenta,        
    filecolor=magenta,      
    urlcolor=cyan           
}

\begin{document}
\title[Pointwise ergodic theorems for bilinear polynomial
averages]{Pointwise ergodic theorems for non-conventional \\ bilinear
polynomial averages}

\author{Ben Krause}
\address[Ben Krause]{
Department of Mathematics,
Princeton University\\
Princeton, NJ 08544, USA}
\email{bkrause@princeton.edu}

\author{Mariusz Mirek }
\address[Mariusz Mirek]{
Department of Mathematics,
Rutgers University,
Piscataway, NJ 08854-8019, USA \\
\&
Instytut Matematyczny,
Uniwersytet Wroc{\l}awski,
Plac Grunwaldzki 2/4,
50-384 Wroc{\l}aw
Poland}
\email{mariusz.mirek@rutgers.edu}

\author{Terence Tao}
\address[Terence Tao]{
Department of Mathematics,
University of California Los Angeles\\
Los Angeles, CA 90095-1555, USA}
\email{tao@math.ucla.edu}

\newcommand{\bk}[1]{{\color{red}{#1}}}

\begin{abstract}
We establish convergence in norm and pointwise almost everywhere for the non-conventional (in the sense of Furstenberg) bilinear polynomial ergodic averages
\[ A_N(f,g)(x) \coloneqq \frac{1}{N} \sum_{n =1}^N f(T^nx) g(T^{P(n)}x)\]
as $N \to \infty$, where $T \colon X \to X$ is a measure-preserving transformation of a $\sigma$-finite measure space $(X,\mu)$, $P(\mathrm{n}) \in \mathbb Z[\mathrm{n}]$ is a polynomial of degree $d \geq 2$, and $f \in L^{p_1}(X), \ g \in L^{p_2}(X)$ for some $p_1,p_2 > 1$ with  $\frac{1}{p_1} + \frac{1}{p_2} \leq 1$.
We also establish an  $r$-variational inequality for these averages (at lacunary scales) in the optimal range $r > 2$. We are also able to ``break duality'' by handling some ranges of exponents $p_1,p_2$ with $\frac{1}{p_1}+\frac{1}{p_2} > 1$, at the cost of increasing $r$ slightly.

This gives an affirmative answer to Problem 11 from Frantzikinakis' open problems survey for the Furstenberg--Weiss averages (with $P(\mathrm{n})=\mathrm{n}^2$), which is a bilinear variant of Question 9 considered by Bergelson in his survey on Ergodic Ramsey Theory from 1996. This also gives a contribution to the Furstenberg--Bergelson--Leibman conjecture.  Our methods combine techniques from harmonic analysis with the recent inverse theorems of Peluse and Prendiville in additive combinatorics. At large scales, the harmonic analysis of the adelic integers $\mathbb A_{\mathbb Z}$ also plays a role.
\end{abstract}

\date{\today}

\thanks{Ben Krause was partially supported by the Simons Foundation Analysis and Geometry Research Grant,  Mariusz Mirek was
partially supported by Department of Mathematics at Rutgers
University, and by the National Science Centre in Poland, grant Opus
2018/31/B/ST1/00204.  Terence Tao was partially supported by NSF grant DMS-1764034 and by a Simons Investigator Award.
MSC class: 37A30, 37A46, 42A45, 42A50, 42A85, 43A25, 11L03, 11L07, 11L15, 11P55.}

\maketitle

\begin{center}
\large\textit{Dedicated to the memory of Jean Bourgain and  Elias M. Stein.}    
\end{center}

\tableofcontents

\section{Introduction}
\label{sec:1}

\subsection{Non-conventional polynomial ergodic averages}  

Define a \emph{measure-preserving system} to be a triple $X = (X,\mu,T)$, where $X = (X,\mu)$ is a $\sigma$-finite measure space, and $T \colon X \to X$ is an invertible bimeasurable map which is measure-preserving in the sense that $\mu(T(E)) = \mu(E)$ for all measurable $E$.  In the literature it is common to also require $X = (X,\mu)$ to have finite measure (and often one normalizes $(X,\mu)$ to be a probability space), but our main theorem will not require this hypothesis.

Let $\Z[\nn]$ denote the space of all formal polynomials $P(\nn)$ in one indeterminate $\nn$ with integer coefficients.  Such a polynomial $P(\nn) \in \Z[\nn]$ can of course be identified with a function $P\colon \Z \to \Z$, thus for instance $\nn$ is identified with the identity function $n \mapsto n$ and $\nn^2$ is identified with the quadratic function $n \mapsto n^2$.  (Later on we will also identify $P$ with maps $P \colon R \to R$ on other commutative rings $R$, such as the reals $\R$, the $p$-adic integers $\Z_p$, or the profinite integers $\hat \Z$.) Given any polynomials $P_1(\nn),\dots,P_k(\nn) \in \Z[\nn]$, measurable functions $f_1,\dots,f_k \in L^0(X)$ (see Section \ref{notation-sec} for a definition of this space), and a real number $N \geq 1$, we can define the non-conventional polynomial ergodic average $A_{N;X}^{P_1(\nn),\dots,P_k(\nn)}(f_1,\dots,f_k) \in L^0(X)$ by the formula
\begin{equation}\label{anx} A_{N;X}^{P_1(\nn),\dots,P_k(\nn)}(f_1,\dots,f_k)(x) \coloneqq \E_{n \in [N]} f_1(T^{P_1(n)} x) \dots f_k(T^{P_k(n)} x),
\end{equation}
where $\E_{n \in [N]} f(n) \coloneqq \frac{1}{\lfloor N\rfloor} \sum_{n=1}^{\lfloor N\rfloor} f(n)$ (see Section \ref{notation-sec} for a more general definition of this averaging notation).  The terminology ``non-conventional'' for such multilinear averages was introduced in \cite{Fur1} and is now standard in the ergodic theory literature (see e.g., \cite{FurWei, HK}). We will usually abbreviate $A_{N;X}^{P_1(\nn),\dots,P_k(\nn)}$ as $A_N^{P_1,\dots,P_k}$ or even $A_N$ when this does not cause confusion.  As $A_N$ only depends on the integer part $\lfloor N \rfloor$ of $N$, one could have restricted $N$ to the positive integers $\Z_+$; however it will be convenient to generalize to real-valued $N$ in order to use certain scaling arguments.

\begin{example}[Integer shift system]\label{integer-shift}  The \emph{integer shift system} $\Z = (\Z,\mu_\Z,T_\Z)$ is the set of integers $\Z$ equipped with counting measure $\mu_\Z$ and the shift $T_\Z(x) \coloneqq x-1$.  For our purposes, this system will be ``universal'' for all other measure-preserving systems, in a sense formalized by the Calder\'on transference principle; see Proposition \ref{transf}(ii).  This will be a particularly convenient system to work in due to the extensive Fourier-analytic structure available on the additive group of integers $\Z$, which can be connected in particular (in the ``major arc'' regime) to the corresponding Fourier-analytic structures on other locally compact abelian groups, such as the adelic integers $\A_\Z$; see Figure \ref{fig:major-l2}.  In this system one has
$$ A_N^{P_1,\dots,P_k} (f_1,\dots,f_k)(x) = \E_{n \in [N]} f_1(x-P_1(n)) \dots f_k(x-P_k(n)).$$
\end{example}

Our main results will concern the bilinear averages
$$ A_N^{\nn, P(\nn)}(f,g)(x) = \E_{n \in [N]} f(T^n x) g(T^{P(n)} x) = \frac{1}{\lfloor N\rfloor} \sum_{n=1}^{\lfloor N\rfloor} f(T^n x) g(T^{P(n)} x)$$
for a given polynomial $P(\nn) \in \Z[\nn]$, but as motivation we shall also discuss the classical ergodic average
$$ A_N^\nn f(x) = \E_{n \in [N]} f(T^n x) = \frac{1}{\lfloor N\rfloor} \sum_{n=1}^{\lfloor N\rfloor} f(T^n x)$$
and the linear polynomial average
$$ A_N^{P(\nn)} f(x) = \E_{n \in [N]} f(T^{P(n)} x) = \frac{1}{\lfloor N\rfloor} \sum_{n=1}^{\lfloor N\rfloor} f(T^{P(n)} x).$$

A central problem in ergodic theory is to understand convergence in norm and pointwise almost everywhere for the non-conventional polynomial ergodic averages \eqref{anx} as $N\to\infty$. This line of investigations has been initiated in the early 1930's by von Neumann's mean ergodic theorem \cite{vN} and Birkhoff's pointwise ergodic theorem \cite{BI} (see Theorem \ref{clas}) and led to profound generalizations such as Bourgain's polynomial pointwise ergodic theorem  \cite{B1, B2, B3} (see Theorem \ref{lin-poly}) and Furstenberg's ergodic proof \cite{Fur0} of Szemer{\'e}di's theorem \cite{Sem1}. Furstenberg's proof was also the starting point of the multiple/multilinear ergodic  theory (see Theorem \ref{walsh-thm} and Theorem \ref{twolin}) arising in  ergodic Ramsey theory that also motivates this paper. 
Pointwise convergence is the most natural as well as the most difficult type of convergence to
establish. It requires sophisticated tools in analysis, ergodic theory and probability. Especially, the context of pointwise convergence of \eqref{anx} will require to understand quantitative forms of pointwise convergence, which we briefly illustrate below.

Given some non-conventional average $A_N(f_1,\dots,f_k)$ of some functions $f_1,\dots,f_k$, with each $f_i$ belonging to some Lebesgue space $L^{p_i}(X)$, one can pose the following questions:
\begin{itemize}
\item[(i)] (Norm convergence)  Does $A_N(f_1,\dots,f_k)$ converge in $L^p(X)$ norm as $N \to \infty$ for some exponent $p>0$?
\item[(ii)] (Almost everywhere convergence)  Does $A_N(f_1,\dots,f_k)$ converge pointwise almost everywhere (with respect to $\mu$, of course) as $N \to \infty$?
\item[(iii)] (Maximal inequality) Can one bound the $L^p(X)$ norm of the maximal function $\sup_{N \in \Z_+} |A_N(f_1,\dots,f_k)|$, (or equivalently, the $L^p(X;\ell^\infty)$ norm of the sequence of averages $(A_N(f_1,\dots,f_k))_{N \in \Z_+}$) for some $p>0$ in terms of the norms $\|f_i\|_{L^{p_i}(X)}$? More precisely, one is concerned with the following bound
\begin{align}
\label{eq:10}
\|\sup_{N \in \Z_+} |A_N(f_1,\dots,f_k)|\|_{L^p(X)} \lesssim_{p_1,\ldots, p_k, p} \|f_1\|_{L^{p_1}(X)} \dots \|f_k\|_{L^{p_k}(X)}.
\end{align}
(See Section \ref{notation-sec} for the asymptotic notation used in this paper.)
\item[(iv)] (Variational inequality)  Can one bound the $L^p(X)$ norm of the $r$-variational norm $\|(A_N(f_1,\dots,f_k))_{N \in \Z_+}\|_{\V^r}$, (or equivalently, the  $L^p(X; \V^r )$ norm of the sequence of averages $(A_N(f_1,\dots,f_k))_{N \in \Z_+}$) for some $p>0$ and some $1 \leq r < \infty$ in terms of the norms $\|f_i\|_{L^{p_i}(X)}$?
More precisely, one is concerned with the following bound
\begin{align}
\label{eq:8}
\quad\big\|\|(A_N(f_1,\dots,f_k))_{N \in \Z_+}\|_{\V^r}\big\|_{L^p(X)} \lesssim_{p_1,\ldots, p_k, p,r} \|f_1\|_{L^{p_1}(X)} \dots \|f_k\|_{L^{p_k}(X)}.
\end{align}
The $r$-variational norm is defined by
\[
\qquad \qquad \|(A_N(f_1,\dots,f_k))_{N \in \Z_+}\|_{\V^r}:=\sup_{N \in \Z_+} |A_N(f_1,\dots,f_k)|+\|(A_N(f_1,\dots,f_k))_{N \in \Z_+}\|_{V^r},
\]
where $\|(A_N(f_1,\dots,f_k))_{N \in \Z_+}\|_{V^r}$ is given by the following expression
\begin{align}
\label{eq:9}
\sup_{J\in\Z_+} \sup_{\substack{N_{0} \leq \dotsb \leq N_{J}\\ N_{j}\in\Z_+}}
\Big(\sum_{j=0}^{J-1}  |A_{N_{j+1}}(f_1,\dots,f_k)-A_{N_{j}}(f_1,\dots,f_k)|^{r} \Big)^{1/r},
\end{align}
here the supremum is taken over all finite increasing sequences in $\Z_+$. (See Section \ref{notation-sec} for a more general definition of the variational norm $\V^r$ and its properties.)
\end{itemize}

These questions are all related to each other.  For instance, if  variational inequality \eqref{eq:8} holds, then one automatically has a maximal inequality \eqref{eq:10}. Moreover, \eqref{eq:8} immediately ensures that the quantity in \eqref{eq:9} is finite  almost everywhere, which in turn implies  almost everywhere convergence
of the sequence $(A_N(f_1,\dots,f_k))_{N \in \Z_+}$ as $N\to \infty$. Norm convergence then also follows (for $p<\infty$) by \eqref{eq:10} and the dominated convergence theorem.  This variational norm approach to ergodic theorems was advocated in particular by Bourgain \cite{B2}, and is very useful in pointwise convergence problems with arithmetic features.

We say that a tuple $(p_1,\dots,p_k,p)$ of exponents is \emph{H\"older} if $\frac{1}{p} = \frac{1}{p_1} + \dots + \frac{1}{p_k}$ and \emph{Banach} if $p_1,\dots,p_k,p \geq 1$.  If a tuple $(p_1,\dots,p_k,p)$ is both H\"older and Banach, then from H\"older's inequality and the triangle inequality in the Banach space $L^p(X)$ one has
\begin{equation}\label{anf}
\| A_N(f_1,\dots,f_k)\|_{L^p(X)} \leq \|f_1\|_{L^{p_1}(X)} \dots \|f_k\|_{L^{p_k}(X)}
\end{equation}
regardless of the choice of polynomials $P_1(\nn),\dots,P_k(\nn)$.  Thus it is natural to restrict attention to the case of exponents that are both H\"older and Banach.  The H\"older hypothesis is particularly essential for ergodic theory applications as it is needed in order to apply the Calder\'on transference principle; see Proposition \ref{transf}(ii). However, we will be able to ``break duality'' in our main result by allowing certain non-Banach exponents $p<1$ while still maintaining the H\"older property; see Section \ref{pless}.  On the integer shift model $\Z$, the estimates become trivial (and of little use) in the super-H\"older regime $\frac{1}{p} < \frac{1}{p_1} + \dots + \frac{1}{p_k}$, and false in the opposite sub-H\"older regime $\frac{1}{p} > \frac{1}{p_1}+\dots+\frac{1}{p_k}$; see Remark \ref{need-hold}.

It is technically convenient to sparsify the set of scales $N$ that one is ranging over to define a maximal or variational function.  For instance, one could replace the positive integers $\Z_+$ by the dyadic integers
$$ 2^\N \coloneqq \{2^k: k \in \N \}.$$
More generally, we can work with sets $\D = \{ N_1,N_2,\dots\}$ of positive reals $1 \leq N_1 < N_2 < \dots$ that are \emph{$\lambda$-lacunary} for some $\lambda > 1$, in the sense that
$$ N_{j+1}/N_j > \lambda$$
for all $j \in \Z_+$; one defines $\lambda$-lacunarity for finite sequences $\{N_1,\dots,N_k\}$ of positive reals $1 \leq N_1 < \dots < N_k$ in a similar fashion.  Variational estimates on such lacunary sets are sometimes referred to as ``long variation estimates'' in the literature; they are somewhat weaker than full variation estimates but are often still sufficient for applications such as demonstrating almost everywhere convergence.

We will only concern ourselves in this paper with the existence of a limit of an ergodic average, and not attempt to compute what the limiting average actually is.  The nature of this limiting average is now fairly well understood (at least when $f_1,\dots,f_k \in L^\infty(X)$ and $X$ has finite measure) thanks to the theory of characteristic factors, and the equidistribution theory of nilmanifolds; see for instance \cite{Ber1}, \cite{Ber2}, \cite{Fra} for further discussion.
In particular, for a description of the limit in the case when the polynomials all have distinct degrees, which is of course the case of primary interest here,
we refer to \cite{CFH}. We also remark that the limit in this case is determined entirely by the projection of the functions to the rational factor (the factor spanned by periodic functions), which is the ergodic theory analogue of the ``major arc'' component of the functions. These results are also related to recurrence and Roth and Szemer\'edi type theorems (see e.g., \cite{Fur0}, \cite{Fur1}, \cite{Fur2}, \cite{BL1}, \cite{Sem1}), which also motivate this paper, but we will not discuss these topics further here.

\subsection{Linear averages}

We now recall the standard ergodic theorems for the classical ergodic averages $A^\nn_N$:

\begin{theorem}[Classical ergodic averages]\label{clas}  Let $X = (X,\mu,T)$ be a measure-preserving system, and let $f \in L^p(X)$ for some $1 \leq p \leq \infty$.
\begin{itemize}
    \item [(i)] (Mean ergodic theorem)  If $1 < p < \infty$, then $A^\nn_N f$ converges in $L^p(X)$ norm.
    \item [(ii)]  (Pointwise ergodic theorem)  If $1 \leq p < \infty$, then $A^\nn_N f$ converges pointwise almost everywhere.
    \item[(iii)] (Maximal ergodic theorem)  If $1 < p \leq \infty$, one has 
    $$\| (A^\nn_N f)_{N \in \Z_+} \|_{L^p(X; \ell^\infty)} \lesssim_p \|f\|_{L^p(X)}.$$ 
    \item[(iv)] (Variational ergodic theorem) If $1 < p < \infty$ and $r > 2$, then one has
    $$\| (A^\nn_N f)_{N \in \Z_+} \|_{L^p(X; \V^r)} \lesssim_{p,r} \|f\|_{L^p(X)}.$$
\end{itemize}
\end{theorem}

\begin{proof}
Parts (i)-(iii) are standard, particularly in the case when $X$ has finite measure, and are due to von Neumann \cite{vN}, Birkhoff \cite{BI}, and Hopf \cite{Hopf}; the maximal inequality (for $\sigma$-finite $X$) can also be established by transference to the integer shift case $(\Z,\mu_\Z,T_\Z)$ and then applying the Hardy--Littlewood maximal inequality.  (This also gives a weak-type endpoint for (iii).)  The variational estimate was established by Bourgain \cite[Corollary 3.26]{B2} in the $p=2$ case, and the general case was established in \cite{J+}; this estimate can then be used to recover the mean and pointwise ergodic theorems in the $\sigma$-finite case as mentioned previously.
\end{proof}

We have (slightly weaker) analogues of these results for other linear polynomial averages:

\begin{theorem}[Linear polynomial averages]\label{lin-poly}  Let $X = (X,\mu,T)$ be a measure-preserving system, let $P(\nn) \in \Z[\nn]$, and let $f \in L^p(X)$ for some $1 \leq p \leq \infty$.
\begin{itemize}
    \item [(i)] (Mean ergodic theorem)  If $1 < p < \infty$, then $A^{P(\nn)}_N f$ converges in $L^p(X)$ norm.
    \item [(ii)]  (Pointwise ergodic theorem)  If $1 < p < \infty$, then $A^{P(\nn)}_N f$ converges pointwise almost everywhere.
    \item[(iii)] (Maximal ergodic theorem)  If $1 < p \leq \infty$, one has
    \begin{align}\label{eq:6}
\| (A^{P(\nn)}_N f)_{N \in \Z_+} \|_{L^p(X;\ell^\infty)} \lesssim_{p,P} \|f\|_{L^p(X)}.
    \end{align}
    \item[(iv)] (Variational ergodic theorem) If $1 < p < \infty$ and $r > 2$, then one has
    \begin{equation}\label{var-poly}
    \| (A^{P(\nn)}_N f)_{N \in \Z_+} \|_{L^p(X; \V^r)} \lesssim_{p,r,P} \|f\|_{L^p(X)}.
    \end{equation}
\end{itemize}
\end{theorem}

\begin{proof} Part (i) follows for $p=2$ by a routine application of the spectral theorem (or one can invoke Theorem \ref{walsh-thm} below), and the other values of $p$ then follow from a density argument.  Parts (ii), (iii) were established by Bourgain \cite[Theorem 1]{B2} (see also \cite{B1}, \cite{B3}).  Part (iv) was established in the $p=2$ case by the first author in \cite[Proposition 1.5]{Kr} by adapting the methods of Bourgain, and in full generality by the second author and his collaborators in \cite{MST}, see also \cite{MSZ3}. In \cite[\S 8]{Kr} it is also shown that \eqref{var-poly} fails at the endpoint $p=r=2$. For $p=1$, in contrast to Theorem \ref{clas}(ii), pointwise convergence  in  Theorem \ref{lin-poly}(ii) fails for any monomial $P(\nn)=\nn^d$ of degree  $d\ge 2$, as was shown in \cite{BM, LaV1}.
\end{proof}

Theorem \ref{lin-poly} is proven via the circle method.  The implementation of this method can be summarized in the following two sentences:
\begin{itemize}
\item[(i)]  Plancherel's theorem and Weyl sum estimates are used to control the contribution of minor arcs.
\item[(ii)]  Multifrequency harmonic analysis is used to control the contribution of major arcs.
\end{itemize}

We now briefly sketch some more details of Bourgain's proof for maximal inequality \eqref{eq:6}.  The key estimate to establish is \eqref{eq:6} when $p=2$ and $(X,\mu, T)$ is the integer shift system, where $N$ is restricted to a finite lacunary set $\I$, and with $f$ assumed to be in the Schwartz--Bruhat space $\Schwartz(\Z) \subset \ell^1(\Z)$ to avoid technicalities, see Section \ref{abstract-sec} for a definition of this space.
In this setting we have the convenient Fourier representation
$$ \F_\Z A^{P(\nn)}_N f(\xi) = \varphi_{N,\Z}(\xi) \F_\Z f(\xi)$$
for any $\xi \in \TT$, where using the averaging notation \eqref{eq:12} the symbol  $\varphi_{N,\Z}(\xi)$ is given by
\begin{equation}\label{symbol-def}
\varphi_{N,\Z}(\xi) \coloneqq \E_{n \in [N]} e( P(n) \xi ),
\end{equation}
where $e(\theta):=e^{2\pi i \theta}$ and the Fourier transform $\F_\Z f$ are defined in Section \ref{abstract-sec}.  Standard Weyl sum estimates (see \cite[Lemma 20.3, p. 462]{IK}) reveal that for some small  $\delta, \varepsilon>0$ one has 
\begin{align}
\label{eq:13}
|\varphi_{N,\Z}(\xi)|\lesssim_P N^{-\delta},
\end{align}
unless $\xi$ is in a \emph{major arc}, which roughly speaking means that $\xi$ is close to $\frac{a}{q} \mod 1$ for some $a \in \Z$ and some small positive integer $1\le q\le N^{\varepsilon}$.  One can then use \eqref{eq:13} and  Plancherel's theorem to dispose of the \emph{minor arc} case when $\xi$ is not in a major arc, and then after a dyadic decomposition the main task is to establish an estimate roughly of the shape
$$    \| ( A^{P(\nn)}_N f )_{N \in \I} \|_{\ell^2(\Z; \ell^\infty)} \lesssim_{r,P,\lambda} 2^{-cl} \|f\|_{\ell^2(\Z)}.
$$    
for all $l \in \N$, $\lambda > 1$ and some constant $c = c_{r,P} >0$, where $\I \subset [1,+\infty)$ is an arbitrary finite $\lambda$-lacunary set and the Fourier transform of $f$ is restricted to the set of ``$l$-major arc'' frequencies $\xi$ of the form $\xi = \frac{a}{q} + O( 2^{-10l} ) \mod 1$ (say) for some $q \sim 2^l$.  (Informally, this is morally equivalent by the uncertainty principle to $f$ being a linear combination of functions that are approximately constant on arithmetic progressions of spacing $q$ for various $q \sim 2^l$ and diameter $\sim 2^{10l}$; see Remark \ref{uncertainty}.)  In fact, at a given (large) scale $N$ one can restrict to even narrower major arcs, of width $O(2^{dl}/N^d)$ say.  A finer analysis of the symbol \eqref{symbol-def} reveals for a major arc frequency $\xi = \frac{a}{q} + \theta \mod 1$, that
\[
\varphi_{N,\Z}\left( \frac{a}{q} + \theta \mod 1 \right)
\]
has an approximate factorization
\begin{equation}\label{approx-factor}
\varphi_{\hat \Z}\left(\frac{a}{q} \mod 1 \right) \varphi_{N,\R}(\theta),
\end{equation}
 where the ``arithmetic symbol'' $\varphi_{\hat \Z} \colon \Q/\Z \to \C$ is defined by
\begin{equation}\label{arith-symbol}
\varphi_{\hat \Z}\left(\frac{a}{q} \mod 1 \right) \coloneqq \E_{n \in \Z/q\Z} e\left( \frac{a P(n)}{q} \right)
\end{equation}
and the ``continuous symbol'' $\varphi_{N,\R} \colon \R \to \C$ is defined by
$$ \varphi_{N,\R}(\theta) \coloneqq \frac{1}{N} \int_0^N e(\theta P(t))\ dt.$$
The influence of the arithmetic symbol $\varphi_{\hat \Z}$ (which does not depend on $N$) can be easily factored out in the $p=2$ case by Plancherel's theorem, and the task then readily reduces to that of establishing a multifrequency maximal inequality (see \cite[Lemma 4.1]{B3}).  This result in turn is ultimately derived from a variational inequality for averages of vector-valued $L^2$ functions (see \cite[Lemma 3.30]{B3}), in the spirit of L\'epingle's inequality.

\subsection{Bilinear averages}

Now we turn to multilinear averages. For the norm convergence problem in the case of finite measure and Banach exponents the situation is well understood, thanks to the following result of Host--Kra and Leibman:

\begin{theorem}[Multilinear mean ergodic theorem]\label{walsh-thm}  Let $(X,\mu,T)$ be a measure-preserving system of finite measure, let $P_1(\nn),\dots,P_k(\nn) \in \Z[\nn]$, and let $f_i \in L^{p_i}(X)$ for all $i=1,\dots,k$ and some exponents $1 \leq p_i \leq \infty$ with $\frac{1}{p_1}+\dots+\frac{1}{p_k} \leq 1$.  Then the averages $A^{P_1,\dots,P_k}_N(f_1,\dots,f_k)$ converge in $L^p(X)$ norm for any $0 < p < \infty$ with $\frac{1}{p_1}+\dots+\frac{1}{p_k} < \frac{1}{p}$.
\end{theorem}

\begin{proof}  The case $p_1=\dots=p_k=\infty$, $p=2$ is established in \cite{HK2}, \cite{Leibman} (see also \cite{W} and \cite{A1} for quite different proofs and  generalizations); one can then extend to other $0 < p < \infty$ by H\"older's inequality, and the case of general $p_1,\dots,p_k$ then follows by a standard limiting argument using \eqref{anf}.
\end{proof}

There is a long history of prior partial results (e.g., \cite{HK}, \cite{Z1}, \cite{Ber0}, \cite{BL1}, \cite{FraKra}, \cite{FurWei}) towards Theorem \ref{walsh-thm}, as well as generalizations to actions of other nilpotent groups than $\Z$ (i.e., averages involving multiple measure-preserving transformations $T_1,\dots,T_k$ that generate a nilpotent group); we refer the reader to \cite{Ber1}, \cite{Ber2}, \cite{Fra} for surveys.  In several cases it is possible to ``break duality'' by permitting $\frac{1}{p_1}+\dots+\frac{1}{p_k}$ to exceed $1$; see Section \ref{pless} below.

For pointwise convergence and for two linear polynomials, one also has the following results:

\begin{theorem}[Two linear polynomials]\label{twolin}  Let $(X,\mu,T)$ be a measure-preserving system with finite measure, let $P_1(\nn), P_2(\nn) \in \Z[\nn]$ have degree $1$ with distinct leading coefficients, and let $1 < p_1,p_2 \leq \infty$ be such that $\frac{1}{p_1} + \frac{1}{p_2} < \frac{3}{2}$.  Then for $f \in L^{p_1}(X), g \in L^{p_2}(X)$, the averages $A^{P_1,P_2}_N(f,g)$ converge pointwise almost everywhere.
\end{theorem}

\begin{proof}  For the case $p_1=p_2=\infty$ see Bourgain \cite{B4}; an alternate proof was also given by Demeter \cite{Demeter}.  To extend to the remaining cases of $p_1,p_2$ one applies a bilinear maximal inequality of Lacey \cite{LAC} and a standard limiting argument. 
\end{proof}

We now at last come to the main result of our paper, which concerns an opposing case to Theorem \ref{twolin} in which one has one linear polynomial and one strictly nonlinear polynomial.

\begin{theorem}[Main theorem]\label{main}  Let $(X,\mu,T)$ be a measure-preserving system, let $P(\nn) \in \Z[\nn]$ have degree $d \geq 2$, and let $f \in L^{p_1}(X), g \in L^{p_2}(X)$ for some $1 < p_1,p_2 < \infty$ with $\frac{1}{p_1} + \frac{1}{p_2} = \frac{1}{p} \leq 1$.
\begin{itemize}
    \item[(i)] (Mean ergodic theorem)  The averages $A^{\nn,P(\nn)}_N(f,g)$ converge in $L^p(X)$ norm.
    \item[(ii)]  (Pointwise ergodic theorem) The averages $A^{\nn,P(\nn)}_N(f,g)$ converge pointwise almost everywhere.
    \item[(iii)] (Maximal ergodic theorem)  One has
    $$\| (A^{\nn,P(\nn)}_N (f,g))_{N \in \Z_+} \|_{L^p(X; \ell^\infty)} \lesssim_{p_1,p_2,P} \|f\|_{L^{p_1}(X)} \|g\|_{L^{p_2}(X)}.$$ 
    \item[(iv)]  (Long variational ergodic theorem) If $r > 2$ and $\lambda > 1$, one has
    \begin{equation}\label{var-poly-main}
    \| (A^{\nn,P(\nn)}_N(f,g))_{N \in \D} \|_{L^p(X; \V^r)} \lesssim_{p_1,p_2,r,P,\lambda} \|f\|_{L^{p_1}(X)} \|g\|_{L^{p_2}(X)}
    \end{equation}
    whenever $\D \subset [1,+\infty)$ is $\lambda$-lacunary.
\end{itemize}
\end{theorem}

We now give some remarks about this theorem.

\begin{enumerate}
    \item Theorem \ref{main}(i) already follows from Theorem \ref{walsh-thm} when $(X,\mu)$ has finite measure; in fact for this particular average the results are essentially already contained in \cite{FurWei}.  However it appears to be new in the $\sigma$-finite setting, and the proof method is completely different from methods used to establish Theorem \ref{walsh-thm}.  
    \item Theorem \ref{main}(ii) is completely new for general measure-preserving systems\footnote{This result (and also part (iii)) was claimed in \cite{EA}.  However, there appear to be several gaps in the arguments.  Firstly, in \cite[pp. 23]{EA} it is claimed without giving details that the Cald\'eron transference principle can be applied for the super-H\"older exponent triplet $\ell^2 \times \ell^2 \to \ell^2$, but if one carefully works through the arguments provided in \cite[pp. 10--11]{EA} for these exponents, one loses a factor of $N^{1/2}$ in the estimates (as $h$ now needs to be controlled in $\ell^2$ norm rather than $\ell^\infty$ norm) and thus cannot pass to the limit $N \to \infty$. Secondly, in \cite[pp. 26]{EA}, bilinear maximal estimates are obtained for the super-H\"older exponent triplets $\ell^1 \times \ell^r \to \ell^r$ and $\ell^r \times \ell^1 \to \ell^r$, but the assertion in that paper that bilinear interpolation then gives H\"older exponent estimates such as $\ell^1 \times \ell^\infty \to \ell^1$ or $\ell^r \times \ell^{r'} \to \ell^1$ is false.}, even when $f,g \in L^\infty(X)$ and $X$ has finite measure. In particular, Theorem \ref{main}(ii) when specialized to the case $P(\nn)=\nn^2$ answers the second part of \cite[Problem 11]{Fra} for the Furstenberg--Weiss averages \cite{FurWei} (see also \cite{Fur1}), which is a bilinear variant of the problem considered by Bergelson \cite[Question 9, pp. 52]{Ber1}; see also \cite[\S 6, pp. 838]{Ber2}.
Theorem \ref{main} is also a contribution towards establishing the Furstenberg--Bergelson--Leibman conjecture {\cite[Section 5.5, p. 468]{BeLe}}, which asserts the following. Given integers $d, k, m, N\in\Z_+$, let $T_1,\ldots, T_d:X\to X$ be a family  of  invertible measure-preserving transformations of a probability measure space    $(X, \mu)$ that generates a nilpotent group of step $m$. Assume that
 $P_{1, 1},\ldots,P_{i, j},\ldots, P_{d, k}\in \Z[\mathrm n]$. Then for any  $f_1, \ldots, f_k\in L^{\infty}(X)$, the non-conventional multiple
polynomial averages
\begin{align}
\E_{n\in[N]}\prod_{j=1}^kf_j(T_1^{P_{1, j}(n)}\cdots T_d^{P_{d, j}(n)} x)
\end{align}
converge pointwise for $\mu$-almost every $x\in X$ as $N\to\infty$. 
This conjecture is a widely open problem in ergodic theory that was promoted in person by Furstenberg, see \cite[p. 6662]{A1} and \cite{Kra}, before being published in \cite{BeLe}. Bergelson--Leibman \cite{BeLe}  showed that convergence may fail if the transformations $T_1,\ldots, T_d$ generate a solvable group. Our main theorem  solves this conjecture  in the case $d=1$, $k=2$ with $P_{1, 1}(\nn)=\nn$ and $P_{1, 2}(\nn)=P(\nn)\in\Z[\nn]$ with $\deg P\ge 2$.
Pointwise convergence for non-conventional polynomial averages has previously been established for some special measure-preserving systems, such as exact endomorphisms and $K$-automorphisms \cite{DL} and nilsystems \cite{Leibman2}. 
    \item Our methods of proofs break down in the linear case $d=1$ (as the minor arc contributions are no longer negligible), and so we are unable to give an alternate proof of Theorem \ref{twolin}.
    \item For $p>1$ (i.e., above the line of duality), Theorem \ref{main}(iii) follows easily from past results. Indeed, from several applications of H\"older's inequality one has
    \begin{align*}
        \| (A^{\nn,P(\nn)}_N &(f,g))_{N \in \Z_+} \|_{L^p(X;\ell^\infty)}\\
&\leq \| (A^{\nn}_N (|f|^{p_0})|)_{N \in \Z_+} \|_{L^{p_1/p_0}(X;\ell^\infty)}^{1/p_0}
\| (A^{P(\nn)}_N (|g|^{p'_0})|)_{N \in \Z_+} \|_{L^{p_2/p'_0}(X;\ell^\infty)}^{1/p'_0}
    \end{align*}
for any $1 < p_0 < \infty$.  In the $p>1$ case one can select $p_0$ so that $p_0 < p_1$ and $p'_0 < p_2$, and the claim now follows from Theorem \ref{lin-poly}(iii).  However, the $p=1$ case (i.e., on the line of duality) is new, even when $p_1=p_2=2$.  Also, the simple argument given above does not seem to easily adapt to give the $p>1$ cases of the other components (i), (ii), (iv) of the theorem, although it does permit one to reduce those cases of (i), (ii) to the case in which $f,g \in L^\infty(X)$ by the usual limiting argument.  A continuous analogue of Theorem \ref{main}(iii) was previously established in \cite{LiXi} (see also \cite{Li}, \cite{GL}).
    \item Theorem \ref{main}(iv) is the key result in the theorem, and easily implies the other parts of the theorem, as we shall show in Section \ref{transfer-sec}.  The condition $r>2$ is necessary, as no variational estimate is possible for $r \leq 2$; see Corollary \ref{r-counter}. The situation can be contrasted with that in \cite{DMT}, in which a certain bilinear paraproduct was shown to enjoy $r$-variation estimates for some values of $r < 2$.
    \item A modification of our arguments (taking particular advantage of linear $L^p$ improving estimates) is able to ``break duality'' and establish some cases of Theorem \ref{main} in the non-Banach regime $p < 1$, with the range of exponents being particularly strong in the case of norm convergence on spaces of finite measure.  See Section \ref{pless}.  A similar ``breaking duality'' phenomenon occurred in \cite{DLTT}; also, in \cite[Theorem 2]{LiXi} a continuous analogue of part (iii) of the theorem was established that ``broke duality'' by allowing $p$ to lie in the range $p>\frac{d-1}{d}$, which is best possible up to the endpoint; see \cite[\S 3]{LiXi}. 
    \item The requirement that $X$ be $\sigma$-finite can be dropped by observing that $f \in L^{p_1}(X), g \in L^{p_2}(X)$ have $\sigma$-finite supports (since $p_1,p_2 < \infty$), and hence the invariant set $\bigcup_{n \in \Z} T^n(\mathrm{supp}(f) \cup \mathrm{supp}(g))$ is also $\sigma$-finite.  Since the averages $A^{\nn,P(\nn)}_N(f,g)$ are all supported on this invariant $\sigma$-finite set, one can restrict to the $\sigma$-finite case without loss of generality.
\end{enumerate}

\subsection{Overview of proof}

We now give an overview of the proof of Theorem \ref{main}.  The arguments follow the basic framework of the arguments used to establish the linear results in Theorem \ref{lin-poly}, but with several new difficulties arising that require substantial new ideas to overcome.  Most notably:

\begin{itemize}
\item[(a)] Plancherel's theorem and Weyl sum estimates (see \cite[Lemma 20.3, p. 462]{IK}) are no longer sufficient by themselves to control the contribution of the minor arcs, thus defeating a ``naive'' implementation of the circle method.
\item[(b)] The bilinear analogue
\begin{equation}\label{symbol-bil-def}
m_{\hat \Z}\left( \frac{a_1}{q} \mod 1, \frac{a_2}{q} \mod 1 \right) \coloneqq \E_{n \in \Z/q\Z} e\left( \frac{a_1n + a_2P(n)}{q}\right)
\end{equation}
of the arithmetic symbol $\varphi_{\hat \Z}$ defined in \eqref{arith-symbol} cannot be factorized as a tensor product of a function of $\frac{a_1}{q} \mod 1$ and a function of $\frac{a_2}{q} \mod 1$.  As a consequence,  symbol \eqref{symbol-bil-def}, despite being independent of $N$; cannot be disposed of purely by linear tools such as Plancherel's theorem (even in the model case $p_1=p_2=2$) due to a bilinear nature of the problem and must be treated in tandem with the continuous features.
\end{itemize}

Our resolution to these problems can be summarized by the following two sentences:

\begin{itemize}
\item[(i)]  Additive combinatorics (and more specifically, Peluse--Prendiville theory), as well as Hahn--Banach separation theorem, Ionescu--Wainger multiplier theory and the $\ell^p(\Z)$ improving theory of Han--Kova\v{c}--Lacey--Madrid--Yang on the integers $\Z$, are used to control the contribution of minor arcs. This is a bilinear theory of the minor arcs.
\item[(ii)]  Adelic harmonic analysis (which combines the continuous harmonic analysis of the reals $\R$ with the arithmetic harmonic analysis of the profinite integers $\hat \Z$), as well as Ionescu--Wainger multiplier theory, two-parameter Rademacher--Menschov argument, shifted square function estimates and the $L^p(\hat \Z)$ improving theory on the profinite integers $\hat \Z$, are used to control the contribution of major arcs. 
\end{itemize}

We now discuss the strategy in more detail.

\subsubsection{Standard reductions}
Following Bourgain \cite{B1}, it suffices to establish the variational estimate \eqref{Var-est} on a finite $\lambda$-dyadic set $\D$ of scales, and by using the Calder\'on transference principle we can work with the integer shift system $\Z$.  For technical reasons it is also convenient to remove the lower half $n \leq N/2$ of the averaging operator \eqref{anx} and only retain the upper half $n > N/2$, but we ignore this step for sake of discussion.  These standard reductions are reviewed in Section \ref{transfer-sec}.  We will need to establish the variational estimate for all choices of $(p_1,p_2)$, but the most important case is when $p_1=p_2=2$ (and hence $p=1$), where it is easiest to establish a certain exponential decay that can then be propagated to all other choices of exponents $(p_1,p_2)$ by interpolation.  For sake of discussion we therefore restrict attention to the $p_1=p_2=2$ case.

\subsubsection{Minor arcs estimates}
Again following Bourgain, we would now like to restrict the functions $f,g$ to major arcs in Fourier space.  In the linear setting this could be accomplished relatively easily using Plancherel's theorem and decay estimates \eqref{eq:13} for the symbol \eqref{symbol-def} on minor arcs.  However, in the bilinear setting Plancherel's theorem and the classical Weyl estimate \cite[Lemma 20.3, p. 462]{IK} are insufficient to obtain satisfactory control on the minor arc contribution.  Instead we use a deep recent inverse theorem of Peluse and Prendiville \cite{PP1} and Peluse \cite{P2} from the additive combinatorics literature, see Theorem \ref{pp}, which asserts that for every $0<\delta\le 1$ and bounded functions  $f,g:[-O(N^d), O(N^d)]\to\C$ with $\|f\|_{\ell^\infty}, \|g\|_{\ell^\infty}\le 1$ if $\|A_N(f, g)\|_{\ell^1}\ge \delta N^d$, then $f$ must weakly correlate with the indicator function of a progression ${P}=\{qm\in\Z_+: m\in[N']\}$ with $q\lesssim \delta^{-O(1)}$ and $\delta^{O(1)}N\lesssim N'\le N$, in the sense that $\|f*\ind{-P}\|_{\ell^1}\gtrsim \delta^{O(1)}N'N^d$ provided that $N\gtrsim \delta^{-O(1)}$. In other words, it says that  the function $f$ has a major arc structure at scale $N$, which is precisely stated  using Fourier-transform language in Proposition \ref{lip}. However, for our application we need to replace the $\ell^\infty$ control with (suitably normalized) $\ell^2$ control.  To do this we shall use the Hahn--Banach theorem to interpret this inverse theorem as a structural description of certain \emph{dual functions} associated to the averaging operator $A_N$, see Corollary \ref{struct}. We also need to utilize the multiplier theory of Ionescu and Wainger \cite{IW} to maintain the separation of major and minor arcs during this process, see Proposition \ref{propdual}.
Then combine the latter with recent linear $L^p$-improving estimates on $\Z$ by Han--Kova{\v c}--Lacey--Madrid--Yang \cite{HKLMY} (see also Dasu--Demeter--Langowski \cite{DDL}) to relax the hypotheses to $\ell^2$, see Corollary \ref{struct-iii}.   The final conclusion of this analysis is the single-scale minor arc estimate in Theorem \ref{improv}, which roughly speaking (with the notation from \eqref{log-scale}) asserts that
\begin{align}
\label{eq:15}
\|A_N(f, g)\|_{\ell^1}\lesssim (2^{-O(l)} + \langle \Log N \rangle^{-O(1)}) \|f\|_{\ell^2}\|g\|_{\ell^2}
\end{align}
unless the Fourier transform of $f$ and $g$ are supported on major arcs of width respectively $O(2^lN^{-1})$ and $O(2^{dl}N^{-d})$, (the disparity is due to the different degrees in the polynomials $\nn, P(\nn)$). Inequality \eqref{eq:15} can be thought of as a bilinear variant of inequality \eqref{eq:13}, which was derived from
classical Weyl's inequality \cite[Lemma 20.3, p. 462]{IK}. This bilinear inequality \eqref{eq:15} is a very useful result that we will apply repeatedly in our arguments.

\subsubsection{Major arcs estimates: a first glimpse}
One can now restrict attention to major arcs, in which $f$ has Fourier support supported at combinations $\alpha + \theta \mod 1$ of ``arithmetic frequencies'' $\alpha \in \Q/\Z$ and ``continuous frequencies'' $\theta \in \R$.  The ``height'' of the arithmetic frequency $\alpha$ will be bounded by some threshold $2^{l_1}$, and the magnitude $|\theta|$ of the continuous frequency will similarly be bounded by some threshold $2^{k_1}$ for some large negative $k_1$.  With some additional effort, $g$ can similarly be restricted to major arc frequencies that are the combination of an arithmetic frequency of height at most $2^{l_2}$ and a continuous frequency of magnitude at most $2^{k_2}$.  Naively, the height of an arithmetic frequency $\alpha = \frac{a}{q} \mod 1$ with $(a,q)=1$ might be defined to equal $q$ (or $\inf \{ 2^l: q \leq 2^l\}$, if one wishes to view height as a dyadic integer).  However for technical reasons it is often more convenient to replace this naive notion of height with a more complicated variant of height implicitly introduced by Ionescu and Wainger \cite{IW} that enjoys better multiplier theory (the losses incurred here are only polynomial in $l$ rather than exponential); see Appendix \ref{iw-app}.  In order to decouple the continuous aspects of the analysis from the arithmetic aspects, it turns out to be convenient to embed the integers $\Z$ into the adelic integers\footnote{One could also work with various projections $\R \times \Z/Q\Z$ of the adelic integers, which amounts to requiring a common denominator $Q$ to the arithmetic frequencies being used; but the adelic formalism is cleaner in that it automatically handles uniformity in the $Q$ parameter.  Also we believe it lends some conceptual clarity to the strategy of separating the continuous and arithmetic aspects of the analysis.} $\A_\Z \coloneqq \R \times \hat \Z = \R \times \prod_p \Z_p$; this embedding $\iota \colon \Z \to \A_\Z$ is the Fourier adjoint of the addition map $\pi \colon \R \times \Q/\Z \to \TT$ defined by $\pi(\theta,\alpha) \coloneqq \alpha+\theta$ that was implicitly used to define major arcs.  The advantage of working in the adelic framework is that several key linear and bilinear Fourier symbols on the integers, when transferred to the adelic integers, can be treated in a fairly unified way and can be  cleanly decomposed or approximated into simpler symbols that exhibit a useful tensor product structure, so that the continuous and arithmetic aspects of the symbols involved become almost completely decoupled; see also Figures \ref{fig:phys}, \ref{fig:freq}.

\subsubsection{Major arcs estimates: paraproduct-type decomposition}
The objective is now to obtain, for a given choice of height scales $l_1,l_2$, variational bounds on the average $A_N(f,g)$ under the assumption that $f,g$ have Fourier supports associated to major arcs of heights $2^{l_1}, 2^{l_2}$ respectively, with the bounds enjoying exponential decay in the parameter $l \coloneqq \max(l_1,l_2)$.  At a given scale $N$, one can use the Ionescu--Wainger multiplier theory to restrict the Fourier transform of $f$ to major arcs of width about $2^{l} N^{-1}$, and similarly restrict the Fourier transform of $g$ to major arcs of width about $2^{dl} N^{-d}$ (as before the disparity is due to the different degrees in the polynomials $\nn, P(\nn)$).  For any given scale $N$, Theorem \ref{improv} gives the desired exponential gain in $l$; the problem is how to sum in $N$. To overcome this difficulty we perform a certain paraproduct decomposition \eqref{FN-def}, \eqref{GN-def} centered around a finite number of (arithmetic) frequencies. This  contrasts sharply with the classical theory of paraproducts that are centered at the frequency origin.  Here, again an indispensable role is played by the Ionescu--Wainger projections \eqref{iw} and \eqref{iw-2}, which will allow us to control ``low-low'', ``low-high'', ``high-low'' and ``high-high'' paraproducts by employing the methods from continuous harmonic analysis. 

\subsubsection{Major arcs estimates: ``low-low'' case and ``small scales''}
For sake of exposition let us initially focus on the ``low-low'' case when one can restrict the width of the major arcs further to $2^{-u} N^{-1}$ and $2^{-du} N^{-d}$ where $u$ is moderately large (about $2^{\rho l}$ for some small constant $\rho$).  The argument then splits into the treatment of ``small scales'' $2^u < N < 2^{2^u}$ and ``large scales'' $N > 2^{2^u}$. (The contribution of extremely small scales $N \leq 2^{u}$ can be easily discarded, thanks to the exponential decay factors present in the single scale estimates). For small scales,  in the linear theory we used the Rademacher--Menshov type inequality \cite{MST}, which was quite efficient. Here, due to the bilinear nature of the problem the situation is much more complicated. We begin with performing some Fourier-analytic approximations at the adelic integer level, analogous to \eqref{approx-factor}, to replace averages such as $A_N(f,g)$ with an expressions of the form $B(f_N,g_N)$, where the bilinear operator $B$ is now independent of $N$. This is the key idea of the major arcs analysis, which  is encapsulated in the model estimate \eqref{other-small} of  Theorem \ref{model-est}. The same idea is also exploited in the ``large scales'' to establish estimate \eqref{all-all-large} of  Theorem \ref{model-est}.   After these approximations, we use a two-parameter Rademacher--Menshov argument and Khinchine's inequality to reduce the variational estimates to a single scale estimates; such arguments lose factors that are essentially logarithmic in the number of scales, which in the small scale regime gives a loss of $u^{O(1)}$, but this is acceptable thanks to the exponential gains in $l$, which again can be  derived from \eqref{eq:15}.

\subsubsection{Major arcs estimates: ``low-low'' case and ``large scales''} At large scales, the major arcs become extremely narrow, so much so that the arithmetic frequencies at the center of these arcs can be given a common denominator $Q$ with $\frac{1}{Q}$ much larger than the width of these arcs.  In this regime it becomes possible to use a quantitative version of the Shannon sampling theorem (Theorem \ref{Sampling}) to transfer from the integers $\Z$ to the adelic integers $\A_\Z = \R \times \hat \Z$ while essentially preserving all function space norms of interest. The behaviour in the continuous variable $\R$ is relatively tractable due to the Ionescu--Wainger multiplier theory and \cite{MST}. The main difficulty is to understand the nature of the associated ``arithmetic'' average $A_{\hat \Z}$ on the profinite integers $\hat \Z$, which is a compact commutative ring.  By some use of $p$-adic methods (see Appendix \ref{sec:app1}), we will obtain a non-trivial $L^p$-improving estimate for this average, while from yet another invocation of Theorem \ref{improv} we will also obtain exponential decay in $l$ for these averages (for the $L^2$ theory at least, and the remaining cases can then be treated by interpolation).  By combining these estimates with some general manipulation of variational norms, and also relying primarily on a vector-valued version of L\'epingle's inequality from \cite{MSZ1} to handle the variational behavior in the continuous variable $\R$, we can obtain acceptable control on the contribution of the large scales.

\subsubsection{Major arcs estimates: remaining cases}

The other cases (``high-high'', ``low-high'', ``high-low'') can be treated by modifications of the method; the main new difficulties are to obtain some additional decay when one is relatively far from the arithmetic frequencies at the center of the major arcs (that is to say, when the continuous component of the frequency is large).  By interpolation one only needs to obtain this decay for the $\ell^2$ theory.  In the ``high-high'' case one can obtain such a decay using Theorem \ref{improv} once again, exploiting almost orthogonality in order to sum over scales $N$.  In the remaining ``low-high'' and ``high-low'' cases we will obtain the required decay by applying an elementary integration by parts to a certain bilinear symbol associated to the averaging operation $A_N$ (see Lemma \ref{integ-ident}).  On the  other hand, this decay is at risk of being overwhelmed by the increased oscillations present in the symbol.  To avoid this we use shifted Calder\'on--Zygmund theory (see Appendix \ref{shift-app}), of the type used for instance in \cite{Lie2}, that allows one to handle certain types of oscillating Fourier multipliers losing only acceptable logarithmic factors in the estimates. The idea of shifted maximal estimates was also recently exploited in \cite{IMMS} in the context of establishing of pointwise ergodic theorems for the polynomial averages on nilpotent groups; and it seems to be decisive in problems when the operators in question cannot be interpreted as convolution operators corresponding to an abelian convolution. 

\subsubsection{Final remarks}
Finally, we emphasize that the proof of Theorem \ref{main} can also be adapted (and simplified) to give an alternate proof of Theorem \ref{lin-poly} (but in which one only controls the long variation rather than the full variation).  We sketch the changes needed to the argument as follows.  The exponent $p_1$ is now fixed to equal $\infty$ (so that $p=p_2$), and the first function $f$ is fixed to equal $1$ (which allows for several simplifications, for instance the parameter $l_1$ can be taken to be $0$, and $s_1$ can be taken to be $-u$).  All appearances of $\ind{p_1=p_2=2}$ are now  replaced by $\ind{p_1=\infty,p_2=2}$.  Various linear estimates, such as Ionescu--Wainger multiplier estimates, shifted Calder\'on--Zygmund estimates, and Lepingle's inequality, do not hold in general at the $\ell^\infty$ endpoint, but are trivially true when applied to the specific function $f=1$ in $\ell^\infty$, so this does not cause difficulty.  Theorem \ref{improv} needs to be modified to an $\ell^\infty \times \ell^2 \to \ell^2$ estimate with $f = 1$, but in this case the required gain of $2^{-cl} + \langle \Log N \rangle^{-cC_1}$ is immediate from Plancherel's theorem and Weyl sum estimates \cite[Lemma 20.3, p. 462]{IK}, thus avoiding the need to invoke the Peluse--Prendiville theory.

\subsection{Open questions}
While our main interest is in averaging operators on the integers $\Z$, in the course of our arguments it became natural to also consider the analogous averaging operators on other locally compact abelian domains such as $\R, \Z/Q\Z, \R \times \Z/Q\Z, \Z/p^j\Z, \Z_p, \hat \Z$, and $\A_\Z$, with the adelic integers $\A_\Z$ playing a particularly central role, at least on a conceptual level; see Figure \ref{fig:phys}.  The connection can be summarized by the slogan
$$ \hbox{Major arc analysis on } \Z \approx \hbox{Low frequency analysis on } \A_\Z,$$
where ``low frequency'' has to be interpreted in both a continuous and arithmetic sense; see Figure \ref{fig:major-l2}.  In particular, the adelic averaging operators $A_{N,\A_\Z}$ defined in \eqref{adelic-average} emerge as a simplified model for the integer averaging operators $A_{N,\Z}$, and further investigations into similar problems in discrete harmonic analysis may wish to begin by first understanding adelic models of such problems, particularly in ``true complexity zero'' situations in which one suspects that the major arc contributions are dominant or equivalently that the minor arc contribution is negligible. In fact, the method of proof of Theorem \ref{main} relies in an essential way on the negligibility of the minor arc contribution; in the language of additive combinatorics, this reflects the fact that the pattern $(x,x-n,x-P(n))$ has ``true complexity zero'' in the sense of Gowers and Wolf \cite{GW}.  In the language of ergodic theory, the corresponding assertion is that the minimal characteristic factor of the averages $A^{\nn,P(\nn)}_N$ is the rational Kronecker (or profinite) factor ${\mathcal K}_{\mathrm{rat}}$ generated by the periodic functions.

\medskip

We close our introduction with some questions relating to Theorem \ref{main} that remain open.

\begin{enumerate}
\item  Does Theorem \ref{main} continue to hold if one of $p_1,p_2$ is allowed to be infinite?  Certainly from Theorem \ref{lin-poly} the maximal inequality (Theorem \ref{main}(iii)) will still hold if one or both of $p_1,p_2$ are infinite, but the situation for the other parts of the theorem are less clear (except in the special case where $p_1=\infty$ and $f$ is constant, or $p_2=\infty$ and $g$ is constant).  Given the ability to break duality, the endpoints $p_1=1$, $p_2=1$ could also be investigated.
\item  Is the analogue of Theorem \ref{main}(iv) true for the full variation, in which the lacunarity hypothesis on $\D$ is omitted?  Equivalently, can the implied constant in \eqref{var-poly-main} be made uniform in $\lambda$?  The problem is likely to be significantly simpler if the sharp truncation $\ind{n \leq N}$ implicit in the definition of the averages $A^{\nn,P(\nn)}_N$ is replaced by a smoother weight.  Note that the linear analogue of this question was already resolved in Theorem \ref{lin-poly}(iv).
\item  To what extent can the results in Theorem \ref{main} extend to other bilinear averages $A^{P_1(\nn), P_2(\nn)}_N$, or more ambitiously to multilinear averages
$A^{P_1(\nn), \dots, P_k(\nn)}_N$? We refer to Bergelson's surveys  \cite[Question 9, pp. 52]{Ber1}, \cite[\S 6, pp. 838]{Ber2}.  It is not difficult to adapt Theorem \ref{main} to cover averages $A^{P_1(\nn), P_2(\nn)}_N$ in which one of the $P_1,P_2$ is linear (i.e., of degree $1$) and the other is non-linear, however when both $P_1,P_2$ are non-linear a refinement of the Peluse--Prendiville theory may be required.  We hope to investigate these averages in future work.
\item    Is there some analogue of these methods that can cover patterns of higher complexity?  A natural first step would be to recover some portion of Theorem \ref{twolin} (which has ``true complexity one'' in the Gowers--Wolf \cite{GW} sense) by these methods.
\item What are explicit ranges of exponents $p_1,p_2$ for which one can ``break duality'' with in Theorem \ref{main}? In the model case $p_1=p_2$ (so that $p=p_1/2=p_2/2$), Lemma \ref{single-below} suggests that one should be able to take $p$ in the range $p > 1 - \frac{1}{d^2+d-1}$, or even $p > 1 - \frac{1}{2d}$ in the $d=2$ case, with the latter range also expected if \cite[Conjecture 1.5]{HKLMY} holds. It should also be possible to recover the optimal range $r>2$ of the variational exponent $r$ below the line of duality (our current arguments incur a loss in this parameter that depends on how close $(1/p_1,1/p_2)$ is to $(1/2,1/2)$).
\item Theorem \ref{main}(iv) gives variational estimates in $\V^r$ norms for $r > 2$, and in Section \ref{quad-var} the $r=2$ endpoint is shown to be false.  However, there still remains the question of whether a jump inequality (analogous to Doob's inequality for martingales) is true at the $r=2$ endpoint.  Such endpoint jump inequalities were established in \cite{MSZ3} for linear polynomial averages on $\Z^k$.
\item Theorem \ref{main} was focused on unweighted averages $$A_N^{\nn,P(\nn)}(f,g)(x) = \E_{n \in [N]} T^n f(x) T^{P(n)} g(x),$$ but one can pose similar questions\footnote{One could also consider fractional integral type expressions $\E_{n \in [N]} (n/N)^{-\alpha} T^n f(x) T^{P(n)} g(x)$ for $0 < \alpha < 1$, but these can be easily expressed as linear combinations of the unweighted averages via summation by parts and so would be expected to obey nearly identical estimates to those averages.} for the truncated singular integral analogue $\sum_{0 < |n| < N} \frac{1}{n} T^n f(x) T^{P(n)} g(x)$; currently only single-scale super-H\"older estimates are known \cite{Dong}.  In the linear setting (resp. the bilinear setting for two linear polynomials), the theory for the averages and the truncated singular integrals are similar; see \cite{MSZ3} (resp. \cite{LAC}).  Bounds on the (untruncated) bilinear continuous singular integrals were obtained in \cite{Li}, \cite{Lie}, \cite{LiXi}, \cite{Lie2}. 
\item To what extent do the implied constants in Theorem \ref{main} depend on the coefficients of $P$?  The estimate in \cite[Theorem 1.6]{HKLMY} suggests that the dependence of constants is at worst polynomial; on the other hand, \cite[Corollary 1.15]{MSZ3} suggests that one may be able to obtain bounds uniform in the coefficients, by lifting the problem to $\Z^d$ and establishing an analogue of Theorem \ref{main} in that setting.  (However, this latter strategy would require a multidimensional version of the theory of Peluse and Prendiville, which may be highly nontrivial.) We also hope to investigate the latter multidimensional strategy in future work.
\item Can the results here on the (rational) integers $\Z$ be extended to rings of integers in more general number fields, such as the ring $\Z[i]$ of Gaussian integers?  Certainly the adelic formalism is exceptionally well adapted to this setting \cite{Tate}, but other components of the argument may require significantly more effort to generalize appropriately.
\item Assuming that $P\in\R[\nn]$, it also makes sense to ask whether Theorem \ref{main} holds with the averages $A_N^{\nn, \lfloor P(\nn)\rfloor}(f, g)$ in place of $A_N^{\nn, P(\nn)}(f, g)$. This kind of question for linear polynomial averages was considered by Bourgain in \cite{B3}. One could also replace the polynomial $P$ with elements of other Hardy fields, in the spirit of \cite{BKQRW, BW}, or by random functions of polynomial growth, in the spirit of \cite{FLW}.  In fact these variants may be simpler than the polynomial case, as the only major arc that is expected to be significant is the one centered at the origin.
\item As mentioned previously, there is a well-developed theory of characteristic factors for the limiting values of non-conventional polynomial averages $A_N(f_1,\dots,f_k)$ when the functions $f_1,\dots,f_k$ lie in $L^\infty(X)$ and $X$ has finite measure; see \cite{Ber1}, \cite{Ber2}, \cite{Fra}.  To what extent does this theory extend to other $L^p$ spaces and to the case when $X$ is merely $\sigma$-finite, for instance for the average $A_N^{\nn,P(\nn)}(f,g)$ studied in Theorem \ref{main}?
\end{enumerate}

\subsection*{Acknowledgments}
We thank Sarah Peluse and Sean Prendiville for several helpful discussions about their inverse theory and for sharing some unpublished notes. We also thank Jim Wright for sharing his unpublished notes on $L^q$-improving estimates in $p$-adic fields $\QQ_p$, and for helpful comments and corrections. We thank Vitaly Bergelson and Bryna Kra for the discussion about the history of Problem 11 from Frantzikinakis' open problems survey \cite{Fra}. We also thank Jaume de Dios and Dariusz Kosz and Wojciech S{\l}omian for further corrections.   Finally, we thank the referees for careful reading of the manuscript and useful remarks that led to the improvement of the presentation.

\section{Notation}\label{notation-sec}

In this section we set out some basic notation used throughout the paper.

\subsection{Elementary number theory}

We use $\Z_+ \coloneqq \{1,2,\dots\}$ to denote the positive integers and $\N \coloneqq \{0,1,2,\dots\}$ to denote the natural numbers.  For any $N > 0$, $[N]$ denotes the discrete interval $[N] \coloneqq \{ n \in \Z_+: n \leq N \}$.  The set $\{2,3,5,\dots\}$ of all prime numbers will be denoted by $\PP$. 
If $q_1,q_2 \in \Z_+$, we write $q_1|q_2$ if $q_1$ divides $q_2$. If $a,q \in \Z_+$, we let $(a,q)$ denote the greatest common divisor of $a$ and $q$. We let $[q]^\times \coloneqq \{ a \in [q]: (a,q) = 1 \}$ denote the elements of $[q]$ that are coprime to $q$.

\subsection{Magnitudes and asymptotic notation}

We use the Japanese bracket notation
$$ \langle x \rangle \coloneqq (1 + |x|^2)^{1/2}$$
for any real or complex $x$.  We use $\lfloor x \rfloor$ to denote the greatest integer less than or equal to $x$. All logarithms in this paper will be to base $2$, and for any $N \geq 1$ we define the \emph{logarithmic scale} $\Log N$ of $N$ by the formula
\begin{equation}\label{log-scale}
 \Log N \coloneqq \lfloor \log N \rfloor
 \end{equation}
thus $\Log N$ is the unique natural number such that $2^{\Log N} \leq N < 2^{\Log N+1}$.

For any two quantities $A, B$ we will write
$A \lesssim B$, $B \gtrsim A$, or $A = O(B)$ to denote the bound
$|A| \leq CB$ for some absolute constant $C$.  If we need the implied constant $C$ to depend on additional parameters we will denote this by subscripts, thus for instance $A \lesssim_\rho B$ denotes the bound $|A| \leq C_\rho B$ for some $C_\rho$ depending on $\rho$.  We write $A \sim B$ for $A \lesssim B \lesssim A$.  To abbreviate the notation we will sometimes explicitly permit the implied constant to depend on certain fixed parameters (such as the polynomial $P$) when the issue of uniformity with respect to such parameters is not of relevance.

\subsection{Averages, indicators, and cutoffs}\label{cutoff-sec}

We use the averaging notation
\begin{align}
\label{eq:12}
 \E_{n \in A} f(n) \coloneqq \frac{1}{\# A} \sum_{n \in A} f(n)
\end{align}
for any finite non-empty set $A$, where $\# A$ denotes the cardinality of $A$; in other words, $\E_{n \in A} f(n)$ is the integral of $f$ against normalized counting measure on $A$.  Note in particular that $\E_{n \in [N]} f(n) = \frac{1}{N} \sum_{n=1}^N f(n)$ when $N \in \Z_+$.  We use $\ind{E}$ to denote the indicator function of a set $E$. Similarly, if $S$ is a statement, we use $\ind{S}$ to denote its indicator, equal to $1$ if $S$ is true and $0$ if $S$ is false. Thus for instance $\ind{E}(x) = \ind{x \in E}$.

Throughout this paper we fix a cutoff function $\eta \colon \R \to [0,1]$ that is a smooth even function supported on $[-1,1]$ that equals one on $[-1/2,1/2]$.  All constants are permitted to depend on $\eta$.  For any $k \in \Z$, we let $\eta_{\leq k} \colon \R \to [0,1]$ denote the rescaled version
\begin{equation}\label{eta-resc}
\eta_{\leq k}(\xi) \coloneqq \eta(\xi/2^k)
\end{equation}
of $\eta$.

\subsection{Function spaces}

All vector spaces in this paper will be over the complex numbers $\C$.

If $T \colon V \to W$ is a continuous linear map between normed vector spaces $V, W$, we use $\|T\|_{V \to W}$ to denote its operator norm. If $B \colon V_1 \times V_2 \to W$ is a continuous bilinear map between normed vector spaces $V_1,V_2,W$, we similarly use $\|B\|_{V_1 \times V_2 \to W}$ to denote its operator norm.

If $(X,\mu)$ is a measure space, we let $L^0(X)$ be the space of all $\mu$-measurable
complex-valued functions defined on $X$, with the usual convention of identifying functions that agree $\mu$-almost everywhere. The space of all functions in $L^0(X)$ whose modulus is
integrable with $p$-th power is denoted by $L^p(X)$ for
$p\in(0, \infty)$, whereas $L^{\infty}(X)$ denotes the space of all
essentially bounded functions in $L^0(X)$. If $1 \leq p \leq \infty$ is an exponent, the dual exponent $1 \leq p' \leq \infty$ is defined by the usual relation $1/p + 1/p' = 1$.  When $X$ is endowed with counting measure, we will abbreviate $L^p(X)$ to $\ell^p(X)$ or even $\ell^p$.

We can extend these notions to functions taking values in a finite dimensional normed vector space $V = (V, \|\cdot\|_V)$, for instance $L^0(X;V)$ is the space of measurable functions from $X$ to $V$ (up to almost everywhere equivalence), and
\begin{align}
\label{eq:11}
L^{p}(X;V)
\coloneqq \left\{F\in L^0(X;V):\|F\|_{L^{p}(X;V)} \coloneqq \left\|\|F\|_V\right\|_{L^{p}(X)}<\infty\right\}.
\end{align}
One can extend these notions to infinite-dimensional $V$, at least if $V$ is separable, but we will almost always be able to work in finite-dimensional settings (or can quickly reduce to such a setting by a standard approximation argument).

For any finite dimensional normed vector space $(B,\|\cdot\|_B)$ and any sequence
 $(\mathfrak a_t)_{t\in\I}$ of elements of $B$ indexed by a totally
 ordered set $\I$, and any exponent $1 \leq r < \infty$, the
 $r$-variation seminorm is defined by the formula
 \begin{equation}\label{var-seminorm}
  \| (\mathfrak a_t)_{t \in \I} \|_{V^r(\I; B)} \coloneqq
 \sup_{J\in\Z_+} \sup_{\substack{t_{0} \leq \dotsb \leq t_{J}\\ t_{j}\in\I}}
\Big(\sum_{j=0}^{J-1}  \|\mathfrak a(t_{j+1})-\mathfrak a(t_{j})\|_B^{r} \Big)^{1/r},
 \end{equation}
where the  supremum is taken over all finite increasing sequences in $\mathbb I$, and is set by convention to equal zero if $\I$ is empty.  Taking limits as $r \to \infty$ we also adopt the convention
 \begin{align*}
 \| (\mathfrak a_t)_{t \in \I} \|_{V^\infty(\I;B)} \coloneqq \sup_{t \leq t' \in \I} \|\mathfrak a(t') - \mathfrak a(t)\|_B.
 \end{align*}

The $r$-variation norm for $1 \leq r \leq \infty$ is defined by
\begin{equation}\label{vardef}
  \| (\mathfrak a_t)_{t \in \I} \|_{\V^r(\I; B)} 
\coloneqq \sup_{t\in\I}\|\mathfrak a_t\|_B+
\| (\mathfrak a_t)_{t \in \I} \|_{V^r(\I;B)}.
\end{equation}
This clearly defines a norm on the space of functions from $\I$ to $B$.
If $B=\C$, then we will abbreviate $V^r(\I;X)$ to $V^r(\I)$ or $V^r$, and $\V^r(\I;X)$ to $ \V^r(\I)$ or $\V^r$. If $(X,\mu)$ is a measure space, then  using \eqref{vardef} and  \eqref{eq:11}, one can explicitly write
\[
L^p(X;\V^r)=\left\{F\in L^0(X;\V^r):\|F\|_{L^{p}(X;\V^r)} \coloneqq \left\|\|F\|_{\V^r}\right\|_{L^{p}(X)}<\infty\right\}.
\]

Note that the $\V^r$ norm is non-decreasing in $r$, and comparable to the $\ell^\infty$ norm when $r=\infty$. We also observe the simple triangle inequality
\begin{equation}\label{simple}
\| (\mathfrak a_t)_{t \in \I} \|_{\V^r(\I;X)} \lesssim \| (\mathfrak a_t)_{t \in \I_1} \|_{\V^r(\I_1;X)} + \| (\mathfrak a_t)_{t \in \I_2} \|_{\V^r(\I_2;X)} 
\end{equation}
whenever $\I = \I_1 \uplus \I_2$ is an ordered partition of $\I$, thus $t_1<t_2$ for all $t_1 \in \I_1, t_2 \in \I_2$.  In a similar spirit we have the bound
\begin{equation}\label{varsum}
 \| (\mathfrak a_t)_{t \in \I} \|_{\V^r(\I;X)} \lesssim \| (\mathfrak a_t)_{t \in \I} \|_{\ell^r(\I;X)} 
 \leq \| (\mathfrak a_t)_{t \in \I} \|_{\ell^1(\I;X)}. 
 \end{equation}
From H\"older's inequality one easily establishes the algebra property
\begin{equation}\label{var-alg}
\| (\mathfrak a_t \mathfrak b_t)_{t \in \I} \|_{\V^r} \lesssim \| (\mathfrak a_t)_{t \in \I} \|_{\V^r} \| (\mathfrak b_t)_{t \in \I} \|_{\V^r}
\end{equation}
for any scalar sequences $(\mathfrak a_t)_{t \in \I}$, $(\mathfrak b_t)_{t \in \I}$.

\subsection{Tensor products}

Given two functions $f \colon X \to \C$, $g \colon Y \to \C$, we define their tensor product $f \otimes g \colon X \to Y \to \C$ by the formula
$$ f \otimes g(x,y) \coloneqq f(x) g(y).$$
One can also define the formal tensor product $f \otimes g$ of elements $f \in V$, $g \in W$ of abstract vector spaces $V,W$, which takes values in the algebraic tensor product $V \otimes W$.  By abuse of notation, we identify these two notions of tensor product.

If $T_1 \colon V_1 \to W_1$, $T_2 \colon V_2 \to W_2$ are linear maps, we define the tensor product $T_1 \otimes T_2 \colon V_1 \otimes V_2 \to W_1 \otimes W_2$ as the unique linear map such that
\begin{equation}\label{tensor-1}
T_1 \otimes T_2 (f_1 \otimes f_2) = (T_1 f_1) \otimes (T_2 f_2)
\end{equation}
whenever $f_1 \in V_1, f_2 \in V_2$.
Similarly, if $B_1 \colon U_1 \times V_1 \to W_1$ and $B_2 \colon U_2 \times V_2 \to W_2$ are bilinear maps, we define $B_1 \otimes B_2 \colon (U_1 \otimes U_2) \times (V_1 \otimes V_2) \to W_1 \otimes W_2$ to be the unique bilinear map such that
\begin{equation}\label{bil-tensor}
B_1 \otimes B_2(f_1 \otimes f_2, g_1 \otimes g_2) = B_1(f_1,g_1) \otimes B_2(f_2,g_2)
\end{equation}
whenever $f_1 \in U_1, g_1 \in V_1, f_2 \in U_2, g_2 \in V_2$.  This algebraic tensor product can often be extended to analytic settings.  For instance, if $T_1 \colon L^p(X_1) \to L^q(Y_1)$ and $T_2 \colon L^p(X_2) \to L^q(Y_2)$ are integral operators of the form
$$ T_1 f_1(y_1) = \int_{X_1} K_1(x_1,y_1) f_1(x_1)\ d\mu_{X_1}(x_1)$$
and
$$ T_2 f_2(y_2) = \int_{X_2} K_2(x_2,y_2) f_2(x_2)\ d\mu_{X_2}(x_2)$$
one can define $T_1 \otimes T_2 \colon L^p(X_1 \times X_2) \to L^q(Y_1 \times Y_2)$ (formally, at least) by
$$ (T_1 \otimes T_2) f(y_1,y_2) = \int_{X_1 \times X_2} K_1(x_1,y_1) K_2(x_2,y_2) f(x_1,x_2)\ d\mu_{X_1}(x_1) d\mu_{X_2}(x_2).$$
We claim the multiplicativity property
\begin{equation}\label{ttensor}
\| T_1 \otimes T_2 \|_{L^p(X_1 \times X_2) \to L^q(Y_1 \times Y_2)} = \|T_1 \|_{L^p(X_1) \to L^q(Y_1)} \|T_2 \|_{L^p(X_2) \to L^q(Y_2)},
\end{equation}
in the case\footnote{There is another case where \eqref{ttensor} holds, namely when $q \geq p$ and no non-negativity hypothesis is assumed, by factoring $T_1 \otimes T_2 = (T_1 \otimes \mathrm{id}) \circ \mathrm{id} \circ (\mathrm{id} \otimes T_2)$ and establishing the inequalities $\| \mathrm{id} \otimes T_2 \|_{L^p(X_1 \times X_2) \to L^p(X_1; L^q(Y_2))} \leq \|T_2 \|_{L^p(X_2) \to L^q(Y_2)}$, $\| \mathrm{id} \|_{L^p(X_1; L^q(Y_2)) \to L^q(Y_2; L^p(X_1))} \leq 1$, and $\| T_1 \otimes \mathrm{id} \|_{L^q(Y_2; L^p(X_1)) \to L^q(Y_1 \times Y_2)} \leq \|T_1 \|_{L^p(X_1) \to L^q(Y_1)}$. However, this argument does not easily extend to the bilinear case, which is the case of most interest to us.}
where one of the kernels (say $K_1$) is non-negative, and assuming  $X_1,X_2,Y_1,Y_2$ are $\sigma$-finite with positive measure to avoid degeneracies, by the following argument.   The lower bound is clear by testing $T_1 \otimes T_2$ on tensor products $f_1 \otimes f_2$, so we focus on the upper bound (which is what is needed in our applications).  If $f \in L^p(X_1 \times X_2)$, we have
$$ (T_1 \otimes T_2) f(y_1,y_2) = \int_{X_1} K_1(x_1,y_1) T_2(f_{x_1})(y_2)\ d\mu_{X_1}(x_1)$$
where $f_{x_1} \colon x_2 \mapsto f(x_1,x_2)$ denotes the slice of $f$, hence for any $y_1 \in Y_1$ and by the non-negativity of $K_1$ we have
$$ \| (T_1 \otimes T_2) f(y_1,\cdot)\|_{L^q(Y_2)} \leq  \|T_2 \|_{L^p(X_2) \to L^q(Y_2)} \int_{X_1} K_1(x_1,y_1) \|f_{y_1}\|_{L^p(X_1)}\ d\mu_{X_1}(x_1).$$
Taking $L^q(Y_1)$ norms of both sides and using the Fubini--Tonelli theorem, we conclude that
$$ \| (T_1 \otimes T_2) f\|_{L^q(Y_1 \times Y_2)} \leq \|T_1 \|_{L^p(X_1) \to L^q(Y_1)} \|T_2 \|_{L^p(X_2) \to L^q(Y_2)} \|f\|_{L^p(X_1 \times X_2)},$$
giving the claim. An analogous argument gives the identity
\begin{multline}\label{ttensor-bil}
\| B_1 \otimes B_2 \|_{L^p(X_1 \times X_2) \times L^q(Y_1 \times Y_2) \to L^r(Z_1 \times Z_2)} \\
= \|B_1 \|_{L^p(X_1) \times L^q(Y_1) \to L^r(Z_1)} \|B_2 \|_{L^p(X_2) \times L^q(Y_2) \to L^r(Z_2)}
\end{multline}
for tensor products of bilinear operators, with (say) $B_1$ arising from a non-negative kernel, again assuming all spaces $\sigma$-finite with positive measure to avoid degeneracies.

\section{Transferring to the integer shift}\label{transfer-sec}

In this section we perform three standard and general reductions for our problem:

\begin{itemize}
    \item[(i)] By standard limiting arguments, we show that long variational estimates, such as the one in Theorem \ref{main}(iv), are sufficient to establish maximal inequalities, norm convergence, and pointwise almost everywhere convergence.  Thus we can focus exclusively on variational estimates in the sequel.
    \item[(ii)]  We apply the Calder\'on transference principle (see e.g., \cite[Appendix A]{DTT}) to transfer the long variational estimates to the integer shift system $\Z = (\Z, \mu_\Z, T_\Z)$.  As mentioned in the introduction, this allows us to exploit the Fourier-analytic structure of $\Z$ (and eventually, $\A_\Z$ as well).
    \item[(iii)]  We use a telescoping argument to replace the averaging operator 
    $$ A^{P_1,\dots,P_k}_N(f_1,\dots,f_k)(x) = \E_{n \in [N]} f_1(T^{P_1(n)} x) \dots f_k(T^{P_k(n)} x)$$
    with the upper half\footnote{One could also work with the normalized upper half $\frac{\lfloor N\rfloor}{\lfloor N/2\rfloor} \tilde A_N$ here if desired, though it makes little difference to the subsequent arguments other than adjusting a few constants by a factor of two.}
    \begin{equation}\label{tan-def}
    \tilde A_N^{P_1,\dots,P_k}(f_1,\dots,f_k)(x) = \E_{n \in [N]} f_1(T^{P_1(n)} x) \dots f_k(T^{P_k(n)} x) \ind{n>N/2}.
    \end{equation}
    This technical reduction is convenient as it allows one to avoid the stationary points of the polynomials $P_1,\dots,P_k$ (in particular, we get good lower bounds on the first derivatives of these polynomials).
\end{itemize}

These reductions are available for arbitrary non-conventional averages, not just for the bilinear averages $A^{\nn,P(\nn)}_N$ treated in this paper, so we give these reductions in the general setting.

\begin{proposition}[General reductions]\label{transf}
Let $(X,\mu,T)$ be a measure-preserving system, let $P_1(\nn),\dots,P_k(\nn) \in \Z[\nn]$, let $0 < p_1,\dots,p_k,p < \infty$, and let $1 \leq r<\infty$.
\begin{itemize}
    \item [(i)]  (Reduction to variational estimate) Suppose one has the variational estimate
    \begin{equation}\label{Var-est}
        \| (A^{P_1,\dots,P_k}_N(f_1,\dots,f_k))_{N \in \D} \|_{L^p(X; \V^r)} \lesssim_{p_1,\dots,p_k,p,P_1,\dots,P_k,r,\lambda} \|f_1\|_{L^{p_1}(X)} \dots \|f_k\|_{L^{p_k}(X)}
    \end{equation}
    for all $\lambda > 1$ and $f_i \in L^{p_i}(X)$, $i=1,\dots,k$, and all finite $\lambda$-lacunary subsets $\D$ of $[1,+\infty)$.  Then one has the maximal inequality
    \begin{equation}\label{max-est}
        \| (A^{P_1,\dots,P_k}_N(f_1,\dots,f_k))_{N \in \Z_+} \|_{L^p(X;\ell^\infty)} \lesssim_{p_1,\dots,p_k,p,P_1,\dots,P_k,r} \|f_1\|_{L^{p_1}(X)} \dots \|f_k\|_{L^{p_k}(X)}
    \end{equation}
    and for any $f_i \in L^{p_i}(X)$, $i=1,\dots,k$, the averages $A^{P_1,\dots,P_k}_N(f_1,\dots,f_k)$ converge pointwise almost everywhere and in $L^p(X)$ norm.
    \item[(ii)]  (Calder\'on transference principle)  Suppose that we are in the H\"older exponent case $\frac{1}{p_1}+\dots+\frac{1}{p_k} = \frac{1}{p}$.  Then in order to establish \eqref{Var-est} for arbitrary measure-preserving systems $X = (X,\mu,T)$, it suffices to show \eqref{Var-est} for the integer shift model $\Z = (\Z,\mu_\Z,T_\Z)$.
    \item[(iii)]  (Telescoping argument)  In order to establish \eqref{Var-est} under the assumptions of (i), it suffices to establish the bound
    \begin{equation}\label{var-half}
    \| (\tilde A^{P_1,\dots,P_k}_N(f_1,\dots,f_k))_{N \in \D} \|_{L^p(X; \V^r)} \lesssim_{p_1,\dots,p_k,p,P_1,\dots,P_k,r,\lambda} \|f_1\|_{L^{p_1}(X)} \dots \|f_k\|_{L^{p_k}(X)}
    \end{equation}
    under the same assumptions, where $\tilde A^{P_1,\dots,P_k}_n$ is defined in \eqref{tan-def}.
\end{itemize}
\end{proposition}

Note that all of the reductions in this proposition apply in both the Banach exponent case $p \geq 1$ and the non-Banach exponent case $0 < p < 1$. However, we emphasize that the Calder\'on transference principle (ii) is only available in the H\"older exponent case $\frac{1}{p_1}+\dots+\frac{1}{p_k}=\frac{1}{p}$.

\begin{proof} To simplify the notation we allow all implied constants to depend on $p_1,\dots,p_k,P_1,\dots,P_k,r$.

We begin with (i). Fix $f_1,\dots,f_k$, and abbreviate $A^{P_1,\dots,P_k}_N(f_1,\dots,f_k)(x)$ as $a_N(x)$ for any $N \geq 1$.
For any $s \in \Z_+$, introduce the $2^{1/s}$-lacunary set 
$$ 2^{\N/s} \coloneqq \{ 2^{n/s} \colon n \in \N \}$$
(note here we exploit the freedom to choose scales $N$ that are real-valued rather than integer-valued). From \eqref{Var-est} and monotone convergence we have
\begin{equation}\label{var2}
 \| (a_N)_{N \in 2^{\N/s} } \|_{L^p(X; \V^r)} \lesssim_s \|f_1\|_{L^{p_1}(X)} \dots \|f_k\|_{L^{p_k}(X)}.
 \end{equation}

To prove \eqref{max-est}, we may assume without loss of generality that $f_1,\dots,f_k$ are non-negative, thanks to the pointwise triangle inequality
$$ |A^{P_1,\dots,P_k}_N(f_1,\dots,f_k)| \leq A^{P_1,\dots,P_k}_N(|f_1|,\dots,|f_k|).$$
In the non-negative case we have the additional pointwise bound
$$ \sup_{N \in \Z_+} a_N(x) \leq 2 \sup_{N \in 2^\N} a_N(x)$$
and the claim \eqref{max-est} now follows from \eqref{var2}.

Now we establish pointwise convergence.  By linearity we may assume that the $f_1,\dots,f_k$ are all non-negative. From \eqref{var2}, \eqref{max-est}, we see that for almost all $x \in X$, the quantity
$$ M(x) \coloneqq \sup_{N \in \Z_+} a_N(x)$$
is finite, as are the variational norms $\| (a_N)_{N \in 2^{\N/s}} \|_{\V^r}$ for every $s \in \Z_+$. From the latter we conclude that the limits $\lim_{N \to \infty; N \in 2^{\N/s}} a_N(x)$ exist almost everywhere for all $s \geq 1$; since $2^{\N} \subset 2^{\N/s}$, this limit is independent of $s$, thus
$$ \lim_{N \to \infty; N \in 2^{\N/s}} a_N(x) = a_\infty(x)$$
for some $a_\infty(x)$.  For any sufficiently large $N$, if we let $N'$ be the first element of $2^{\N/s}$ greater than or equal to $N$ we see from the triangle inequality that
$$ a_N(x) = a_{N'}(x) + O( M(x)/s )$$
hence on taking limits
$$ \liminf_{N \to \infty} a_N(x), \limsup_{N \to \infty} a_N(x) = a_\infty(x) + O(M(x)/s);$$
sending $s \to \infty$, we conclude that $a_N(x)$ converges to $a_\infty(x)$ as $N \to \infty$ as claimed.
Finally, norm convergence follows from pointwise convergence, the maximal inequality, and the dominated convergence theorem.  This proves (i).

Now we prove (ii).  This follows from the general Calder\'on transference principle \cite{C1}, but for the convenience of the reader we supply a proof here.  We first observe from the Fubini--Tonelli theorem and H\"older's inequality (and the H\"older exponent hypothesis $\frac{1}{p_1}+\dots+\frac{1}{p_k} = \frac{1}{p}$) that if \eqref{Var-est} is established for the integer shift model $(\Z,\mu_\Z,T_\Z)$ then it automatically holds for any product system $(X \times \Z, \mu \times \mu_\Z, \mathrm{id} \times T_\Z)$, where $(X,\mu)$ is an arbitrary $\sigma$-finite measure space and $\mathrm{id} \times T_\Z$ is the shift $(x,n) \mapsto (x,n-1)$, since there is no interaction between the individual fibers $\{x\} \times \Z, x \in X$ of this system.  

Now let $(X,\mu,T)$ be an arbitrary measure-preserving system.  
To prove \eqref{Var-est}, it suffices by multilinearity to do so when the $f_i$ are non-negative.  We may assume that each of the $f_i$ are bounded and supported on a set of finite measure.  We may normalize $\|f_i\|_{L^{p_i}(X)}=1$ for $i=1,\dots,k$, thus our task is now to show that
$$\| (A^{P_1,\dots,P_k}_{N,X}(f_1,\dots,f_k))_{N \in \D} \|_{L^p(X; \V^r)} \lesssim_\lambda 1.$$
Now let $M$ be a large natural number, let $D:=\max_{i\in[k]}\deg P_i$, and let $C>0$ be a quantity to be specified later that can depend on $\D,P_1,\dots,P_k$ but is independent of $M$.
On the product system $X \times \Z = (X \times \Z, \mu \times \mu_\Z, \mathrm{id} \times T_\Z)$ define the functions
$$ f_{i,M}(x,n) \coloneqq \ind{[3CM^D]}(n) f_i(T^{-n} x)$$
for $i=1,\dots,k$.  From the Fubini--Tonelli theorem and the measure-preserving nature of $T$ one has
$$ \| f_{i,M} \|_{L^{p_i}(X \times \Z)} = (3CM)^{D/p_i}.$$
Also, we observe the identity
$$ \| ( A^{P_1,\dots,P_k}_{N,X \times \Z}(f_{1,M},\dots,f_{k,M})(x,n) )_{N \in \D\cap[M]} \|_{\V^r}
= \| ( A^{P_1,\dots,P_k}_{N,X}(f_1,\dots,f_k)(T^{-n} x))_{N \in \D\cap[M]} \|_{\V^r}$$
whenever $CM^D\leq n \leq 2CM^D$.  From the Fubini--Tonelli theorem again, we conclude that
\begin{multline*}
\| ( A^{P_1,\dots,P_k}_{N,X \times \Z}(f_{1,M},\dots,f_{k,M}))_{N \in \D} \|_{L^p(X \times \Z; \V^r)}\\
\geq (CM^D+1)^{1/p} \| ( A^{P_1,\dots,P_k}_{N,X}(f_1,\dots,f_k))_{N \in \D\cap[M]} \|_{L^p(X;\V^r)}.    
\end{multline*}
Applying \eqref{Var-est} to the product system $X \times \Z$, we conclude that
$$ \| ( A^{P_1,\dots,P_k}_{N,X}(f_1,\dots,f_k))_{N \in \D\cap[M]} \|_{L^p(X;\V^r)} \lesssim_\lambda M^{-D/p} M^{D/p_1} \dots M^{D/p_k};$$
using the H\"older exponent hypothesis $1/p_1+\dots+1/p_k=1/p$ and sending $M \to \infty$, we obtain the claim.

Finally, we prove (iii).  By linearity we may take $f_1,\dots,f_k$ to be nonnegative.  Fix $\lambda > 1$, and set
$$ \tilde a_N(x) \coloneqq \tilde A^{P_1,\dots,P_k}_N(f_1,\dots,f_k)(x).$$
We observe the telescoping identity
$$ a_N = \sum_{k=0}^\infty \frac{\lfloor N/2^k \rfloor}{\lfloor N \rfloor} \tilde a_{N/2^k} \ind{2^k \leq N}.$$
We have $\frac{\lfloor N/2^k \rfloor}{\lfloor N \rfloor} = 2^{-k} + O(1/N)$, and hence by the triangle inequality we have the pointwise estimate
$$
\| (a_N)_{N \in \D} \|_{\V^r} \leq \sum_{k=0}^\infty 2^{-k} \| (\tilde a_{N/2^k} \ind{2^k \leq N})_{N \in \D} \|_{\V^r}
+ O\Big( \sum_{k=0}^\infty \sum_{N \in \D} \frac{1}{N} \ind{2^k \leq N} |\tilde a_{N/2^k}| \Big)$$
for all $x \in X$.  Since the rescaling $\{ N/2^k: N \in \D, 2^k \leq N\}$ of a $\lambda$-lacunary set $\D$ is still $\lambda$-lacunary, we have from \eqref{var-half} that
$$ \| (\tilde a_{N/2^k} \ind{2^k \leq N})_{N \in \D} \|_{L^p(X;\V^r)} \lesssim_\lambda \|f_1\|_{L^{p_1}(X)} \dots \|f_r\|_{L^{p_r}(X)}.$$
From \eqref{var-half} applied to singleton $\lambda$-lacunary sets we have
$$ \| \tilde a_{N/2^k} \|_{L^p(X)} \lesssim_\lambda \|f_1\|_{L^{p_1}(X)} \dots \|f_r\|_{L^{p_r}(X)}.$$
Summing in $N,k$, using the triangle inequality $\| \sum_i f_i \|_{L^p(X)} \leq \sum_i \|f_i\|_{L^p(X)}$ (when $p \geq 1$) or the quasi-triangle inequality 
\begin{equation}\label{quasi}
\Big\| \sum_i f_i \Big\|_{L^p(X)}^p \leq \sum_i \|f_i\|_{L^p(X)}^p
\end{equation}
(when $0 < p < 1$), we obtain the claim.
\end{proof}

\begin{remark} A modification of the Calder\'on transference principle also allows us to handle measure-preserving systems in which the shift map $T$ is not assumed to be invertible, as long as we also require the polynomials $P_1,\dots,P_k$ to be non-negative on $\Z_+$ so that the averaging operators $A^{P_1,\dots,P_k}_N$ remain well-defined.  We leave the details to the interested reader.  
\end{remark}

In view of this general proposition, Theorem \ref{main} will now follow from

\begin{theorem}[Variational ergodic theorem on the integers]\label{end-var} Let
$P(\nn) \in \Z[\nn]$ have degree $d \geq 2$, let $1 \leq p_1,p_2,p < \infty$ be such that $\frac{1}{p_1}+\frac{1}{p_2}=\frac{1}{p}$, and let $f \in \ell^{p_1}(\Z), g \in \ell^{p_2}(\Z)$. If $r>2$ and $\lambda>1$, then
    \begin{equation}\label{var-poly-main-int}
    \| ( \tilde A^{\nn,P(\nn)}_N(f,g) )_{N \in \D} \|_{\ell^p(\Z; \V^r)} \lesssim_{p_1,p_2,r,P,\lambda} \|f\|_{\ell^{p_1}(\Z)} \|g\|_{\ell^{p_2}(\Z)}
    \end{equation}
    for all finite $\lambda$-lacunary subsets $\D$ of $[1,+\infty)$.
\end{theorem}

It remains to establish Theorem \ref{end-var}.  This is the objective of much of the remainder of the paper. 

\begin{remark}\label{need-hold} It is essential in Theorem \ref{end-var} for ergodic theory applications that one has the H\"older condition $\frac{1}{p_1}+\frac{1}{p_2}=\frac{1}{p}$.  In the super-H\"older regime $\frac{1}{p_1} + \frac{1}{p_2} > \frac{1}{p}$ it is easy to establish \eqref{var-poly-main-int}; for instance when $(p_1,p_2,p)=(2,2,\infty)$ it follows from Cauchy-Schwarz that 
\begin{equation}\label{apn-decay}
\| \tilde A^{\nn,P(\nn)}_N(f,g)\|_{\ell^\infty(\Z)} \lesssim_P N^{-1} \|f\|_{\ell^2(\Z)} \|g\|_{\ell^2(\Z)},
\end{equation}
and by interpolating this with \eqref{anf} it is not difficult to establish \eqref{var-poly-main-int} for any $1 < p_1,p_2,p \leq \infty$ with $\frac{1}{p_1}+\frac{1}{p_2} > \frac{1}{p}$.  However, in this regime the Calder\'on transference principle no longer applies and so no consequences to general measure preserving systems (in particular those of finite measure) can be concluded.  Indeed, the decay in $N$ exhibited by \eqref{apn-decay} is not possible in the finite measure setting since $A^{\nn,P(\nn)}(1,1)=1$.  In the opposite sub-H\"older regime $\frac{1}{p_1} + \frac{1}{p_2} < \frac{1}{p}$ even single-scale boundedness $\| \tilde A^{\nn, P(\nn)}_N \|_{\ell^{p_1}(\Z) \times \ell^{p_2}(\Z) \to \ell^p(\Z)} < \infty$ fails on the integer shift model, as can be seen by testing the operator on indicator functions of large intervals.  (However, on finite measure systems one can of course deduce sub-H\"older exponent estimates from H\"older exponent estimates by applying H\"older's inequality.)
\end{remark}

\section{Abstract harmonic analysis: relating the integers to the adelic integers}\label{abstract-sec}

We will be performing Fourier analysis on many different groups in this paper, and in particular exploiting the close relationship between major arc Fourier analysis on the integers $\Z$ on the one hand, and low frequency Fourier analysis on the adelic integers $\A_\Z$ on the other hand (see Figure \ref{fig:major}).  It will be convenient to set out some abstract harmonic analysis notation to perform this analysis in a unified fashion. We let $\TT \coloneqq \R/\Z$ denote the unit circle, and $e \colon \TT \to \C$ denote the standard character $e(\theta) \coloneqq e^{2\pi i \theta}$.

\begin{definition}[Pontryagin duality]  An \emph{LCA group} is a locally compact abelian group $\G = (\G,+)$ equipped with a Haar measure $\mu_{\G}$.  A \emph{Pontryagin dual} of an LCA group $\G$ is an LCA group $\G^* = (\G^*,+)$ with a Haar measure $\mu_{\G^*}$ and a continuous bihomomorphism $(x,\xi) \mapsto x \cdot \xi$ (which we call a \emph{pairing}) from $\G \times \G^*$ to the unit circle $\TT = \R/\Z$, such that the Fourier transform $\F_{\G} \colon L^1(\G) \to C(\G^*)$ defined by
$$ \F_{\G} f(\xi) \coloneqq \int_{\G} f(x) e(x \cdot \xi)\ d\mu_{\G}(x)$$
extends to a unitary map from $L^2(\G)$ to $L^2(\G^*)$; in particular we have the Plancherel identity
$$ \int_\G |f(x)|^2\ d\mu_\G(x) = \int_{\G^*} |\F_\G f(\xi)|^2\ d\mu_{\G^*}(\xi)$$
for all $f \in L^2(\G)$.

If $\Omega \subset \G^*$ is measurable, we say that $f \in L^2(\G)$ is \emph{Fourier supported} in $\Omega$ if $\F_{\G} f$ vanishes outside of $\Omega$ (modulo null sets).  The space of such functions will be denoted $L^2(\G)^\Omega$.
\end{definition}

As is well known (see e.g., \cite{Rudin}), every LCA group $\G$ has a Pontryagin dual $\G^*$, and the inverse Fourier transform $\F_\G^{-1} \colon L^2(\G^*) \to L^2(\G)$ is then given for  $F\in L^1(\G^*) \cap L^2(\G^*)$ by the formula
$$ \F_\G^{-1} F(x) = \int_{\G^*} F(\xi) e(-x \cdot \xi)\ d\mu_{\G^*}(\xi).$$

We will work with the following concrete pairs $(\G,\G^*)$ of Pontryagin dual LCA groups:

\begin{itemize}
    \item [(i)]  If $\G = \R$ with Lebesgue measure $\mu_\R = dx$, then $\G^* = \R^* = \R$ with Lebesgue measure $\mu_{\R^*} = d\xi$ is a Pontryagin dual, with pairing $x \cdot \xi \coloneqq x \xi \mod 1$.
    \item[(ii)]  If $\G = \Z$ with counting measure $\mu_\Z$, then $\G^* = \TT$ with Lebesgue measure $\mu_{\TT} = d\xi$ is a Pontryagin dual, with pairing $x \cdot \xi \coloneqq x \xi$.
    \item[(iii)]  If $\G = \Z/Q\Z$ is a cyclic group for some $Q \in \Z_+$ with normalized counting measure $\int_{\Z/Q\Z} f(x)\ d\mu_{\Z/Q\Z}(x) \coloneqq \E_{x \in \Z/Q\Z} f(x)$, then the \emph{dual cyclic group} $\G^* = \frac{1}{Q} \Z/\Z$ with counting measure $\mu_{\frac{1}{Q}\Z/\Z}$ is a Pontryagin dual, with pairing $x \cdot \xi \coloneqq x \xi$.
    \item[(iv)]  If $\G = \Z_p \coloneqq \varprojlim_j \Z/p^j\Z$ is the compact group of $p$-adic integers with Haar probability measure $\mu_{\Z_p}$ (the inverse limit of normalized counting measures on $\Z/p^j\Z$) for some prime $p \in \PP$, then the discrete group $\G^* = \Z_p^* = \varinjlim_j \frac{1}{p^j}\Z/\Z = \Z[\frac{1}{p}]/\Z$ with counting measure $\mu_{\Z_p^*}$ is a Pontragin dual, with pairing $x \cdot (\frac{a}{p^j} \mod 1) \coloneqq \frac{xa \mod p^j}{p^j}$.
    \item[(v)]  If $\G = \hat \Z \coloneqq \prod_{p \in \PP} \Z_p$ is the compact group of profinite integers with Haar probability measure, then the discrete group $\G^* = \hat \Z^* = \coprod_{p \in \PP} \Z_p^* = \Q/\Z$ of ``arithmetic frequencies'' with counting measure $\mu_{\Q/\Z}$ is a Pontragin dual, with pairing $x \cdot (\frac{a}{q} \mod 1) \coloneqq \frac{xa \mod q}{q}$.
    \item[(vi)]  If $\G_1,\G_2$ are LCA groups with Pontryagin duals $\G_1^*, \G_2^*$, then the product $\G_1 \times \G_2$ (with product Haar measure) is an LCA group with Pontryagin dual $\G_1^* \times \G_2^*$ and pairing $(x_1,x_2) \cdot(\xi_1,\xi_2) \coloneqq x_1 \cdot \xi_1 + x_2 \cdot \xi_2$.  In particular, if $\G = \A_\Z \coloneqq \R \times \hat \Z$ is the adelic integers\footnote{The adelic integers $\A_\Z$ should not be confused with the larger ring $\A_\Q = \A_\Z \otimes_\Z \Q$ of adelic numbers, which we will not use in this paper.} (with the product Haar measure $\mu_{\A_\Z} \coloneqq \mu_\R \times \mu_{\hat \Z}$), then \emph{adelic frequency space} $\G^* = \A_\Z^* = \R \times \Q/\Z$ is a Pontryagin dual (with product measure $\mu_{\R \times \Q/\Z} \coloneqq \mu_\R \times \mu_{\Q/\Z}$ and the indicated pairing).  Similarly, for any $Q \in \Z_+$, $\R \times \Z/Q\Z$ has $\R \times \frac{1}{Q}\Z/\Z$ as its Pontryagin dual.
\end{itemize}

\begin{remark}  Heuristically, one can think of analysis on the adelic integers $\A_\Z$ (resp. the profinite integers $\hat \Z$, or the $p$-adic integers $\Z_p$) as an abstraction of analysis on the product groups $\R \times \Z/Q\Z$ (resp. the cyclic groups $\Z/Q\Z$, $\Z/p^j \Z$) in which all estimates are required to be uniform in the parameter $Q$ or $p^j$.  These abstractions are convenient to use in settings in which one does not wish to fix an ambient modulus $Q$ or $p^j$ in advance.
\end{remark}

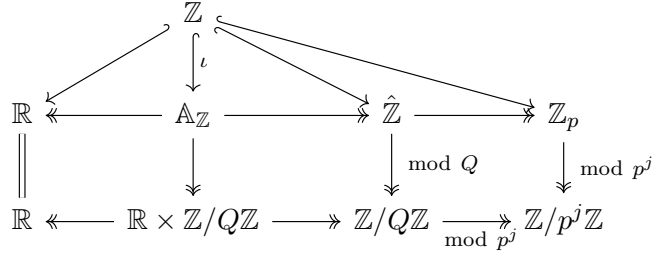
\begin{figure}
    \centering
    \begin{tikzcd}
    & \Z \arrow[d, hook,"\iota"] \arrow[dl,hook'] \arrow[dr,hook'] \arrow[drr,hook] \\ \R \arrow[d, equal] & \A_\Z \arrow[l,two heads] \arrow[d,two heads] \arrow[r,two heads] & \hat \Z \arrow[d,two heads, "\mod Q"] \arrow[two heads,r] & \Z_p \arrow[d, two heads,"\mod p^j"] \\
    \R & \R \times \Z/Q\Z \arrow[l,two heads] \arrow[r,two heads] & \Z/Q\Z  \arrow[r, "\mod p^j"',two heads] & \Z/p^j \Z
    \end{tikzcd}
    \caption{A commutative diagram of the various physical space LCA groups used in this paper, with the arrows indicating continuous homomorphisms.  Here $Q$ is a positive integer, and $p^j$ is a prime power dividing $Q$. Double-headed arrows are surjective; arrows with hooks are injective. The left column contains ``continuous'' groups, the right two columns contain ``arithmetic'' groups (and are compact), and the second column from the left contain groups exhibiting both continuous and arithmetic aspects. The second row is the inverse limit of the third. Note the central role played by the adelic integers $\A_\Z$.}
    \label{fig:phys}
\end{figure}

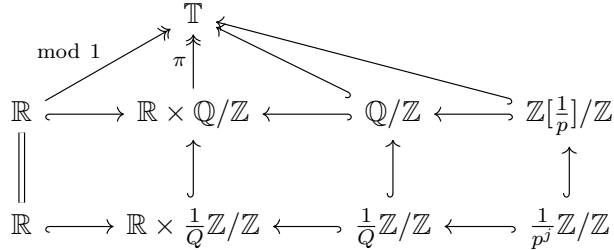
\begin{figure}
    \centering
    \begin{tikzcd}
    & \TT \\ \R \arrow[ur,two heads, "\mod 1"] \arrow[d,equal] \arrow[r,hook'] & \R \times \Q/\Z \arrow[u, two heads,"\pi"]  & \Q/\Z \arrow[ul,hook] \arrow[l,hook] & \Z[\frac{1}{p}]/\Z \arrow[l,hook] \arrow[ull,hook] \\
    \R  \arrow[r,hook'] & \R \times \frac{1}{Q} \Z/\Z \arrow[u,hook]  & \frac{1}{Q}\Z/\Z \arrow[u,hook] \arrow[l,hook] & \frac{1}{p^j}\Z/\Z \arrow[l,hook] \arrow[u,hook]
    \end{tikzcd}
    \caption{A commutative diagram of the various frequency space LCA groups used in this paper.  The groups in the right two columns are discrete. The second row is the direct limit of the third. Note the duality with Figure \ref{fig:phys} (this can be made precise using Fourier adjoint relationships such as \eqref{adjoint}).}
    \label{fig:freq}
\end{figure}

Observe that we have quotient homomorphisms $x \mapsto x \mod Q$ from $\Z$ to $\Z/Q\Z$ or from $\hat \Z$ to $\Z/Q\Z$, $x \mapsto x \mod p^j$ from $\Z_p$ to $\Z/p^j\Z$, and $x \mapsto x \mod 1$ from $\R$ to $\TT$.  The adelic integers $\A_\Z$ capture two important limiting behaviours of the integers $\Z$; the continuous behaviour (as described by the $\R$ factor), and the arithmetic behaviour (as described by the $\hat \Z$ factor).  We also have the inclusion homomorphism $\iota \colon \Z \to \A_\Z$ defined by
$$ \iota(x) \coloneqq \left(x, ( (x \mod p^j)_{j \in \N} )_{p \in \PP} \right)$$
and the addition homomorphism $\pi \colon \R \times \Q/\Z \to \TT$ defined by
$$ \pi(\theta,\alpha) \coloneqq \alpha + \theta;$$
these two maps are Fourier adjoint to each other in the sense that
\begin{equation}\label{adjoint} \iota(x) \cdot \xi = x \cdot \pi(\xi)
\end{equation}
for all $x \in \Z$ and $\xi \in \R \times \Q/\Z$.
In  ``major arc'' regimes we will be able to use these homomorphisms to ``approximate'' $\Z$ by $\A_\Z$, which in principle decouples the discrete harmonic analysis of $\Z$ from the continuous harmonic analysis of $\R$ and the arithmetic harmonic analysis of $\hat \Z$.  We summarized the relations between the various LCA groups in Figures \ref{fig:phys}, \ref{fig:freq}.

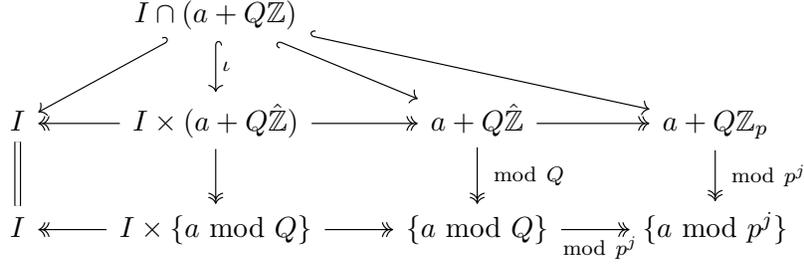
\begin{figure}
    \centering
    \begin{tikzcd}
    & I \cap (a+Q\Z) \arrow[d, hook,"\iota"] \arrow[dl,hook'] \arrow[dr,hook'] \arrow[drr,hook] \\ I \arrow[d,equal] & I \times (a+Q\hat \Z) \arrow[l,two heads] \arrow[d,two heads] \arrow[r,two heads] & a+Q\hat \Z \arrow[d,two heads, "\mod Q"] \arrow[two heads,r] & a+Q \Z_p \arrow[d, two heads,"\mod p^j"] \\
    I & I \times \{ a \mod Q \} \arrow[l,two heads] \arrow[r,two heads] & \{ a \mod Q\}  \arrow[r, "\mod p^j"',two heads] & \{a \mod p^j\}
    \end{tikzcd}
    \caption{A restriction of the physical space diagram in Figure \ref{fig:phys} to an arithmetic progression $I \cap (a+Q\Z)$ formed by intersecting an interval $I \subset \R$ with an infinite arithmetic progression $a+Q\Z$.  The sets here are no longer groups in general (except in an ``approximate'' sense) and so the arrows no longer denote homomorphisms.  As in previous figures, $p^j$ is understood to be a prime power dividing $Q$.  Note how this diagram separates an arithmetic progression into its continuous and arithmetic components.}
    \label{fig:unc}
\end{figure}

\begin{remark}  As is well known, the embedding $\iota$ identifies $\Z$ with a cocompact lattice $\iota(\Z)$ in $\A_\Z$ (thus $\iota(\Z)$ is a discrete subgroup of $\A_\Z$ and the quotient $\A_\Z/\iota(\Z)$ is compact).  Thus $\A_\Z$ is in some sense only ``slightly'' larger than $\Z$ itself, but has the advantage of splitting completely into a continuous component $\R$ and an arithmetic component $\hat \Z$, whereas $\Z$ does not directly have such a splitting.  However, the point is that after restricting attention to major arcs, one can partially move back and forth between the integers and adelic integers, and thus have some chance of exploiting the product structure of $\A_\Z = \R \times \hat \Z$ to decouple the continuous and arithmetic aspects of the analysis.
\end{remark}

For various LCA groups $\G$ we shall work with a space $\Schwartz(\G) \subset L^1(\G) \cap L^\infty(\G)$ of Schwartz--Bruhat functions $f \colon G \to \C$, generalizing the classical class of Schwartz functions on $\R$ that serve as a useful class of ``nice'' functions that are dense in $L^p(\G)$ for every $1 \leq p < \infty$ and behave well with respect to Fourier-analytic operations.  A definition of this space for arbitrary LCA groups can be found for instance in \cite{bruhat}, \cite{osb}, but for the purpose of this paper we shall only need the following special cases:
\begin{itemize}
    \item[(i)]  $\Schwartz(\R)$ is the space of Schwartz functions on $\R$.
    \item[(ii)]  $\Schwartz(\Z)$ is the space of rapidly decreasing functions on $\Z$, and $\Schwartz(\TT)$ is the space of smooth functions on $\TT$.
    \item[(iii)]  $\Schwartz(\Z/Q\Z)$ is the space of arbitrary functions on $\Z/Q\Z$, and similarly for $\Schwartz(\frac{1}{Q}\Z/\Z)$.
    \item[(iv)]  $\Schwartz(\Z_p)$ is the space of locally constant functions $f$ on $\Z_p$, or equivalently those functions of the form $f(x) = f_j(x \mod p^j)$ for some $j \in \N$ and some function $f_j \colon \Z/p^j\Z \to \C$.  $\Schwartz(\Z_p^*)$ is the space of finitely supported functions on $\Z_p^*$.
    \item[(v)]  $\Schwartz(\hat \Z)$ is the space of locally constant functions $f$ on $\hat \Z$, or equivalently those functions of the form $f(x) = f_Q(x \mod Q)$ for some $Q \in \Z_+$ and $f_Q \colon \Z/Q\Z \to \C$.  $\Schwartz(\hat \Z^*)$ is the space of finitely supported functions on $\hat \Z^*$.
    \item[(vi)] $\Schwartz(\R \times \Z/Q\Z)$ is the space of functions that is Schwartz in the $\R$ variable, and similarly for $\Schwartz(\R \times \frac{1}{Q} \Z/\Z)$.
    \item[(vii)]  $\Schwartz(\A_\Z)$ is the space of functions of the form $f(x,y) = f_Q(x,y \mod Q)$ for some $Q \in \Z_+$ and $f_Q \colon \R \times \Z/Q\Z \to \C$ that is Schwartz in the $\R$ variable.  $\Schwartz(\R \times \Q/\Z)$ is the space of functions supported on $\R \times \Sigma$ for some finite set $\Sigma \subset \Q/\Z$ and Schwartz in the $\R$ variable.
    \item[(viii)]  If $\G_1,\G_2$ are any two of the groups listed above, we define the Schwartz--Bruhat space $\Schwartz(\G_1 \times \G_2)$ on the product LCA group $\G_1 \times \G_2$ in the obvious fashion, and note that if $f_1 \in \Schwartz(\G_1)$ and $f_2 \in \Schwartz(\G_2)$ then $f_1 \otimes f_2$ can be identified with an element of $\Schwartz(\G_1 \times \G_2)$.
\end{itemize}

One could place a topology on the Schwartz--Bruhat spaces $\Schwartz(\G)$, but we will not need to do so here.
As is well known, the Fourier transform $\F_{\G}$ is a bijection from $\Schwartz(\G)$ to $\Schwartz(\G^*)$ for any of the groups $\G$ in Figure \ref{fig:phys}.  The Fourier transform can also be extended to vector-valued functions taking values in a finite-dimensional vector space $V$ in the obvious fashion.

If $\Omega \subset \G^*$, we let $\Schwartz(\G)^\Omega$ denote the subspace of $\Schwartz(\G)$ consisting of functions that are Fourier supported on $\Omega$, and $\Schwartz(\Omega)$ the subspace of $\Schwartz(\G^*)$ consisting of functions that are supported on $\Omega$.  Thus $\F_{\G}$ is also a bijection between $\Schwartz(\G)^\Omega$ and $\Schwartz(\Omega)$.

The inclusion homomorphism $\iota \colon \Z \to \A_\Z$ gives rise to a \emph{sampling map} $\Sample \colon \Schwartz(\A_\Z) \to \Schwartz(\Z)$ defined by
$$ \Sample f(x) \coloneqq f( \iota(x) )$$
for $x \in \Z$ and $f \in \Schwartz(\A_\Z)$.  Dually, the addition homomorphism $\pi \colon \R \times \Q/\Z \to \TT$ gives rise to a \emph{projection map} $\Proj \colon  \Schwartz(\R \times \Q/\Z) \to \Schwartz(\TT)$, defined by the formula
$$ \Proj F(\xi) \coloneqq \sum_{(\theta, \alpha) \in \pi^{-1}(\xi)} F(\theta, \alpha)$$
for $\theta \in \R$, $\alpha \in \Q/\Z$, and $F \in \Schwartz(\R \times \Q/\Z)$ (note that the definition of $\Schwartz(\R \times \Q/\Z)$ ensures that this sum contains at most countably many non-zero terms).  From \eqref{adjoint} one has the identity
$$ \F_\Z^{-1} \circ \Proj = \Sample \circ \F_{\A_\Z}^{-1}$$
or equivalently the adelic Poisson summation formula
$$ \F_\Z \circ \Sample = \Proj \circ \F_{\A_\Z}$$
and so we have the commutative diagram
\begin{center}
    \begin{tikzcd}
    \Schwartz(\Z) \arrow[d, "\F_\Z"] & \Schwartz(\A_\Z) \arrow[l, "\Sample"] \arrow[d, "\F_{\A_\Z}"] \\
    \Schwartz(\TT) & \Schwartz(\R \times \Q/\Z) \arrow[l,"\Proj"]
\end{tikzcd}.
\end{center}
See also Figures \ref{fig:schwartz-phys}, \ref{fig:schwartz-freq}.

\begin{figure}
    \centering
    \begin{tikzcd}
    & \Schwartz(\Z) \\
 \Schwartz(\R) \arrow[d, equal] \arrow[ur] \arrow[r,dotted] & \Schwartz(\A_\Z) \arrow[u,"\Sample"] & \Schwartz(\hat \Z)  \arrow[l,dotted] & \Schwartz(\Z_p) \arrow[l]  \\
    \Schwartz(\R) \arrow[r,dotted] & \Schwartz(\R \times \Z/Q\Z) \arrow[u] & \Schwartz(\Z/Q\Z) \arrow[u] \arrow[l,dotted] & \Schwartz(\Z/p^j \Z) \arrow[l] \arrow[u]
    \end{tikzcd}
    \caption{Schwartz--Bruhat spaces on physical space LCA groups.  Solid arrows indicate canonical linear maps of a ``sampling'' or ``pullback'' nature; dotted arrows from two spaces $V_1,V_2$ to a third $V$ indicate the existence of a tensor product operation $\otimes \colon V_1 \times V_2 \to V$.  The second row is the direct limit of the third. Compare with Figure \ref{fig:phys}. (Some arrows in that figure do not have an analogue here, basically because $\Schwartz(\R)$ does not contain a multiplicative unit $1$, and the inclusions of $\Z$ into $\hat \Z$ and $\Z_p$ are not proper.)}
    \label{fig:schwartz-phys}
\end{figure}
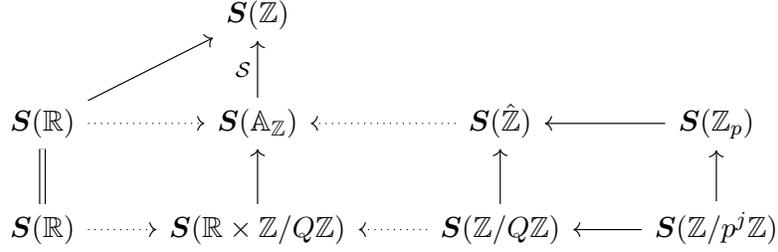

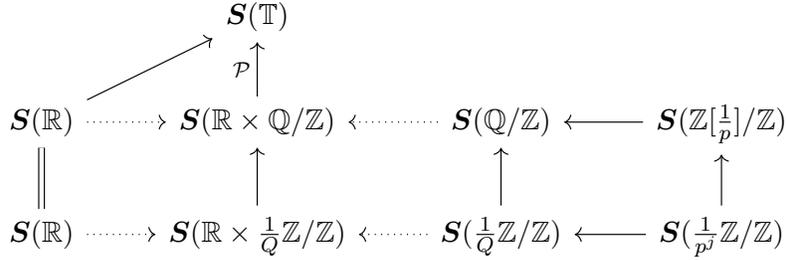
\begin{figure}
    \centering
    \begin{tikzcd}
    & \Schwartz(\TT) \\
 \Schwartz(\R) \arrow[d, equal] \arrow[ur] \arrow[r,dotted] & \Schwartz(\R \times \Q/\Z) \arrow[u,"\Proj"] & \Schwartz(\Q/\Z) \arrow[l,dotted] & \Schwartz(\Z[\frac{1}{p}]/\Z) \arrow[l] \\
    \Schwartz(\R) \arrow[r,dotted] & \Schwartz(\R \times \frac{1}{Q}\Z/\Z) \arrow[u] & \Schwartz(\frac{1}{Q}\Z/\Z) \arrow[u] \arrow[l,dotted] & \Schwartz(\frac{1}{p^j}\Z/\Z) \arrow[l] \arrow[u]
    \end{tikzcd}
    \caption{Schwartz--Bruhat spaces on frequency space LCA groups.  Solid arrows indicate canonical linear maps of a ``projection'' or ``pushforward'' nature; dotted arrows indicate a tensor product as in Figure \ref{fig:schwartz-phys}. The second row is the direct limit of the third.  This figure and the preceding one are intertwined by the Fourier transform via various forms of the Poisson summation formula.  Compare also with Figure \ref{fig:freq}. (Some arrows in that figure do not have an analogue here, basically because $\Schwartz(\R)$ does not contain a convolution unit $\delta$, and the embeddings of $\Q/\Z$ and $\Z[\frac{1}{p}]/\Z$ into $\TT$ are not open.)}
    \label{fig:schwartz-freq}
\end{figure}

A key difficulty here is that of \emph{aliasing}: the non-injectivity of $\pi \colon \R \times \Q/\Z \to \R/\Z$ causes the sampling map $\Sample \colon \Schwartz(\A_\Z) \to \Schwartz(\Z)$ to also be non-injective. Indeed, if $(\xi_1,\alpha_1), (\xi_2,\alpha_2)$ are distinct elements of $\R \times \Q/\Z$ are such that $\pi(\xi_1,\alpha_1) = \pi(\xi_2,\alpha_2)$, then for any non-zero $F \in \Schwartz(\A_\Z)$, the functions $F_1(x,y) \coloneqq e(x \xi_1 + y \cdot \alpha_1) F(x,y)$ and $F_2(x,y) \coloneqq e(x \xi_2 + y \cdot \alpha_2) F(x,y)$ are distinct elements of $\Schwartz(\A_\Z)$ which are ``aliased'' in the sense that $\Sample F_1 = \Sample F_2$.  However, we can avoid this problem by restricting attention to a compact subset $\Omega$ of adelic frequency space $\R \times \Q/\Z$ which is \emph{non-aliasing} in the sense that the addition homomorphism $\pi$ is injective on $\Omega$, so that $\Proj$ becomes an algebra homomorphism from $\Schwartz(\Omega)$ to $\Schwartz(\pi(\Omega))$, thus
\begin{equation}\label{proj-hom} \Proj(FG) = \Proj(F) \Proj(G)
\end{equation}
for all $F,G \in \Schwartz(\Omega)$, and one has the commutative diagram
\begin{equation}\label{comm-diag}
    \begin{tikzcd}
    \Schwartz(\Z)^{\pi(\Omega)} \arrow[d, "\F_\Z"] & \Schwartz(\A_\Z)^\Omega \arrow[l, "\Sample"] \arrow[d, "\F_{\A_\Z}"] \\
    \Schwartz(\pi(\Omega)) & \Schwartz(\Omega) \arrow[l,"\Proj"]
\end{tikzcd}.
\end{equation}
In this case one verifies that the lower three maps $\F_\Z, \Proj, \F_{\A_\Z}$ are invertible, hence the upper map $\Sample$ is also. 
In particular to any non-aliasing compact set of adelic frequencies $\Omega$ we can associate an \emph{interpolation operator} $\Sample_\Omega^{-1} \colon \Schwartz(\Z)^{\pi(\Omega)} \to \Schwartz(\A_\Z)^\Omega$ that extends any Schwartz--Bruhat function on the integers with Fourier support in $\pi(\Omega)$ to the unique Schwartz--Bruhat extension on the adelic integers with Fourier support in $\Omega$.  Note from \eqref{comm-diag} and Plancherel's theorem that the sampling operator $\Sample$ and the interpolation operator $\Sample^{-1}_\Omega$ extend to unitary maps between $\ell^2(\Z)^{\pi(\Omega)}$ and $L^2(\A_\Z)^\Omega$ which invert each other.

The diagram \eqref{comm-diag} allows us to equate certain portions of Fourier analysis on the integers $\Z$ with corresponding portions of Fourier analysis of the adelic integers $\A_\Z$; this will be useful for clarifying Fourier analysis on major arcs ${\mathcal M}_{\leq l,\leq k}$, which in this perspective are interpreted as projections of a certain non-aliasing Cartesian product $\R_{\leq l} \times (\Q/\Z)_{\leq k}$ of adelic frequency space; see Figure \ref{fig:major} and Section \ref{iw-decomp-sec} for definitions.

\begin{example}\label{sampling-ex}  If $Q \in \Z_+$ and $r>0$, then $[-r,r] \times \frac{1}{Q} \Z/\Z$ is non-aliasing if and only if $r < \frac{1}{2Q}$. The injectivity of $\Sample$ in this case is a variant of the classical Shannon sampling theorem.  See also Theorem \ref{Sampling} below.
\end{example}

Now we define Fourier multiplier operators.  A continuous function $\varphi \colon \G^* \to \C$ is said to be \emph{smooth tempered} if $\varphi F \in \Schwartz(\G^*)$ whenever $F \in \Schwartz(\G^*)$.  For instance, $\varphi \colon \R \to \C$ is smooth tempered if and only if all derivatives exist and are of at most polynomial growth.

\begin{definition}[Fourier multiplier operators]\label{fourier-mult}  Let $\G$ be one of the LCA groups in Figure \ref{fig:phys}.
\begin{itemize}
    \item [(i)]  If $\varphi \colon \G^* \to \C$ is a smooth tempered function, we define the Fourier multiplier operator $\T_\varphi \colon \Schwartz(\G)  \to \Schwartz(\G)$ by the formula
    $$ \F_{\G} \T_\varphi = \varphi \F_{\G}$$
    or equivalently
    $$ \T_\varphi f(x) = \int_{\G^*} \varphi(\xi) \F_{\G} f(\xi) e(-x \cdot \xi)\ d\mu_{\G^*}(\xi)$$
    for $f \in \Schwartz(\G)$ and $x \in \G$. We refer to $\varphi$ as the \emph{symbol} of $\T_\varphi$.
    \item[(ii)]  If $m \colon \G^* \times \G^* \to \C$ is a smooth tempered function, we define the bilinear Fourier multiplier operator $\B_m \colon \Schwartz(\G) \times \Schwartz(\G) \to \Schwartz(\G)$ by the formula
    $$ \B_m(f,g)(x) = \int_{\G^*} \int_{\G^*} m(\xi_1,\xi_2) \F_{\G} f(\xi_1) \F_{\G} g(\xi_2) e(-x \cdot (\xi_1+\xi_2))\ d\mu_{\G^*}(\xi_1) d\mu_{\G^*}(\xi_2).$$
    We refer to $m$ as the \emph{symbol} of $\B_m$.
\end{itemize}
\end{definition}

Clearly $\T_\varphi$ depends linearly on $\varphi$, and $\B_m$ depends linearly on $m$.  We also observe the functional calculus identities
\begin{equation}\label{func-calc}
\begin{split}
    \T_1 f &= f, \\
    \B_1(f,g) &= fg, \\
    \T_{\varphi_1 \varphi_2} f &= \T_{\varphi_1} \T_{\varphi_2} f, \\
    \B_{m (\varphi_1 \otimes \varphi_2)}(f,g) &= \B_m( \T_{\varphi_1} f, \T_{\varphi_2} g)
\end{split}
\end{equation}
whenever $f,g \in \Schwartz(\G)$ and $\varphi_1,\varphi_2,m$ are smooth tempered functions on $\G^*, \G^*, \G^* \times \G^*$ respectively.  Finally we observe that $\T_\varphi$ is self-adjoint on $L^2(\G)$ when $\varphi$ is real-valued.  We can also extend the linear Fourier multipliers $\T_\varphi$ to Schwartz--Bruhat functions $\Schwartz(\G;V)$ taking values in a finite-dimensional vector space $V$ in the obvious fashion.

\begin{example}[Averaging operators as Fourier multipliers]\label{avg-mult}  We work on the integer shift system.  If $P \in \Z[\nn]$, the averaging operator $A_N^{P(\nn)}$ is a linear Fourier multiplier operator on $\Schwartz(\Z)$ with symbol
$$ \varphi_{N,\Z}(\xi) \coloneqq \E_{n \in [N]} e( P(n) \xi )$$
for $\xi \in \TT$.
Similarly, if $P_1,P_2 \in \Z[\nn]$, then the averaging operator $A_N^{P_1(\nn), P_2(\nn)}$ is a bilinear Fourier multiplier operator on $\Schwartz(\Z)$ with symbol
$$ m_{N,\Z}(\xi_1,\xi_2) \coloneqq \E_{n \in [N]} e( P_1(n) \xi_1 + P_2(n) \xi_2 )$$
and $\tilde A_N^{P_1(\nn), P_2(\nn)}$ similarly has symbol
$$ \tilde m_{N,\Z}(\xi_1,\xi_2) \coloneqq \E_{n \in [N]} e( P_1(n) \xi_1 + P_2(n) \xi_2 ) \ind{n>N/2}$$
for $\xi_1,\xi_2 \in \TT$. If $\G$ is one of the compact rings $\Z/p^j\Z$, $\Z/Q\Z$, $\Z_p$, or $\hat \Z$, then $P_1,P_2$ can be thought of as continuous maps from $\G$ to itself, and we can define the averaging operator $A_{\G} = A_{\G}^{P_1(\nn),P_2(\nn)} \colon \Schwartz(\G) \times \Schwartz(\G) \to \Schwartz(\G)$ by the formula
$$ A_{\G}(f,g)(x) \coloneqq \int_{\G} f(x-P_1(y)) g(x-P_2(y))\ d\mu_{\G}(y).$$
From the Fourier inversion formula and the Fubini--Tonelli theorem we see that $A_{\G}$ is a bilinear Fourier multiplier operator with symbol
$$ m_\G(\xi_1,\xi_2) \coloneqq \int_{\G} e( P_1(y) \xi_1 + P_2(y) \xi_2)\ d\mu_{\G}(y)$$
for $\xi_1,\xi_2 \in \G^*$.
\end{example}

\begin{example}[Tensor products of multipliers]\label{tensor-mult} Let $\G_1,\G_2$ be LCA groups from Figure \ref{fig:phys}.  If $\T_{\varphi_1}$ is a linear Fourier multiplier operator on $\Schwartz(\G_1)$ and $\T_{\varphi_2}$ is a linear Fourier multiplier operator on $\Schwartz(\G_2)$, then $\T_{\varphi_1 \otimes \varphi_2}$ is a linear Fourier multiplier operator on $\Schwartz(\G_1 \times \G_2)$ which is the tensor product of $\T_{\varphi_1}$ and $\T_{\varphi_2}$ in the sense that \eqref{tensor-1} holds for all $f_1 \in \Schwartz(\G_1), f_2 \in \Schwartz(\G_2)$.  Similarly, if $\B_{m_1}, \B_{m_2}$ are bilinear Fourier multiplier operators on $\Schwartz(\G_1), \Schwartz(\G_2)$ respectively then the bilinear Fourier multiplier operator $\B_{m_1 \otimes m_2}$ is the tensor product of $\B_{m_1}$ and $\B_{m_2}$ in the sense that \eqref{bil-tensor} holds for all $f_1,g_1 \in \Schwartz(\G_1)$, $f_2,g_2 \in \Schwartz(\G_2)$.
\end{example}

As previously mentioned, if $\Omega$ is a non-aliasing subset of $\R \times \Q/\Z$, then the sampling operator $\Sample$ restricts to a unitary map from $L^2(\A_\Z)^\Omega$ to $\ell^2(\Z)^{\pi(\Omega)}$, or equivalently the interpolation operator $\Sample_\Omega^{-1}$ is a unitary map from $\ell^2(\Z)^{\pi(\Omega)}$ to $L^2(\A_\Z)^\Omega$.  This suggests that Fourier multiplier operators on $L^2(\A_\Z)^\Omega$ can be identified with Fourier multiplier operators on $\ell^2(\Z)^{\pi(\Omega)}$.  This is indeed the case:

\begin{lemma}[Adelic and integer Fourier multipliers]\label{adel}  Let $\Omega \subset \R \times \Q/\Z$ be a non-aliasing compact subset of adelic frequency space.  Then for any $\varphi \in \Schwartz(\Omega)$, the diagram
\begin{center}
    \begin{tikzcd}
    \Schwartz(\Z)^{\pi(\Omega)} \arrow[ddd,"\F_\Z"] &&& \Schwartz(\A_\Z)^\Omega \arrow[lll,"\Sample"] \arrow[ddd,"\F_{\A_\Z}"] \\
    & \Schwartz(\Z)^{\pi(\Omega)} \arrow[ul, "\T_{\Proj \varphi}"] \arrow[d, "\F_\Z"] & \Schwartz(\A_\Z)^\Omega \arrow[l,"\Sample"] \arrow[d,"\F_{\A_\Z}"] \arrow[ur, "\T_\varphi"'] \\
    & \Schwartz(\pi(\Omega)) \arrow[dl, "\Proj\varphi"'] & \Schwartz(\Omega) \arrow[l,"\Proj"] \arrow[dr,"\varphi"] \\
    \Schwartz(\pi(\Omega)) &&& S(\Omega) \arrow[lll,"\Proj"]
\end{tikzcd}
\end{center}
commutes, where $\varphi$ denotes the operation of pointwise multiplication by $\varphi$, and similarly for $\Proj \varphi$. In particular, one has
\begin{equation}\label{pj}
\T_{\Proj \varphi} \Sample f = \Sample \T_\varphi f
\end{equation}
for all $f \in \Schwartz(\A_\Z)^\Omega$.
\end{lemma}

\begin{proof} This is immediate from \eqref{comm-diag}, \eqref{proj-hom}, Definition \ref{fourier-mult}, and a routine diagram chase using the invertibility of the Fourier transform.
\end{proof}

Another way of writing \eqref{pj} is as
\begin{equation}\label{pj-alt}
\T_{\Proj \varphi} f = \Sample \T_\varphi \Sample_\Omega^{-1} f
\end{equation}
for all $f \in \Schwartz(\Z)^{\pi(\Omega)}$.

There is a bilinear version of the formula \eqref{pj}.  Define the tensor square $\Proj^{\otimes 2} \colon \Schwartz( (\R \times \Q/\Z)^2 ) \to \Schwartz(\TT^2)$ of the projection operator $\Proj$ by the formula
$$ \Proj^{\otimes 2} m(\xi_1,\xi_2) \coloneqq \sum_{(\theta_1,\alpha_1) \in \pi^{-1}(\xi_1)}
\sum_{(\theta_2,\alpha_2) \in \pi^{-1}(\xi_2)} m((\theta_1,\alpha_1),(\theta_2,\alpha_2))
$$
for all $m \in \Schwartz( (\R \times \Q/\Z)^2 )$.
If $\Omega_1, \Omega_2 \subset \R \times \Q/\Z$ are non-aliasing compact subsets of adelic frequency space, then $\Proj^{\otimes 2}$ is an algebra homomorphism from $\Schwartz(\Omega_1 \times \Omega_2)$ to $\Schwartz(\pi(\Omega_1) \times \pi(\Omega_2))$, and is the tensor product of the algebra homomorphisms $\Proj \colon \Schwartz(\Omega_1) \to \Schwartz(\pi(\Omega_1))$ and $\Proj \colon \Schwartz(\Omega_2) \to \Schwartz(\pi(\Omega_2))$ in the sense of \eqref{tensor-1}.  A routine calculation (or a chase of a more complicated version of the commutative diagram in Lemma \ref{adel}) then gives the bilinear variant
\begin{equation}\label{pj-bil}
\B_{\Proj^{\otimes 2} m}( \Sample f, \Sample g ) = \Sample \B_m(f,g)
\end{equation}
of \eqref{pj} whenever $f \in\Schwartz(\A_\Z)^{\Omega_1}$, $g \in \Schwartz(\A_\Z)^{\Omega_2}$, and $m \in \Schwartz(\Omega_1 \times \Omega_2)$; equivalently, one has
\begin{equation}\label{pj-bil-alt}
\B_{\Proj^{\otimes 2} m}( f, g ) = \Sample \B_m(\Sample_{\Omega_1}^{-1} f,\Sample_{\Omega_2}^{-1} g)
\end{equation}
whenever $f \in \Schwartz(\Z)^{\pi(\Omega_1)}$, $g \in \Schwartz(\Z)^{\pi(\Omega_2)}$.  From \eqref{func-calc} we also observe the projected functional calculus
\begin{equation}\label{func-calc-2}
\B_{\Proj^{\otimes 2} m}( \T_{\Proj \varphi_1} f, \T_{\Proj \varphi_2} g ) = \B_{\Proj^{\otimes 2} (m (\varphi_1 \otimes \varphi_2))}(f,g)
\end{equation}
whenever $f \in \Schwartz(\Z)^{\pi(\Omega_1)}$, $g \in \Schwartz(\Z)^{\pi(\Omega_2)}$, $\varphi_1 \in \Schwartz(\Omega_1)$, $\varphi_2 \in \Schwartz(\Omega_2)$, and $m \in \Schwartz(\Omega_1 \times \Omega_2)$.

The point of the identities \eqref{pj-alt}, \eqref{pj-bil-alt} is that complicated linear and bilinear Fourier multiplier operators $\T_{\Proj \varphi}, \B_{\Proj^{\otimes 2} m}$ on the integers $\Z$ can be expressed (in non-aliasing regions of adelic frequency space) by simpler linear and bilinear Fourier multiplier operators $\T_\varphi, \B_m$ on the adelic integers $\A_\Z$.  For the multiplier operators of interest in this paper, the adelic symbols $\varphi, m$ often have a tensor product structure (or at least can be decomposed or approximated by symbols with such a structure), allowing us to decouple the Fourier analysis into the continuous Fourier analysis of $\R$ and the arithmetic Fourier analysis of $\hat \Z$.  In many cases the arithmetic symbol factors further, allowing one to work on smaller factor groups such as $\Z/Q\Z$, $\Z_p$, or $\Z/p^j \Z$.

As already observed, whenever $\Omega$ is a non-aliasing compact subset of $\R \times \Q/\Z$, the sampling operator $\Sample \colon \Schwartz(\A_\Z)^\Omega \to \Schwartz(\Z)^{\pi(\Omega)}$ and the interpolation operator $\Sample_\Omega^{-1} \colon \Schwartz(\Z)^{\pi(\Omega)} \to \Schwartz(\A_\Z)^\Omega$ both preserve the $L^2$ norm.  The situation for other function space norms is less clear.  However the situation is particularly favorable in the case of Example \ref{sampling-ex}, in that the sampling and interpolation operators essentially preserve all $L^p$ norms, even for non-Banach exponents $0 < p < 1$ or for vector-valued functions (or both):

\begin{theorem}[Quantitative Shannon sampling theorem]\label{Sampling}  Let $0 < p \leq \infty$, and let $B$ be a finite-dimensional normed vector space.  If $F \in \Schwartz(\A_\Z; B)$ has Fourier support in $[-\frac{c_0}{Q},\frac{c_0}{Q}] \times \frac{1}{Q}\Z/\Z$ for some $Q \in \Z_+$ and some $0 < c_0 < \frac{1}{2}$, then
\begin{equation}\label{e:SAMP} \| \Sample F \|_{\ell^p(\Z;B)} \sim_{c_0,p} \|F\|_{L^p(\A_\Z; B)}
\end{equation}
where we extend the sampling operator $\Sample$ to vector-valued functions in the obvious fashion.
\end{theorem}

See also the sampling principle of Magyar--Stein--Wainger \cite[Corollary 2.1, pp. 196]{MSW} as well as \cite[Proposition 4.4, pp. 816]{MSZ1} for closely related statements.  Theorem \ref{Sampling} implies that if $\Omega$ is a compact subset of $[-\frac{c_0}{Q},\frac{c_0}{Q}] \times \frac{1}{Q}\Z/\Z$, then $\Sample \colon \Schwartz(\A_\Z)^\Omega \to \Schwartz(\Z)^{\pi(\Omega)}$ and $\Sample_\Omega^{-1} \colon \Schwartz(\Z)^{\pi(\Omega)} \to \Schwartz(\A_\Z)^\Omega$ are both bounded on $L^p$ with norm $O_{c_0}(1)$.

\begin{proof}  As $F$ has Fourier support on the Pontryagin dual $\R \times \frac{1}{Q}\Z/\Z$ of $\R \times \Z/Q\Z$, we can descend to the quotient group $\R \times \Z/Q\Z$ and establish the bound
$$ \| \Sample_{Q} F \|_{\ell^p(\Z;B)} \sim_{c_0} \|F\|_{L^p(\R \times \Z/Q\Z; B)}$$
whenever $F \in \Schwartz(\R \times \Z/Q\Z;B)$ has Fourier support in $[-\frac{c_0}{Q},\frac{c_0}{Q}] \times \frac{1}{Q}\Z/\Z$ and
$$ \Sample_{Q} F(x) \coloneqq F(x, x \mod Q ).$$
By splitting $\Z$ into residue classes $a+Q\Z$ for $a \in [Q]$, and similarly splitting $\R \times \Z/Q\Z$ into copies $\R \times \{a \mod Q\}$ of $Q$, it suffices by the Fubini--Tonelli theorem to establish the bound
$$ \| f \|_{\ell^p(a+Q\Z;B)} \sim_{c_0,p} Q^{-1/p} \|f\|_{L^p(\R;B)}$$
whenever $a \in [Q]$ and $f \in \Schwartz(\R;B)$ has Fourier support in $[-\frac{c_0}{Q}, \frac{c_0}{Q}]$.  After applying translation and rescaling, it suffices to show that
$$ \| f \|_{\ell^p(\Z;B)} \sim_{c_0,p} \|f\|_{L^p(\R;B)}$$
whenever $f \in \Schwartz(\R;B)$ has Fourier support in $[-c_0,c_0]$.  It will suffice to establish the bound
$$ \| f \|_{\ell^p(\Z+\theta;B)} \sim_{c_0,p} \|f\|_{\ell^p(\Z;B)}$$
uniformly for all $0 \leq \theta \leq 1$, as the claim then follows by taking $L^p$ norms in $\theta$ and applying the Fubini--Tonelli theorem.  By translation and reflection symmetry it suffices to establish the upper bound
\begin{equation}\label{upper-lp}
 \| f \|_{\ell^p(\Z+\theta;B)} \lesssim_{c_0,p} \|f\|_{L^p(\Z;B)}.
 \end{equation}

Let $\psi = \psi_{c_0} \in \Schwartz(\R)$ be a function chosen so that $\F_\R \psi$ is supported on $[-1/2,1/2]$ and equals one on $[-c_0,c_0]$, so that the upper bound now follows from Schur's test.  From the Poisson summation formula we have
$$ f(y) = \sum_{x \in \Z} \psi(y-x) f(x)$$
for all $y \in \R$, hence by the triangle inequality
$$ \|f(y)\|_B \leq \sum_{x \in \Z} |\psi(y-x)| \|f(x)\|_B.$$
For $p \geq 1$ this gives \eqref{upper-lp} from Schur's test and the rapid decrease of $\psi$.  For $p<1$ we use the previous inequality to obtain
$$ \|f(y)\|_B^p \leq \sum_{x \in \Z} |\psi(y-x)|^p \|f(x)\|_B^p$$
and the claim follows from the triangle inequality and the rapid decrease of $\psi$.
\end{proof}

Because of this theorem and \eqref{pj}, \eqref{pj-bil}, the $L^p$ multiplier theory for both linear and bilinear Fourier multiplier operators $\T_{\Proj \varphi}$, $\B_{\Proj^{\otimes 2} m}$ on $\Schwartz(\Z)^{\pi(\Omega)}$ can be easily transferred to the corresponding $L^p$ multiplier theory of $\T_\varphi, \B_m$ on $\Schwartz(\A_\Z)^{\Omega}$ when $\Omega$ is of the form in Example \ref{sampling-ex} (or a compact subset of that example).
Unfortunately this situation only occurs for us in certain ``large-scale'' settings, in which the widths of the major arcs are extremely narrow compared to the height.  In the opposite ``small-scale'' regime we will be able to use the Ionescu--Wainger multiplier theorem (see Lemma \ref{mag-lem}(iv) and Remark \ref{iw-samp} below) as a partial replacement\footnote{Another partial replacement of Theorem \ref{Sampling} in this setting was recently established in \cite[Theorem 1.6]{T}.} of this transference, at least at the level of linear Fourier multiplier operators.  The Ionescu--Wainger theory does not directly treat the ``twisted'' bilinear multipliers $\B^{l_1,l_2,m_{\hat \Z}}_m$ that we will eventually need to handle (see \eqref{bil-twist}), so we will need to first apply a two-parameter Rademacher--Menshov argument in order to reduce the bilinear analysis to linear estimates that can be treated by that theory; see Section \ref{small-sec}.

We close this section with some crude multiplier estimates on $\Z$ and on $\R$.

\begin{lemma}[Crude multiplier bound]\label{crude-mult} Let $\G = \Z$ or $\G=\R$.
\begin{itemize}
    \item[(i)] Let $\varphi \in \Schwartz(\G^*)$ and $r>0$.  When $\G=\Z$ we also require $r \leq 1$.  Then for any $1 \leq p \leq \infty$, $\T_\varphi$ extends continuously to a linear map from $L^p(\G)$ to $L^p(\G)$ with
    \begin{equation}\label{mult-1}
    \|\T_\varphi \|_{L^p(\G) \to L^p(\G)}  \lesssim \sup_{0 \leq j \leq 2} \int_{\G^*} r^{j-1} \left|\frac{d^j}{d \xi^j} \varphi(\xi)\right|\ d\xi.
    \end{equation}
    \item[(ii)] Let $m \in \Schwartz(\G^* \times \G^*)$, $r_1,r_2 > 0$, and $1 \leq p,p_1,p_2 \leq \infty$ with $\frac{1}{p_1}+\frac{1}{p_2} = \frac{1}{p}$.  When $\G=\Z$ we also require $r_1,r_2 \leq 1$.  Then $\B_m$ extends continuously to a bilinear map from $L^{p_1}(\G) \times L^{p_2}(\G)$ to $L^p(\G)$ with
    \begin{equation}\label{mult-2}
     \|\B_m\|_{L^{p_1}(\G) \times L^{p_2}(\G) \to L^p(\G)} \lesssim \sup_{0 \leq j_1,j_2 \leq 2} \int_{\G^*} \int_{\G^*} r_1^{j_1-1} r_2^{j_2-1} \left|\frac{\partial^{j_1}}{\partial \xi_1^{j_1}}\frac{\partial^{j_2}}{\partial \xi_1^{j_2}} m(\xi_1,\xi_2)\right|\ d\xi_1 d\xi_2.
    \end{equation}
    The same bound also holds when the hypothesis $1 \leq p,p_1,p_2 \leq \infty$ is replaced by $1 <p_1,p_2 \leq \infty$, except now the implied constant in \eqref{mult-2} is permitted to depend on $p_1,p_2$.
\end{itemize}
\end{lemma}
\begin{proof}  We just prove (ii) in the case $\G=\Z$, as all the other cases are similar. It suffices to prove the claim for Schwartz functions.  We may normalize the right-hand side of \eqref{mult-2} to be $1$. We can express $\B_m$ in physical space as 
$$ \B_m(f,g)(x) = \sum_{y_1,y_2 \in \Z} K(y_1,y_2) f(x-y_1) g(x-y_2), $$
where
$$ K(y_1,y_2) \coloneqq \int_{\TT^2} m(\xi_1,\xi_2) e(-y_1\xi_1 - y_2 \xi_2)\ d\xi_1 \xi_2.$$
Suppose first that we are in the case $1 \leq p,p_1,p_2 \leq \infty$.
By Minkowski's inequality we have
$$ \|\B_m \|_{\ell^{p_1}(\Z) \times \ell^{p_2}(\Z) \to \ell^p(\Z)} \leq \| K \|_{\ell^1(\Z^2)}.$$
On the other hand, from the normalization of \eqref{mult-2} and integration by parts we have
$$ K(y_1,y_2) \lesssim r_1^{1-j_1} r_2^{1-j_2} |y_1|^{-j_1} |y_2|^{-j_2}$$
for any $y_1,y_2 \in \Z$ and $0 \leq j_1,j_2 \leq 2$ (with the claim being vacuously true if the right-hand side is infinite), thus
\begin{equation}\label{kbound}
K(y_1,y_2) \lesssim r_1 \langle r_1 y_1 \rangle^{-2} r_2 \langle r_2 y_2 \rangle^{-2}
\end{equation}
and the claim follows.  In the case $1 < p_1,p_2 \leq \infty$, we can instead use \eqref{kbound} to bound $\B_m(f,g)$ pointwise by the product of the Hardy--Littlewood maximal functions of $f,g$, and the claim now follows from H\"older's inequality and the Hardy--Littlewood maximal inequality.
\end{proof}

\section{Ionescu--Wainger decomposition: reducing to major arcs}\label{iw-decomp-sec}

We now begin the proof of Theorem \ref{end-var}.  Henceforth the parameters $P, d, p_1, p_2, p, r, \lambda$ are fixed to obey the hypotheses of this theorem, and all implied constants in the asymptotic notation are allowed to depend on these parameters.  We also fix the finite $\lambda$-lacunary subset $\D$ of $\Z_+$, although we require all our estimates to be uniform in the choice of $\D$. We abbreviate $\tilde A^{\nn,P(\nn)}_{N}$ as $\tilde A_N$.  

We will also need four large constants:

\begin{itemize}
    \item[(i)] We choose a constant $C_0 \in \Z_+$ that is sufficiently large depending on the fixed parameters $P,d,p_1,p_2,p,r,\lambda$. (This constant is used to define a maximum height scale $l_{(N)}$ associated to each physical scale $N$; see \eqref{kl-def}.)
    \item[(ii)] We choose a constant $C_1 \in \Z_+$ that is sufficiently large depending on the fixed parameters $P,d,p_1,p_2,p,r,\lambda$ and on $C_0$. (This constant is used to define the Ionescu--Wainger parameter $\rho$; see \eqref{rho-def}.)
    \item[(iii)] We choose a constant $C_2 \in \Z_+$ that is sufficiently large depending on the fixed parameters $P,d,p_1,p_2,p,r,\lambda$ and on $C_0, C_1$.  (This quantity is used to define an auxiliary scale $u$ associated to a given height scale $l$; see \eqref{u-def}.)
    \item[(iv)] We choose a constant $C_3 \in \Z_+$ that is sufficiently large depending on the fixed parameters $P,d,p_1,p_2,p,r,\lambda$ and on $C_0, C_1, C_2$. (This quantity will be used to lower bound the physical scale $N$, as well as to bound implied constants in estimates.)
\end{itemize}

We also use $c>0$ to denote various small exponents that depend only on $d,p_1,p_2,p,r$, and which will vary from line to line.  Occasionally we will also need $c$ to depend on some other parameters and we will indicate this by additional subscripts, for instance $c_q$ will be a positive constant depending on $d,p_1,p_2,p,r,q$.  Importantly, these constants $c$ will \emph{not} depend on the large constants $C_0,C_1,C_2,C_3$ just introduced. Specifically, $c$ will be independent on the Ionescu--Wainger parameter $\rho$, see \eqref{rho-def}.

Define the \emph{naive height} $\Height_\naive(\alpha) \in 2^\N$ of an arithmetic frequency $\alpha = \frac{a}{q} \mod 1 \in \Q/\Z$ by the formula
$$
\Height_\naive\left( \frac{a}{q} \mod 1 \right) \coloneqq \inf \{ 2^l: l \in \N, q \leq 2^l \} = 2^{\lceil \log q \rceil} \sim q
$$
whenever $q \in \Z_+$ and $a \in [q]^\times$.  For any $l \in \N$, $k \in \Z$, we can then define the naive arithmetic frequency sets
$$ (\Q/\Z)_{\leq l,\naive} \coloneqq \Height_\naive^{-1}([2^l]) = \{\alpha \in \Q/\Z \colon \Height_\naive(\alpha) \leq 2^l\}$$
and the continuous frequency sets
$$ \R_{\leq k} \coloneqq [-2^k,2^k]$$
and then define the naive major arcs
$$ {\mathcal M}_{\leq l, \leq k, \naive} \coloneqq \pi( \R_{\leq k} \times (\Q/\Z)_{\leq l, \naive} ),$$
thus ${\mathcal M}_{\leq l, \leq k, \naive}$ consists of all elements of $\TT$ of the form $\frac{a}{q} + \theta \mod 1$ for some $q \in [2^l]$, $a \in [q]^\times$, and $\theta \in [-2^{-k},2^k]$.  These would be the obvious choice of major arcs to restrict attention to in our Fourier-analytic manipulations.  Unfortunately, the $L^p$ multiplier theory on such arcs is unfavorable.  To obtain a better theory, we follow Ionescu and Wainger \cite{IW} and replace the naive height $\Height_\naive(\alpha)$ of an arithmetic frequency by a smaller quantity, which we call the \emph{Ionescu--Wainger height} $\Height(\alpha) = \Height_\rho(\alpha) \in 2^\N$.  This height depends on an additional small parameter $0 < \rho < 1$, which we now fix in our hierarchy of constants as
\begin{equation}\label{rho-def}
\rho \coloneqq 1/C_1.
\end{equation}
The precise definition of this height is technical and is postponed to Appendix \ref{iw-app}.  However, for our purposes we can summarize the main properties of this height as follows.  Using this height, we define the Ionescu--Wainger arithmetic frequency sets
$$ (\Q/\Z)_{\leq l} \coloneqq \Height^{-1}([2^l]) = \{\alpha \in \Q/\Z \colon \Height(\alpha) \leq 2^l\}$$
and the Ionescu--Wainger major arcs or simply major arcs
$$  {\mathcal M}_{\leq l, \leq k} \coloneqq \pi( \R_{\leq k} \times (\Q/\Z)_{\leq l} );$$
see Figure \ref{fig:major}.  These arcs will be somewhat larger than their naive counterparts, but this is more than compensated for by their superior Fourier multiplier theory. We also use the variants
$$ (\Q/\Z)_l \coloneqq (\Q/\Z)_{\leq l} \backslash (\Q/\Z)_{\leq l-1} =  \Height^{-1}(2^l) = \{ \alpha \in \Q/\Z \colon \Height(\alpha)=2^l\}$$
and
$$  {\mathcal M}_{l, \leq k} \coloneqq \pi( \R_{\leq k} \times (\Q/\Z)_{l} )$$
with the convention that $(\Q/\Z)_{\leq -1}$ is empty.

\begin{figure}
    \centering
    \begin{tikzcd}
    & {\mathcal M}_{\leq l, \leq k} \\ 
    \R_{\leq k} \arrow[hook,r] \arrow[ur] \arrow[d,equal] & \R_{\leq k} \times (\Q/\Z)_{\leq l} \arrow[u, two heads,"\pi"]  \arrow[d,hook] & (\Q/\Z)_{\leq l} \arrow[d,hook] \arrow[l,hook] \arrow[ul,hook] & (\Z[\frac{1}{p}]/\Z)_{\leq l} \arrow[l,hook] \arrow[d,hook] \arrow[ull, hook] \\
    \R_{\leq k} \arrow[hook,r] & \R_{\leq k} \times \frac{1}{Q_{\leq l}}\Z/\Z  & \frac{1}{Q_{\leq l}}\Z/\Z \arrow[l,hook] & \arrow[l,hook] \frac{1}{p^j}\Z/\Z
    \end{tikzcd}
    \caption{The commutative diagram in Figure \ref{fig:freq}, restricted to major arcs, where $p^j$ is the largest power of $p$ dividing $Q_{\leq l}$ and $(\Z[\frac{1}{p}]/\Z)_{\leq l} \coloneqq (\Q/\Z)_{\leq l} \cap \Z[\frac{1}{p}]/\Z$.  When $(l,k)$ has good major arcs, the product set $\R_{\leq k} \times (\Q/\Z)_{\leq l}$ is non-aliasing, and the indicated map $\pi$ can be upgraded from a surjection to a bijection.  Most of the spaces in this diagram are no longer groups and so the arrows are now downgraded from continuous homomorphisms to continuous maps.  Note the approximate duality with Figure \ref{fig:unc}.}
    \label{fig:major}
\end{figure}
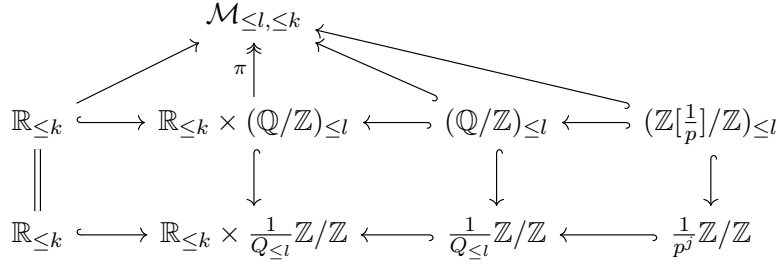

\begin{figure}
    \centering
    \begin{tikzcd}
    & \ell^2(\Z)^{{\mathcal M}_{\leq l, \leq k}} \arrow[d,"\Sample^{-1}_{\R_{\leq k} \times (\Q/\Z)_{\leq l}}",hook', two heads, shift left = 1.5ex] \\
    L^2(\R)^{\R_{\leq k}} \arrow[dotted,r] \arrow[ur,hook] \arrow[d,equal] & L^2(\A_\Z)^{\R_{\leq k} \times (\Q/\Z)_{\leq l}} \arrow[u, hook',two heads,"\Sample"]  \arrow[d,hook] & L^2(\hat \Z)^{(\Q/\Z)_{\leq l}} \arrow[d,hook] \arrow[l,dotted]  & L^2(\Z_p)^{(\Z[\frac{1}{p}]/\Z)_{\leq l}} \arrow[l,hook] \arrow[d,hook] \\
    L^2(\R)^{\R_{\leq k}} \arrow[dotted,r] & L^2(\R \times \Z/Q_{\leq l}\Z)^{\R_{\leq k} \times \frac{1}{Q_{\leq l}}\Z/\Z}  & L^2(\Z/Q_{\leq l}\Z) \arrow[l,dotted] & \arrow[l,hook] L^2(\Z/p^j\Z)
    \end{tikzcd}
    \caption{The $L^2$ version of Figure \ref{fig:major}, under the hypothesis of good major arcs.  Solid (hooked) arrows are Hilbert space isometries, double-headed arrows are unitary maps, and dotted arrows indicate a (Hilbert space) tensor product. We thus see that the major arc component $\ell^2(\Z)^{{\mathcal M}_{\leq l,\leq k}}$ of $\ell^2(\Z)$ can be identified with the tensor product of the low (continuous) frequency component $L^2(\R)^{\R_{\leq k}}$ of $L^2(\R)$ and the low (arithmetic) frequency component $L^2(\hat \Z)^{(\Q/\Z)_{\leq l}}$ of $L^2(\hat \Z)$, with the latter component identifiable in turn with a subspace of $L^2(\Z/Q_{\leq l}\Z)$. As with Figure \ref{fig:schwartz-phys}, some arrows are missing due to the failure of $L^2(\R)^{\R_{\leq k}}$ to contain a unit $1$.}
    \label{fig:major-l2}
\end{figure}
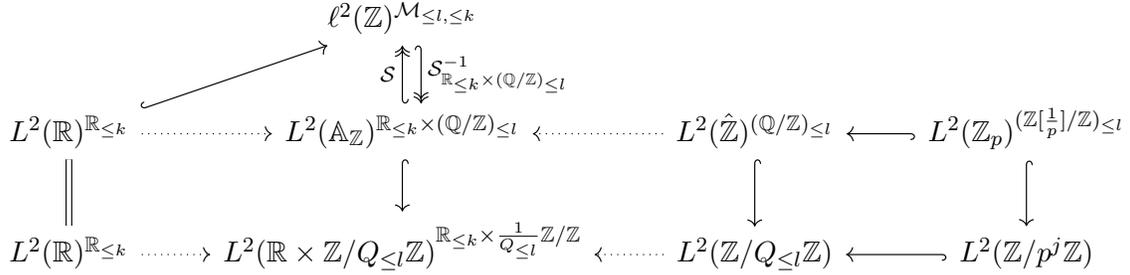

\begin{lemma}[Properties of height]\label{mag-lem} \ 
\begin{itemize}
    \item [(i)] (Naive height controls height) For any $\alpha \in \Q/\Z$, one has
    \begin{equation}\label{height-naive}
    \Height(\alpha) \leq \Height_\naive(\alpha).
    \end{equation}
    In particular, $(\Q/\Z)_{\leq l,\naive} \subset (\Q/\Z)_{\leq l}$ and ${\mathcal M}_{\leq l,\leq k, \naive} \subset {\mathcal M}_{\leq l, \leq k}$ for any $(l,k) \in \N \times \Z$.  If $\alpha \in \frac{1}{p}\Z/\Z$ for a prime $p$, then equality holds in \eqref{height-naive}.
    \item[(ii)]  (Cyclic structure)  For any $l \in \N$, $(\Q/\Z)_{\leq l}$ is the union of finitely many dual cyclic groups $\frac{1}{q}\Z/\Z$ with
    $$ q \lesssim_\rho 2^{2^{\rho l}}$$
    and is contained in a single dual cyclic group $\frac{1}{Q_{\leq l}}\Z/\Z$ with
    $$ Q_{\leq l} \lesssim 2^{O_\rho(2^l)}.$$
    In fact, the integer $Q_{\leq l}\in\Z_+$ can be defined explicitly as in \eqref{def:Q_l}.
    \item[(iii)]  (Cardinality bound)  
For any $l \in \N$, one has
    $$ \# (\Q/\Z)_{\leq l} \lesssim_\rho 2^{2^{\rho l}}.$$
\end{itemize}
\end{lemma}

\begin{proof} See Appendix \ref{iw-app}.
\end{proof}

The linear Fourier multiplier operators $\T^{\leq l}_\varphi$ and $\T^{l}_\varphi$ defined by
\begin{align}
\label{tlphi-def}
\T^{\leq l}_\varphi &\coloneqq \T_{\Proj(\varphi \otimes \ind{(\Q/\Z)_{\leq l}})},\\
\label{tlphi-def-2}
\T^{l}_\varphi &\coloneqq \T_{\Proj(\varphi \otimes \ind{(\Q/\Z)_{l}})},
\end{align}
will play a key role in our analysis.  They can be written more explicitly as
\begin{align*}
\T^{\leq l}_\varphi f(x) &= \sum_{\alpha \in (\Q/\Z)_{\leq l}} \int_\R \varphi(\theta) \F_\Z f(\alpha+\theta) e(-x(\alpha+\theta))\ d\theta,\\
\T^{l}_\varphi f(x) &= \sum_{\alpha \in (\Q/\Z)_{l}} \int_\R \varphi(\theta) \F_\Z f(\alpha+\theta) e(-x(\alpha+\theta))\ d\theta.
\end{align*}
From \eqref{proj-hom}, \eqref{func-calc} one has the functional calculus
\begin{equation}\label{func-calc-l}
\T^{\leq l}_{\varphi_1 \varphi_2} = \T^{\leq l}_{\varphi_1} \circ T^{\leq l}_{\varphi_2}
\end{equation}
whenever $(l,k)$ has good major arcs and $\varphi_1, \varphi_2 \in \Schwartz(\R_{\leq k})$.  Similarly with $\leq l$ replaced by $l$ in \eqref{func-calc-l}.
The principal tool in bounding operators \eqref{tlphi-def} and \eqref{tlphi-def-2} is the Ionescu--Wainger multiplier theorem \cite{IW}, which for our purposes can be formulated as follows: 

\begin{theorem}[Vector-valued Ionescu--Wainger multiplier theorem]
\label{thm:iw}
 If $(l,k) \in \N \times \Z$ has \emph{good major arcs} in the sense that
\begin{equation}\label{m-small}
k \leq - C_\rho 2^{\rho l}
\end{equation}
for a sufficiently large constant $C_\rho$ depending only on $\rho$, then the compact set $\R_{\leq k} \times (\Q/\Z)_{\leq l} \subset \R \times \Q/\Z$ is non-aliasing.  Furthermore, if $q \in 2\N \cup (2\N)'$ is either an even integer or the dual of an even integer, then the linear Fourier multiplier operator $\T^{\leq l}_\varphi$ from \eqref{tlphi-def} obeys the multiplier bound
\begin{equation}\label{iw-mult}
\| \T^{\leq l}_\varphi \|_{\ell^q(\Z; H) \to \ell^q(\Z;H)} \lesssim_{\rho,q} \langle l \rangle \| \T_\varphi \|_{L^q(\R) \to L^q(\R)}
\end{equation}
for any $\varphi \in \Schwartz(\R_{\leq k})$, and any finite-dimensional Hilbert space $H$.  Similarly for the multiplier operator $\T^{l}_\varphi$
from \eqref{tlphi-def-2}.
\end{theorem}

\begin{proof}
 See Appendix \ref{iw-app}.
\end{proof}

\begin{remark}
\label{rem:1}
Some remarks about Theorem \ref{thm:iw} are in order.
\begin{itemize}
\item[(i)] Theorem \ref{thm:iw} in the scalar-valued setting was first established by Ionescu and Wainger
\cite{IW}  with the factor $\langle l \rangle^{\lfloor2/\rho\rfloor+1}$ in
place of $\langle l \rangle$ in \eqref{iw-mult}. Their proof is based on an
intricate inductive argument that exploits super-orthogonality
phenomena. A slightly different proof (giving the factor $\langle l \rangle$ in
\eqref{iw-mult}) using certain recursive arguments, which
clarified the role of the underlying square functions and
orthogonalities, was presented in \cite{M1}. A vector-valued Ionescu--Wainger multiplier theorem (in the spirit of
\cite{M1}) can be found in \cite[Section 2]{MSZ3}. A uniform  vector-valued Ionescu--Wainger multiplier theorem, where the factor $\langle l \rangle$ 
is removed from \eqref{iw-mult}, was recently proved by the third author \cite{T}. The latter proof provides also explicit constants in \eqref{iw-mult} and allows us to handle adelic Fourier multipliers.  The super-orthogonality phenomena are discussed in the survey of Pierce \cite{Pierce} in a much broader
context.

\item[(ii)] The fact that the losses in \eqref{iw-mult} are only polynomial in the logarithmic height scale $l$ instead of exponential will be essential to our arguments, and form the main reason why we cannot work with the naive notion of heights, as the analogous multiplier theorem is not available for such heights.

\item[(iii)]  As we are focused on variational estimates even the factors like $2^{O(\rho l)}$ will have to be handled, see the constants produced by the Rademacher--Menshov inequality in Section \ref{small-sec}. From this point of view, even though the uniform  vector-valued Ionescu--Wainger multiplier theorem \cite{T} is now available, and the factor $\langle l \rangle$ can be deleted, this does not significantly improve the main result or simplify the proof. Hence, we will use the vector-valued Ionescu--Wainger multiplier theorem from \cite[Section 2]{MSZ3}.

\item[(iv)] The restriction in Theorem \ref{thm:iw} to the case when $q$ is an even integer or the dual of an even integer can be ignored in practice because in all the applications of Theorem \ref{thm:iw} we will have good $L^q(\R)$ operator norm bounds on $\T_\varphi$ for \emph{all} $1 < q < \infty$, and then by applying \eqref{iw-mult} for $q \in 2\N \cup (2\N)'$ and then interpolating we can recover good bounds for all $1 < q < \infty$.  See also the discussion after \cite[Theorem 2.1]{MSZ3}.
\end{itemize}
\end{remark}

\begin{remark}\label{iw-samp}  When $(l,k)$ has good major arcs, the corresponding sampling operator $\Sample \colon L^2(\A_\Z)^{\R_{\leq k} \times (\Q/\Z)_{\leq l}} \to \ell^2(\Z)^{\mathcal{M}_{\leq l,\leq k}}$ is unitary thanks to \eqref{comm-diag}, and is inverted by the interpolation operator $\Sample_{\R_{\leq k} \times (\Q/\Z)_{\leq l}}^{-1}$; see Figure \ref{fig:major-l2}.  For $L^p$ norms one no longer expects to have the isometry property even at an approximate level (except in the large scale case when Theorem \ref{Sampling} applies), but \eqref{iw-mult} shows that at least the linear Fourier multiplier theory on $\ell^q(\Z)^{\mathcal{M}_{\leq l,\leq k}}$ is basically controlled (up to small losses) by that of $L^q(\A_\Z)^{\R_{\leq k} \times (\Q/\Z)_{\leq l}}$ (at least when $q \in (2\N) \cup (2\N)'$), which serves as a partial substitute for an isometry property for the sampling operator.
\end{remark}

A crucial component of our arguments is the assertion that the bilinear averaging operator $\tilde A^{\nn,P(\nn)}_N(f,g)$ is negligible when the Fourier transform of $f$ or $g$ vanishes on major arcs. More precisely, we have the following improvement of \eqref{anf} in this case.  

\begin{theorem}[Single scale minor arc estimate]\label{improv} Let $N \geq 1$, let $l \in \N$, and suppose that $f,g \in \ell^2(\Z)$ obeys one of the following assumptions:
\begin{itemize}
    \item[(i)]  $\F_\Z f$ vanishes on ${\mathcal M}_{\leq l, \leq - \Log N  + l}$;
    \item[(ii)]  $\F_\Z g$ vanishes on ${\mathcal M}_{\leq l, \leq -d\Log N  + dl}$,
\end{itemize}
where the logarithmic scale $\Log N$ of $N$ was defined in \eqref{log-scale}.
Then one has
\begin{equation}\label{single-decay}
\| \tilde A_N(f,g) \|_{\ell^1(\Z)} \lesssim_{C_1} (2^{-cl} + \langle \Log N \rangle^{-cC_1}) \|f\|_{\ell^2(\Z)} \|g\|_{\ell^2(\Z)}.
\end{equation}
\end{theorem}

This theorem will be used repeatedly in our arguments. The parameter $c>0$ from \eqref{single-decay} will be independent on the Ionescu--Wainger parameter $\rho$, see \eqref{rho-def}.  The secondary term $\langle \Log N \rangle^{-cC_1}$ is negligible in practice; the key point is the primary term $2^{-cl}$ that exhibits exponential decay on the height scale $l$.  It is important to note that only one of the hypotheses (i), (ii), as opposed to both, are required to hold in order to obtain this decay.  The asymmetry between (i) and (ii) is entirely caused by the different degrees in the two polynomials $\nn, P(\nn)$ used to form the averaging operator $\tilde A_N$.  This theorem only gives exponential decay directly for $\ell^2(\Z) \times \ell^2(\Z) \to \ell^1(\Z)$ operator norms, but in practice one can use interpolation to then obtain similar decay for other $\ell^{p_1}(\Z) \times \ell^{p_2}(\Z) \to \ell^p(\Z)$ operator norms.  We remark that it is essential in Theorem \ref{improv} that we are in the nonlinear regime $d \geq 2$, as there are easy counterexamples to this theorem in the linear case $d=1$ (as can be seen by testing \eqref{single-decay} against plane waves multiplied by suitable cutoff functions).

The proof of Theorem \ref{improv} will be somewhat lengthy, and relies on several deep results in the literature, including the inverse theory of Peluse and Prendiville \cite{PP1} and Peluse \cite{P2}, (see also \cite{PP2} and the survey of Prendiville \cite{Pr})  and $L^p$-improving estimates of Han--Kova{\v c}--Lacey--Madrid--Yang \cite{HKLMY} (see also Dasu--Demeter--Langowski
\cite{DDL}); we also use the properties of the Ionescu--Wainger projections that we shall define later in this section.  A key difficulty in the proof of Theorem \ref{improv} will be that the functions $f,g$ are only controlled in $\ell^2(\Z)$ rather than $\ell^\infty(\Z)$.
We will establish this bound in Section \ref{inverse-sec}.  We remark that a continuous analogue of Theorem \ref{improv}, with the domain $\Z$ replaced by $\R$, and with the major arc set replaced by an interval centered at the frequency origin, was established in \cite[Lemma 5]{BRoth} for monomial $P$ and in \cite[Lemma 1.4]{DGR} in the general case.

\begin{example}  Let $l \in \N$, and let $N$ be a sufficiently large integer depending on $l,P$.  Let $q$ be a prime number with $2^l < q \leq 2^{l+1}$ (which implies in particular $\frac{1}{q} \mod 1$ has height $2^{l+1}$), and consider the functions
\begin{align*}
f(n) &\coloneqq e(-n/q) \sum_{j \in [N^{d-1}]} \epsilon_j (\F_\R^{-1} \eta)\left( \frac{n-jN}{N} \right), \\
g(n) &\coloneqq e(-n/q) (\F_\R^{-1} \eta)\left( \frac{n}{N^d} \right),
\end{align*}
where $\epsilon_1,\dots, \epsilon_{N^{d-1}} \in \{-1,+1\}$ are arbitrary signs, and $\eta$ is defined in Section \ref{cutoff-sec}.  Then ${\mathcal F}_\Z f$ and ${\mathcal F}_\Z g$ vanish on ${\mathcal M}_{\leq l, \leq - \Log N + l}$ and ${\mathcal M}_{\leq l, \leq - d\Log N + dl}$ respectively, and routine calculations show that
$$ \| {\mathcal F}_\Z f\|_{\ell^2(\Z)}, \| {\mathcal F}_\Z g\|_{\ell^2(\Z)} \lesssim N^{d/2}$$
and also
$$ \| \tilde A_N(f,g) \|_{\ell^1(\Z)} \lesssim N^d \left( \left|\E_{n \in \Z/q\Z} e\left(\frac{n+P(n)}{q}\right)\right| + N^{-c} \right).$$
Standard exponential sum estimates (see e.g., \cite{IK}) reveal that
$$ \E_{n \in \Z/q\Z} e\left(\frac{n+P(n)}{q}\right) \lesssim q^{-c} \lesssim 2^{-cl}$$
(indeed, the Weil bounds allow one to take $c=1/2$ here), and so this example is consistent with Theorem \ref{improv}.  Variations of this example can also be used to explain the appearance of the scales $-\Log N$ and $-d\Log N$ in Theorem \ref{improv}(i), (ii), which are the frequency dual scales to the spatial scales $\Log N$, $\Log N^d$ associated to the shifts $n, P(n)$ for $n \in [N]$ arising in the definition of $A^{\nn,P(\nn)}_N$; we leave the details to the interested reader.
\end{example}

For the remainder of this section, let us assume Theorem \ref{improv} and see how we can use it to attack Theorem \ref{end-var}.  We will need an adelic version of Littlewood--Paley projection operators.  Let $\eta_{\leq k}$ be the cutoff functions from Section \ref{cutoff-sec}.  The Fourier multipliers $\T_{\eta_{\leq k}}$ are then standard Littlewood--Paley Fourier projections on $\Schwartz(\R)$ to the frequency interval $\R_{\leq k}$.  Motivated by this, we define the \emph{Ionescu--Wainger Fourier projection operator} $\Pi_{\leq l, \leq k}$ for any $(l,k) \in \N \times \Z$ using the construction \eqref{tlphi-def} by the formula
\begin{equation}\label{iw}
\Pi_{\leq l, \leq k} \coloneqq \T^{\leq l}_{\eta_{\leq k}}.
\end{equation}
More explicitly, one has
$$ \Pi_{\leq l, \leq k} f(x) = \sum_{\alpha \in (\Q/\Z)_{\leq l}} \int_\R \eta(\theta/2^k) \F_\Z f(\alpha+\theta) e(-x(\alpha+\theta))\ d\theta.$$
Note that $\Pi_{\leq l,\leq k}$ is self-adjoint on $\ell^2(\Z)$, and its symbol is supported on ${\mathcal M}_{\leq l,\leq k}$.  We similarly define
\begin{equation}\label{iw-2}
\Pi_{l, \leq k} \coloneqq \T^{l}_{\eta_{\leq k}} = \Pi_{\leq l,\leq k} - \Pi_{\leq l-1,\leq k}
\end{equation}
with the convention $\Pi_{\leq -1,k}=0$.

When $(l,k)$ have good major arcs, these operators have good properties:

\begin{lemma}[Properties of Ionescu--Wainger projections]\label{iw-prop} Let $(l,k) \in \N \times \Z$ be such that $(l,k)$ has good major arcs. \begin{itemize}
    \item [(i)] (Boundedness) The operator $\Pi_{\leq l, \leq k}$ is a contraction on $\ell^2(\Z)$. Furthermore, for any $1 < q < \infty$, one has
\begin{equation}\label{pilk}
\| \Pi_{\leq l, \leq k} f \|_{\ell^q(\Z)} \lesssim_{C_1,q} \langle l \rangle \|f\|_{\ell^q(\Z)}.
\end{equation}
In particular, $\Pi_{\leq l, \leq k}$ extends to a bounded linear operator on $\ell^q(\Z)$.  If $f$ is furthermore supported on an interval $I$, we have the off-diagonal decay bound
\begin{equation}\label{off-decay}
\| \Pi_{\leq l, \leq k} f \|_{\ell^q(J)} \lesssim_{C_1,q,M} \langle l \rangle \langle 2^k \mathrm{dist}(I,J) \rangle^{-M} \|f\|_{\ell^q(I)}
\end{equation}
for any interval $J$, and any $M \in \N$.
    \item[(ii)] (Fourier support) If $f \in \ell^2(\Z)$, then $\Pi_{\leq l, \leq k} f$ is Fourier supported in ${\mathcal M}_{\leq l, \leq k}$, and $\Pi_{\leq l, \leq k} f=f$ when $\F_\Z f$ is Fourier supported in ${\mathcal M}_{\leq l, \leq k-1}$.
\end{itemize}
All these claims also hold when all occurrences of $\leq l$ are replaced by $l$.
\end{lemma}

\begin{proof} See Appendix \ref{iw-app}.
\end{proof}

\begin{remark}[Physical space interpretation of major arcs]\label{uncertainty} By uncertainty principle heuristics, functions $f \in \Schwartz(\Z)$ which have Fourier support in ${\mathcal M}_{\leq l, \leq k}$, where $(l,k) \in \N \times \Z$ satisfy \eqref{m-small}, can be viewed as behaving like linear combinations of indicator functions $\ind{P}$ of arithmetic progressions $P$ of spacing $O(2^l)$ and diameter $O(2^{-k})$, and behave like constants on arithmetic progressions of spacing $Q_{\leq l}$ and diameter $O(2^{-k})$; the latter is only non-vacuous in the ``large-scale'' regime in which $2^{-k}$ is larger than $Q_{\leq l}$.  
Dually, functions $f \in \Schwartz(\Z)$ whose Fourier transform vanishes on ${\mathcal M}_{\leq l, \leq k}$ morally have negligible mean on the two types of arithmetic progressions just mentioned.  The reader is invited to compare Figure \ref{fig:unc} with Figure \ref{fig:major} through the lens of this uncertainty principle.
\end{remark}

Now we can use Theorem \ref{improv} and Lemma \ref{iw-prop} to achieve some reductions to prove Theorem \ref{end-var}.  It will suffice to establish the estimate
\begin{equation}\label{vrb}
\| ( \tilde A_N(f,g))_{N \in \D} \|_{\ell^p(\Z; \V^r)} \lesssim_{C_3} \|f\|_{\ell^{p_1}(\Z)} \|g\|_{\ell^{p_2}(\Z)}.
\end{equation}

For each individual $N < C_3$ this claim is immediate from \eqref{anf}, so we may assume without loss of generality that $N \geq C_3$ for all $N \in \D$.
If $N \geq C_3$, define the quantities
\begin{equation}\label{kl-def}
 l_{(N)} \coloneqq C_0 \Log \Log N.
 \end{equation}
Then by \eqref{m-small} the pairs $(l_{(N)},-\Log N + l_{(N)})$, $(l_{(N)}, -d\Log N + dl_{(N)})$ have good major arcs, and hence by Lemma \ref{iw-prop}(i), (ii) and Theorem \ref{improv}, if $C_1\ge C_0$ one has the estimate
$$ \| \tilde A_N( (1 - \Pi_{\leq l_{(N)}, -\Log N + l_{(N)}}) f,g) \|_{\ell^1(\Z)} \lesssim_{C_1} (\Log N)^{-cC_0} \| f\|_{\ell^2(\Z)} \|g\|_{\ell^2(\Z)}.$$
On the other hand, from Lemma \ref{iw-prop}(i) and \eqref{anf} one also has
$$ \| \tilde A_N( (1 - \Pi_{\leq l_{(N)}, -\Log N + l_{(N)}}) f,g) \|_{\ell^q(\Z)} \lesssim_{C_1,q,q_1,q_2} (\Log\Log N) \| f\|_{\ell^{q_1}(\Z)} \|g\|_{\ell^{q_2}(\Z)}$$
for any $1 < q_1,q_2 < \infty$ with $1/q_1+1/q_2=1/q \leq 1$.  Interpolating, we conclude that
$$ \| \tilde A_N( (1 - \Pi_{\leq l_{(N)}, -\Log N + l_{(N)}}) f,g) \|_{\ell^p(\Z)} \lesssim_{C_1} (\Log\Log N)^{O(1)} (\Log N)^{-cC_0} \| f\|_{\ell^{p_1}(\Z)} \|g\|_{\ell^{p_2}(\Z)}$$
(recall that $c$ varies from line to line and is allowed to depend on $p_1,p_2,p$). In particular, for $C_0$ large enough one has
$$ \| \tilde A_N( (1 - \Pi_{\leq l_{(N)}, -\Log N + l_{(N)}}) f,g) \|_{\ell^p(\Z)} \lesssim_{C_1} (\Log N)^{-10} \| f\|_{\ell^{p_1}(\Z)} \|g\|_{\ell^{p_2}(\Z)}$$
(say).  A similar argument gives
\begin{multline*}
   \| \tilde A_N( \Pi_{\leq l_{(N)}, -\Log N + l_{(N)}} f, (1 - \Pi_{\leq l_{(N)}, \leq -d\Log N + dl_{(N)}}) g) \|_{\ell^p(\Z)} \\
   \lesssim_{C_1} (\Log N)^{-10} \| f\|_{\ell^{p_1}(\Z)} \|g\|_{\ell^{p_2}(\Z)}; 
\end{multline*}
by the triangle inequality and bilinearity of $\tilde A_N$, we conclude that
\begin{multline*}
    \| \tilde A_N( f,g) - \tilde A_N( \Pi_{\leq l_{(N)}, \leq -\Log N + l_{(N)}} f, \Pi_{\leq l_{(N)}, \leq -d\Log N + dl_{(N)}} g) \|_{\ell^p(\Z)} \\
    \lesssim_{C_1} (\Log N)^{-10} \| f\|_{\ell^{p_1}(\Z)} \|g\|_{\ell^{p_2}(\Z)}.
\end{multline*}
From the $\lambda$-lacunary nature of $\D$ we have
$$ \sum_{N \in \D: N \geq C_3} (\Log N)^{-10} \lesssim 1$$
and hence by \eqref{varsum} we have that
\begin{multline*}
\| ( \tilde A_N(f,g) - \tilde A_N( \Pi_{\leq l_{(N)}, \leq -\Log N + l_{(N)}} f, \Pi_{\leq l_{(N)}, \leq -d\Log N + dl_{(N)}} g))_{N \in \D} \|_{\ell^p(\Z;\V^r)} \\
\lesssim_{C_1} \|f\|_{\ell^{p_1}(\Z)} \|g\|_{\ell^{p_2}(\Z)}.    
\end{multline*}
By a further application of the triangle inequality, we conclude that to establish \eqref{vrb}, it suffices to prove the major arc bound
$$
\| ( \tilde A_N( \Pi_{\leq l_{(N)}, \leq -\Log N + l_{(N)}} f, \Pi_{\leq l_{(N)}, \leq -d\Log N + dl_{(N)}} g))_{N \in \D} \|_{\ell^p(\Z;\V^r)} \lesssim_{C_3} \|f\|_{\ell^{p_1}(\Z)} \|g\|_{\ell^{p_2}(\Z)}.
$$
We now perform an ``arithmetic'' dyadic decomposition
$$ \Pi_{\leq l, \leq m} = \sum_{0 \leq l' \leq l} \Pi_{l',\leq m}.$$
By the triangle inequality, it now suffices to show the bound
\begin{multline}\label{l1l2}
\| ( \tilde A_N( \Pi_{l_1, \leq -\Log N + l_{(N)}} f, \Pi_{l_2, \leq -d\Log N + dl_{(N)}} g) \ind{l_1,l_2 \leq l_{(N)}})_{N \in \D} \|_{\ell^p(\Z; \V^r)}\\
\lesssim_{C_3} 2^{-\rho l} \|f\|_{\ell^{p_1}(\Z)} \|g\|_{\ell^{p_2}(\Z)}
\end{multline}
for all $l_1,l_2 \in \N$, where 
\begin{equation}\label{l-def}
l \coloneqq \max(l_1,l_2).
\end{equation}
Note that the constraint $l_1,l_2 \leq l_{(N)}$ serves as an additional lower bound on $N$ (and in particular the left-hand side of \eqref{l1l2} vanishes for all but finitely many $l_1,l_2$, thanks to the finite nature of $\D$), so we may also write this bound as
\begin{multline}\label{l1l2-alt}
\| ( \tilde A_N( \Pi_{l_1, \leq -\Log N + l_{(N)}} f, \Pi_{l_2, \leq -d\Log N + dl_{(N)}} g))_{N \in \D; l_1,l_2 \leq l_{(N)}}  \|_{\ell^p(\Z;\V^r)} \\
\lesssim_{C_3} 2^{-\rho l} \|f\|_{\ell^{p_1}(\Z)} \|g\|_{\ell^{p_2}(\Z)}.
\end{multline}
Fix $l_1,l_2$ (and hence $l$), and then introduce the quantity
\begin{equation}\label{u-def}
u \coloneqq \lfloor C_2 2^{2\rho l} \rfloor.
\end{equation}
We now combine the previous ``arithmetic'' dyadic decomposition with a ``continuous'' dyadic decomposition
\begin{align*}
\Pi_{l_1, \leq -\Log N + l_{(N)}} f &= \sum_{-u \leq s_1 \leq l_{(N)}} F^{u,l_1,s_1}_N, \\
\Pi_{l_2, \leq -d\Log N + dl_{(N)}} g &= \sum_{-u \leq s_2 \leq l_{(N)}} G^{u,l_2,s_2}_N,
\end{align*}
where
\begin{equation}\label{FN-def}
 F^{u,l_1,s_1}_N \coloneqq \begin{cases} \Pi_{l_1, \leq -\Log N + s_1} f - \Pi_{l_1, \leq -\Log N + s_1-1} f & s_1 > -u \\
\Pi_{l_1, \leq -\Log N - u} f & s_1 = -u
\end{cases}
\end{equation}
and
\begin{equation}\label{GN-def}
G^{u,l_2,s_2}_N \coloneqq \begin{cases} \Pi_{l_2, \leq d(-\Log N+s_2)} g - \Pi_{l_2, \leq d(-\Log N+s_2-1)} g  & s_2 > -u \\
\Pi_{l_2, \leq d(-\Log N-u)} g & s_2 = -u.
\end{cases}
\end{equation}
Informally, $F^{u,l_1,-u}_N$, $G^{u,l_2,-u}_N$ represent the ``low (continuous) frequency'' components of $f,g$ respectively, whereas $F^{u,l_1,s_1}_N, s_1>-u$ and $G^{u,l_2,s_2}_N, s_2>-u$ represent the ``high (continuous) frequency'' components.

By the triangle inequality we can bound the left-hand side of \eqref{l1l2-alt} by
$$ \sum_{s_1,s_2 \geq -u} \| ( \tilde A_N( F^{u,l_1,s_1}_N, G^{u,l_2,s_2}_N))_{N \in \I^{l,s_1,s_2}} \|_{\ell^p(\Z; \V^r)}, $$
where $\I^{l,s_1,s_2}$ denotes the index set
\begin{equation}\label{index-def}
 \I^{l,s_1,s_2} \coloneqq \{ N \in \D:  l,s_1,s_2 \leq l_{(N)}\}.
\end{equation}
The expression $\tilde A_N( F^{u,l_1,s_1}_N, G^{u,l_2,s_2}_N)$ can be viewed as (the scale $N$ component of) a paraproduct of $F$ and $G$, but centered around a finite number of (arithmetic) frequencies, in contrast to the classical paraproducts that are centered at the frequency origin; also, the paraproduct symbol exhibits some additional oscillation compared to classical paraproducts when $s_1,s_2$ become large.  We shall sometimes distinguish between the ``high-high'' case $s_1,s_2 > -u$, the ``low-high'' case $s_2 > s_1=-u$, the ``high-low'' case $s_1 > s_2 =-u$, and the ``low-low'' case $s_1=s_2=-u$ of these paraproducts. But for now we can treat all choices of $s_1,s_2$ in a unified fashion.

By several applications of the triangle inequality, the bound \eqref{l1l2-alt}, and hence Theorem \ref{end-var}, now follows from the following variational paraproduct estimates, in which we request an exponential gain in the $p_1=p_2=2$ case and relatively small losses in all other cases:

\begin{theorem}[Variational paraproduct estimates]\label{varp}  Let the hypotheses be as in Theorem \ref{end-var}, and the notational conventions be as in this section.  Let $l_1,l_2 \in\N$, and define $l,u$ by \eqref{l-def}, \eqref{u-def} respectively.  Let $s_1,s_2 \geq -u$, and then let $F_N := F^{u,l_1,s_1}_N$, $G_N := G^{u,l_2,s_2}_N$, $\I := \I^{l,s_1,s_2}$ be defined respectively by \eqref{FN-def}, \eqref{GN-def}, \eqref{index-def}.  Then
\begin{multline}\label{all-all-weak}
\| ( \tilde A_N( F_N, G_N ))_{N \in \I} \|_{\ell^p(\Z;\V^r)} \\
\lesssim_{C_3} \langle \max(l,s_1,s_2) \rangle^{O(1)} 2^{O(\rho l)-c \max(l,s_1,s_2) \ind{p_1=p_2=2}}
\|f\|_{\ell^{p_1}(\Z)} \|g\|_{\ell^{p_2}(\Z)}.
\end{multline}
Here the constant $c$ does not depend on $\rho$, see the discussion below Theorem \ref{improv}.
\end{theorem}

Indeed, by interpolating \eqref{all-all-weak} between the case $(p_1,p_2,p)=(2,2,1)$ and the case where $(p_1,p_2,p)$ are close to $(1,\infty,1)$, $(\infty, 1,1)$, or $(\infty,\infty,\infty)$, we see that
\begin{equation}\label{all-all}
\| ( \tilde A_N( F_N, G_N ))_{N \in \I} \|_{\ell^p(\Z;\V^r)} \lesssim_{C_3} \langle  \max(l,s_1,s_2) \rangle^{O(1)} 2^{-10\rho \max(l,s_1,s_2)} \|f\|_{\ell^{p_1}(\Z)} \|g\|_{\ell^{p_2}(\Z)}
\end{equation}
(say), and then using $\langle a\rangle^{O(1)} 2^{-10\rho a} \lesssim_{C_3} 2^{-8\rho a}$ for $a \geq 0$ and summing the bound in \eqref{all-all} over $s_1,s_2$ we see that to obtain the \eqref{l1l2-alt} from Theorem \ref{varp}, it suffices to establish the bound
$$  \sum_{s_1,s_2 \geq -u} 2^{-8\rho \max(l,s_1,s_2)}\lesssim_{C_3} 2^{-\rho l};$$
bounding
$$ 2^{-8\rho \max(l,s_1,s_2)} \leq 2^{-4\rho \max(l,s_1)} 2^{-4\rho \max(l,s_2)}$$
it suffices to show that
$$ \sum_{s_0 \geq -u} 2^{-4\rho \max(l,s_0)} \lesssim_{C_3} 2^{-\rho l/2}.$$
But this is clear from the geometric series formula since there are only $O_{C_2}( 2^{2\rho l})$ scales $s_0$ with $-u \leq s_0 \leq l$.

It remains to establish Theorem \ref{improv} and Theorem \ref{varp}.  Theorem \ref{improv} will be established in the next section; the rest of the paper is then devoted to the proof of Theorem \ref{varp}.  For now, we use Theorem \ref{improv} to deal with one case of Theorem \ref{varp}:

\begin{proposition}[High-high $\ell^2(\Z)$ case]\label{high-high-l2}  Theorem \ref{varp} holds when $s_1,s_2 > -u$ and $p_1=p_2=2$.
\end{proposition}

In view of this proposition, for the purposes of proving Theorem \ref{varp} we may assume that at least one of $s_1=-u$, $s_2=-u$, or $(p_1,p_2) \neq (2,2)$ holds.

\begin{proof}  From \eqref{varsum} we have
$$ \| (\tilde  A_N( F_N, G_N ))_{N \in \I} \|_{\ell^1(\Z;\V^r)} \lesssim \sum_{N \in \I} \| \tilde A_N(F_N,G_N) \|_{\ell^1(\Z)}.$$
Observe (using Lemma \ref{mag-lem}, \eqref{FN-def}) that for $N \in \I$, $\F_\Z F_N$ vanishes on the major arcs  ${\mathcal M}_{\leq \max(l_1,s_1)-1, \leq -\Log N+\max(l_1,s_1)-1}$, and hence by Theorem \ref{improv} we have
$$ \| \tilde A_N(F_N,G_N) \|_{\ell^1(\Z)} \lesssim_{C_1} (2^{-c\max(l_1,s_1)} + \langle \Log N \rangle^{-cC_1}) \|F_N\|_{\ell^2(\Z)} \|G_N\|_{\ell^2(\Z)}.
$$
A similar argument gives
$$ \| \tilde A_N(F_N,G_N) \|_{\ell^1(\Z)} \lesssim_{C_1} (2^{-c\max(l_2,s_2)} + \langle \Log N \rangle^{-cC_1}) \|F_N\|_{\ell^2(\Z)} \|G_N\|_{\ell^2(\Z)}
$$
and hence on taking geometric means
$$ \| \tilde A_N(F_N,G_N) \|_{\ell^1(\Z)} \lesssim_{C_1} (2^{-c\max(l,s_1,s_2)} + \langle \Log N \rangle^{-cC_1}) \|F_N\|_{\ell^2(\Z)} \|G_N\|_{\ell^2(\Z)}.
$$
From \eqref{index-def} we have $\langle \Log N \rangle^{-cC_1} \lesssim_{C_3} 2^{-c\max(l,s_1,s_2)}$, hence
$$ \| \tilde A_N(F_N,G_N) \|_{\ell^1(\Z)} \lesssim_{C_3} 2^{-c\max(l,s_1,s_2)} \|F_N\|_{\ell^2(\Z)} \|G_N\|_{\ell^2(\Z)}.
$$
By the Cauchy--Schwarz inequality, it thus suffices to establish the Bessel-type inequalities
$$ \sum_{N \in \I} \| F_N \|_{\ell^2(\Z)}^2 \lesssim \|f\|_{\ell^2(\Z)}^2$$
and
$$ \sum_{N \in \I} \| G_N \|_{\ell^2(\Z)}^2 \lesssim \|g\|_{\ell^2(\Z)}^2.$$
But this follows from the easily verified pointwise bounds
\begin{align*}
\sum_{N \in \I} |\F_\Z F_N(\xi)|^2 &\lesssim |\F_\Z f(\xi)|^2, \\
\sum_{N \in \I} |\F_\Z G_N(\xi)|^2 &\lesssim |\F_\Z g(\xi)|^2 
\end{align*}
and Plancherel's theorem.
\end{proof}

\section{Minor arc single scale estimate: applying Peluse--Prendiville theory}\label{inverse-sec}

In this section we establish Theorem \ref{improv}.  The arguments here will be lengthy, but are not needed elsewhere in this paper.

It will be convenient to exploit duality and work with trilinear forms $\langle \tilde A_N(f,g), h \rangle$ instead of bilinear operators $\tilde A_N(f,g)$.  We use the inner product
$$ \langle f, g \rangle \coloneqq \sum_{x \in \Z} f(x) g(x)$$
on $\Schwartz(\Z)$ (there will be no advantage for us in this bilinear analysis in inserting a complex conjugation into the inner product), and observe the identities
\begin{equation}\label{transpose}
 \langle \tilde A_N(f,g), h \rangle = \langle \tilde A_N^*(h,g), f \rangle = \langle \tilde A_N^{**}(f,h), g \rangle
\end{equation}
for $f,g,h \in \Schwartz(\Z)$, where the
transpose operators $\tilde A_N^*, \tilde A_N^{**}$ are the averaging operators
\begin{equation}\label{bkhg}
\tilde A_N^*(h,g)(x) \coloneqq \tilde A^{-\nn,P(\nn)-\nn}_N(h,g)(x) =  \E_{n \in [N]}  h(x+n) g(x+n-P(n)) \ind{n>N/2}
\end{equation}
and
\begin{align}
\label{eq:bkhf}
\tilde A_N^{**}(f,h)(x) \coloneqq \tilde A^{\nn-P(\nn), -P(\nn)}_N(f,h)(x) = \E_{n \in [N]}  f(x+P(n)-n) h(x+P(n)) \ind{n>N/2}.
\end{align}
In the language of additive combinatorics, the functions
$\tilde A_N^*(h,g), \tilde A_N^{**}(f,h)$ are referred to as \emph{dual functions}.

\subsection{Proof of Theorem \ref{improv}(i)}

Our starting point is the following deep inverse theorem of
Peluse--Prendiville \cite{PP1} in the quadratic case $P(\nn)=\nn^2$ (see also \cite{PP2} and \cite{Pr}), and Peluse \cite{P2} for general polynomials $P(\nn)$ of degree $d \geq 2$.

\begin{theorem}[Peluse  inverse theorem]\label{pp}  Let $N \geq 1$ and $0 < \delta \leq 1$, and let $N_0$ be a quantity with $N_0 \sim N^d$.  Let $f,g,h \in \Schwartz(\Z)$ be supported on $[-N_0,N_0]$ with $\|f\|_{\ell^\infty(\Z)},$$ \|g\|_{\ell^\infty(\Z)},$ $ \|h\|_{\ell^\infty(\Z)} \leq 1$, obeying the lower bound
\begin{equation}\label{hyp}
 |\langle \tilde A_N(f,g), h \rangle| \geq \delta N^d.
 \end{equation}
Then one of the following holds:
\begin{itemize}
    \item[(i)]  ($N$ not too large) One has $N \lesssim \delta^{-O(1)}$.
    \item[(ii)]  ($f$ has major arc structure at scale $N$) There exists a positive integer $q \lesssim \delta^{-O(1)}$ and a positive integer $\delta^{O(1)} N \lesssim N' \leq N$ such that
    $$ \frac{1}{N^d} \Big|\sum_{x \in \Z} \E_{m \in [N']} f(x+qm)\Big| \gtrsim \delta^{O(1)}.$$
\end{itemize}
\end{theorem}

Note from the uncertainty principle (cf. Remark \ref{uncertainty}) that conclusion (ii) of Theorem \ref{pp} is morally equivalent to asserting that the Fourier transform $\F_\Z f$ has a large presence on a major arc set ${\mathcal M}_{\leq l, \leq k, \naive}$ with $2^l \lesssim \delta^{-O(1)}$ and $2^{-k} \lesssim \delta^{-O(1)}/N$. This intuition will be formalized in Proposition \ref{lip} below.

\begin{proof} We expand out \eqref{hyp} as
\begin{align*}
\frac{1}{N^{d+1}} \Big|\sum_{n\in [N]}\sum_{x\in\Z}h(x)f(x-n)g(x-P(n)) \ind{n>N/2}\Big|\ge \delta.
\end{align*}
By the triangle inequality, we thus have
$$ \frac{1}{(N')^{d+1}} \Big|\sum_{n\in [N']}\sum_{x\in\Z}h(x)f(x-n)g(x-P(n)) \Big|\gtrsim \delta.
$$
for either $N'=N$ or $N'=\lfloor N/2\rfloor$. The claim now follows from \cite[Theorem 3.3]{P2} (after some minor
changes of notation) with parameters $(m,q,N,M,P_1,P_2) = (2,1,N_0,N',\nn, P(\nn))$.  In that theorem, the functions $f,g,h$ were assumed to be supported on $[1,(N')^d]$ rather than $[-N_0,N_0]$, but it is a
routine matter to see that the arguments continue to hold with this
slightly more general support hypothesis.
\end{proof}

We will now gradually manipulate Theorem \ref{pp} in a sequence of steps to make it more closely resemble (the contrapositive of) Theorem \ref{improv}(i), until we are able to actually establish that part of the theorem; we will then adapt the argument (focusing on $g$ instead of $f$) to also establish Theorem \ref{improv}(ii).

The first step is to make the conclusion of Theorem \ref{pp} more Fourier-analytic in nature.  We need a technical calculation:

\begin{lemma}[Smooth approximation to $\ind{[a,b]}$] \label{calc}  Let $\psi \in \Schwartz(\RR)$ with $\int_\RR \psi(x)\ dx = 1$.  Then for any interval $[a,b] \subset \R$ and any $0 < \eps \leq 1$ one has the pointwise bound
$$ \sum_{y \in [a,b] \cap \Z} \eps \psi(\eps(x-y)) - \ind{[a,b]}(x) \lesssim_\psi \eps^{10} + \langle \eps(x-a) \rangle^{-10} + \langle \eps(x-b) \rangle^{-10}.$$
for all $x \in \Z$.
\end{lemma}

\begin{proof}  By the triangle inequality it suffices to show that
$$ \sum_{y \in \Z: y \geq a} \eps \psi(\eps(x-y)) - \ind{x \geq a} \lesssim_\psi \eps^{10} + \langle \eps(x-a) \rangle^{-10}$$
since the claim then follows by subtracting this estimate from the analogous estimate for $b$ (adjusting $b$ by an infinitesimal amount if necessary).  By translation invariance we may set $a=0$.  From the Poisson summation formula and the rapid decrease of $\F \psi$ one has
$$ \sum_{y \in \Z} \eps \psi(\eps(x-y)) = 1 + O_\psi(\eps^{10})$$
so by reflection symmetry and the triangle inequality it suffices to show that
$$ \sum_{y \in \Z: y \geq 0} \eps \psi(\eps(x-y))  \lesssim_\psi \langle \eps x \rangle^{-10}$$
when $x < 0$.  But this follows from the rapid decrease of $\psi$.
\end{proof}

\begin{proposition}[Alternate inverse theorem for $f$]\label{lip}  Under the hypotheses and notation of Theorem \ref{pp}, there exists a function $F \in \ell^2(\Z)$ with 
    \begin{equation}\label{F-bounds}
    \|F\|_{\ell^\infty(\Z)} \lesssim 1; \quad \|F\|_{\ell^1(\Z)} \lesssim N^d
    \end{equation}
    and with $\F_{\Z}F$  supported in the $O(\delta^{-O(1)}/N)$-neighborhood of some $\alpha \in \Q/\Z$ of naive height $O(\delta^{-O(1)})$ such that
    \begin{equation}\label{F-corr}
    |\langle f, F \rangle| \gtrsim \delta^{O(1)} N^d.
    \end{equation}
\end{proposition}

\begin{proof}  If $N \lesssim \delta^{-O(1)}$ then we can simply take $F =\tilde A^*_N(h,g)$ and $a/q=1/1$ and use \eqref{transpose} and \eqref{anf} to conclude.  Thus we may assume that $N \geq C_* \delta^{-C_*}$ for a sufficiently large $C_*$.
In particular, by Theorem \ref{pp}, we can find $N', q \in \Z_+$ with $q \lesssim \delta^{-O(1)}$ and $\delta^{O(1)} N \lesssim N' \leq N$ such that
$$ \sum_{x \in \Z} |\E_{m \in [N']} f(x+qm)| \gtrsim \delta^{O(1)} N^d.$$
Observe that the summand vanishes unless $|x| \leq N_0 + O( qN' ) \lesssim N^d$, thus
$$ \sum_{x = O(N^d)} |\E_{m \in [N']} f(x+qm)| \gtrsim \delta^{O(1)} N^d.$$

Now we smooth out the inner average $\E_{m \in [N']}$.
Let $0 < \eps \leq 1$ be a parameter to be chosen later.  From Lemma \ref{calc} one has
$$ \ind{[N']}(m) = \eps \sum_{m' \in [N']} \F_\R^{-1} \eta(\eps(m-m')) + O(\eps^{10} + \langle \eps m \rangle^{-10} + \langle \eps (m-N') \rangle^{-10})$$
for any $m \in \Z$, where $\eta$ is the cutoff from Section \ref{cutoff-sec}. Hence from the boundedness of $f$
$$\E_{m \in [N']} f(x+qm) = \eps \sum_{m \in \Z} \E_{m' \in [N']} \F_\R^{-1} \eta(\eps(m-m')) f(x+qm) + O\left( \eps^{10} + \frac{1}{\eps N'} \right).$$
If we choose $\eps \coloneqq C \delta^{-C} / N$ for some large $C$ (depending only on $\eta,P$), and take $C_*$ large enough depending on $C$, we conclude that
$$ \sum_{x = O(N^d)} \Big|\eps \sum_{m \in \Z} \E_{m' \in [N']} \F_\R^{-1} \eta(\eps(m-m')) f(x+qm)\Big| \gtrsim \delta^{O(1)} N^d.$$
In the latter case, there exists $G \in \ell^\infty(\Z)$ supported on $[-O(N^d),O(N^d)]$ with $\|G\|_{\ell^\infty(\Z)} \leq 1$ with
$$ \Big|\sum_{x \in \Z} G(x) \eps \sum_{m \in \Z} \E_{m' \in [N']} \F_\R^{-1} \eta(\eps(m-m')) f(x+qm)\Big| \gtrsim \delta^{O(1)}N^d.$$
We thus have the claim \eqref{F-corr} with
$$ F(x) \coloneqq \eps \sum_{m \in \Z} \E_{m' \in [N']} \F_\R^{-1} \eta(\eps(m-m')) G(x-qm).$$
From the hypotheses on $\eta,G$ we easily verify the bounds \eqref{F-bounds}.  A routine calculation using the Poisson summation formula reveals the identity
$$ {\mathcal F}_\Z F(\xi \mod 1) = {\mathcal F}_\Z G(\xi \mod 1) \E_{m' \in [N']}e(qm'\xi) \sum_{n \in \Z} {\mathcal F}_\RR \F_\R^{-1} \eta\left(\frac{q\xi-n}{\eps}\right)$$
for any $\xi \in \R$, which in particular implies from the support of $\eta$ that ${\mathcal F}_\Z F$ is supported in the set
$$ \pi\left( \left[-\frac{\eps}{q},\frac{\eps}{q}\right] \times \bigg(\frac{1}{q}\Z / \Z\bigg) \right).$$
By applying suitable Fourier multiplier operators, one can then decompose $F = \sum_{a \in [q]} F_a$, where each $F_a$ obeys essentially the same bounds \eqref{F-bounds} as $F$ and is supported in the $\frac{\eps}{q}$-neighborhood of $\frac{a}{q} \mod 1$.  The claim now follows from the pigeonhole principle and the bounds on $\eps,q$. \end{proof}

We now dualize the above proposition using the Hahn--Banach theorem to obtain control on dual functions $\tilde A^*(h,g)$. Specifically, we shall use the following lemma.

\begin{lemma}[Application of Hahn--Banach]\label{hahn} Let $A,B > 0$,
 and let $G$ be an element of
$\ell^2(\Z)$.  Let $\Phi$ be a family of vectors in $\ell^2(\Z)$, and assume
the following inverse theorem: whenever $f \in \ell^2(\Z)$ is such that
$\|f\|_{\ell^\infty(\Z)} \leq 1$ and $|\langle f, G \rangle| > A$, then
$|\langle f, \phi \rangle| > B$ for some $\phi \in \Phi$. Then $G$
lies in the closed convex hull of
\begin{equation}\label{dip}
V= \{ \lambda\phi\in \ell^2(\Z): \phi \in \Phi, \ |\lambda| \leq A/B\}
\cup \{ h \in \ell^2(\Z): \|h\|_{\ell^1(\Z)} \leq A \}.
\end{equation}
\end{lemma}

\begin{proof}
Observe that the set $\overline{{\rm conv} V}^{\|\cdot\|_{\ell^2(\Z)}}$ is
balanced.  Therefore, if the claim of Lemma \ref{hahn} failed, then
from the Hahn--Banach theorem and the Riesz representation theorem
there exists $f \in \ell^2(\Z)$ such that
$\mathrm{Re} \langle f, G \rangle > A$, but
$\mathrm{Re} \langle f, h \rangle \leq A$ for all $h\in V$.  In
particular, this gives $|\langle f, h\rangle|\le A$ for all $h\in V$,
which implies that
\[
|\langle f, \phi\rangle| \leq B
\]
for all $\phi \in \Phi$, and that
\[
\|f\|_{\ell^\infty(\Z)}=\sup_{\|h\|_{\ell^1(\Z)}\le1} |\langle f, h\rangle|\le 1,
\]
contradicting the hypothesis.  This completes the
proof of the lemma.
\end{proof}

\begin{corollary}[Structure of dual function, I]\label{struct}  Let $N \geq 1$, let $N_0 \sim N^d$, and let $g,h \in \Schwartz(\Z)$ be supported on $[-N_0,N_0]$ with
$\|g\|_{\ell^\infty(\Z)}, \|h\|_{\ell^\infty(\Z)} \leq 1$, and let $0 < \delta \leq 1$.  Then there exists a decomposition
\begin{equation}\label{decomp}
\tilde A_N^*(h,g) = \sum_{\alpha \in \Q/\Z: \Height_\naive(\alpha) \lesssim \delta^{-O(1)}} F_{\alpha} + E_1 + E_2,
\end{equation} 
where each $F_{\alpha} \in \ell^2(\Z)$ has Fourier transform supported in the $O(\delta^{-O(1)}/N)$-neighborhood of $\alpha$
and obeys the bounds
\begin{equation}\label{faq-bound}
\| F_{\alpha}\|_{\ell^\infty(\Z)} \lesssim \delta^{-O(1)};
\quad\text{ and }\quad
\|F_{\alpha}\|_{\ell^1(\Z)} \lesssim \delta^{-O(1)}N^d,
\end{equation}
and the error terms $E_1 \in \ell^1(\Z)$ and $E_2 \in \ell^2(\Z)$ obey the bounds
\begin{equation}\label{e-bound}
\|E_1\|_{\ell^1(\Z)} \leq \delta N^d;
\quad \text{ and } \quad
\|E_2\|_{\ell^2(\Z)} \leq \delta. 
\end{equation}
\end{corollary}

For similar applications of the Hahn--Banach theorem to analyze the structure of dual functions in additive combinatorics, see \cite[pp. 221]{HK3}, \cite[Theorem 3.8]{Gowers}.

\begin{proof} If there exists $f\in \ell^{\infty}(\Z)$ with $\|f\|_{\ell^{\infty}(\Z)}\le 1$ such that
\begin{align}
\label{eq:16}
|\langle f, \tilde A_N^*(h,g)  \rangle|> \delta N^d.
\end{align}
Applying Proposition \ref{lip} we obtain
\begin{align}
\label{eq:4}
|\langle f, F  \rangle|\gtrsim \delta^{O(1)} N^d
\end{align}
for some function $F\in \ell^2(\Z)$ obeying the properties of Proposition \ref{lip}.
Invoking  Lemma \ref{hahn} with $A=\delta N^d/2$ and $B \sim \delta^{O(1)}N^d$ and the set
\[
\Phi=\{\phi_{\alpha}\in \ell^2(\Z): \alpha \in \Q/\Z; \Height_\naive(\alpha) \lesssim \delta^{-O(1)}\},
\]
we obtain a decomposition
\begin{align}
\label{eq:17}
\tilde A_N^*(h,g) = \sum_{j=1}^{\infty}c_j\phi_j+E_1+E_2,
\end{align}
with the following properties:
\begin{itemize}
\item[(i)]  for each $j\in\Z_+$ we have that $\phi_j=\lambda_j\phi_{\alpha_j}$ for some $\phi_{\alpha_j}\in\Phi$ and $\lambda_j\in\C$ such that  $|\lambda_j|\lesssim \delta^{-O(1)}$;
\item[(ii)] the coefficients $c_j$ are non-negative with
$\sum_{j=1}^{\infty}c_j\le 1$, and all but finitely $c_j$ vanish;
\item[(iii)] the error term $E_1\in \ell^1(\Z)$  satisfies  $\|E_1\|_{\ell^1(\Z)} \leq \delta N^d$;
\item[(iv)] the error term $E_2\in \ell^2(\Z)$  satisfies  $\|E_2\|_{\ell^2(\Z)} \leq \delta$.
\end{itemize}
The latter error term arises as a consequence of the fact that
one is working with the closed convex hull instead of the convex hull. In fact, its $\ell^2(\Z)$ norm can be made arbitrarily small, but
$\delta$ will suffice for our purposes.
Grouping together terms associated to each arithmetic frequency $\alpha$ in \eqref{eq:17} and using the triangle inequality, we obtain the desired decomposition from
\eqref{decomp} that satisfies \eqref{faq-bound} and \eqref{e-bound}.
\end{proof}

Corollary \ref{struct} is not directly suitable for our applications
for three reasons: firstly, $E_1$ is  controlled in $\ell^1(\Z)$ rather than 
in $\ell^2(\Z)$; secondly, $g$ is required to be controlled in $\ell^\infty(\Z)$ rather
than in $\ell^2(\Z)$; and thirdly the support of $g$ is restricted to an
interval.  Using the Ionescu--Wainger projections, we now address the
first issue, at the cost of worsening the control of the structured
component of the decomposition \eqref{decomp}, and also requiring
$\delta$ to not be too small.

\begin{proposition}[Structure of dual function, II]\label{propdual}
If $N,N_0 \in \Z_+$ with $N_0 \sim N^d$ and $l \in \N$, with 
\begin{equation}\label{2l}
\Log N \geq C_\rho 2^{\rho l}
\end{equation}
for a sufficiently large constant $C_\rho$ depending on $\rho$, one has the estimate
\[
\|  (1 - \Pi_{\leq l, \leq -\Log N+l}) \tilde A_N^*(h,g) \|_{\ell^2(\Z)} \lesssim_{C_1} 2^{-cl}
 N^{d/2}  \|h\|_{\ell^\infty(\Z)} \|g\|_{\ell^\infty(\Z)},
\]
whenever $g,h \in \Schwartz(\Z)$ are supported on $[-N_0,N_0]$.
\end{proposition}

\begin{proof} We can assume $N$ is sufficiently large depending on $C_1$, as the claim follows from \eqref{anf} otherwise.
We may also normalize $\|g\|_{\ell^\infty(\Z)} = \|h\|_{\ell^\infty(\Z)} = 1$, so our task is now to show that
$$ \| (1 - \Pi_{\leq l, \leq -\Log N+l}) \tilde A_N^*(h,g) \|_{\ell^2(\Z)} \lesssim_{C_1} 2^{-c l}  N^{d/2}$$
for some $c>0$ depending on $P$.

We apply Corollary \ref{struct} with $\delta = 2^{-c'l}$ for a sufficiently small $c'>0$ depending only on $P$.  
Because of \eqref{2l} and the hypothesis that $N$ is large, we see from \eqref{m-small} that $(l, -\Log N+l)$ has good major arcs. By choice of $\delta$ and the Fourier support of $F_{\alpha}$, we have from Lemma \ref{iw-prop} that
$$ (1 - \Pi_{\leq l, \leq -\Log N+l}) F_{\alpha} = 0$$
for all $F_{\alpha}$ in the decomposition \eqref{decomp}, and hence
\begin{align}
\label{eq:2}
(1 - \Pi_{\leq l, \leq -\Log N+l}) \tilde A_N^*(h,g) = (1 - \Pi_{\leq l, \leq -\Log N+l}) E_1+(1 - \Pi_{\leq l, \leq -\Log N+l}) E_2. \qquad
\end{align}
Since 
\[
\|(1 - \Pi_{\leq l, \leq -\Log N+l}) E_2\|_{\ell^2(\Z)}\lesssim \delta
\]
it suffices to show that
\begin{align}
\label{eq:1}
\|(1 - \Pi_{\leq l, \leq -\Log N+l}) E_1\|_{\ell^2(\Z)}\lesssim \langle l \rangle  \delta^{1/4} N^{d/2},
\end{align}
which will give the claim by the choice of $\delta$. We now establish \eqref{eq:1}.

The function $\tilde A_N^*(h,g)$ is bounded in $\ell^\infty(\Z)$ norm by $O(1)$.  From
\eqref{decomp} and the triangle inequality, we thus have
$$  \|E_1\|_{\ell^\infty(\Z)} \lesssim \delta^{-O(1)},$$
since $E_2\in \ell^{q}(\Z)$ for any $2\le q\le  \infty$ and $\|E_2\|_{\ell^q(\Z)}\le \|E_2\|_{\ell^2(\Z)}\le \delta$,
so by interpolation with \eqref{e-bound} we have
\[
\|E_1\|_{\ell^p(\Z)}\lesssim \delta^{1/2} N^{d/p}
\]
for some absolute constant $1 < p < 2$ that is sufficiently close to $1$.  
By the latter bound and Lemma \ref{iw-prop}, we conclude that
\begin{align}
\label{eq:3}
\| (1 - \Pi_{\leq l, \leq -\Log N+l}) E_1 \|_{\ell^p(\Z)} \lesssim_{p} \langle l \rangle \delta^{1/2}  N^{d/p}.
\end{align}
Also, as $\tilde A_N^*(h,g)$ is bounded by $O(1)$ and supported on
$[-N_0,N_0]$ with $N_0\simeq N^d$ we have
$$ \| A_N^*(h,g) \|_{\ell^{p'}(\Z)} \lesssim_{p} N^{d/p'},$$
and thus by Lemma \ref{iw-prop} again
$$ \| (1 - \Pi_{\leq l, \leq -\Log N+l}) \tilde A_N^*(h,g) \|_{\ell^{p'}(\Z)} \lesssim_{p,C_1} \langle l \rangle N^{d/p'},$$
and since $\|E_2\|_{\ell^{p'}(\Z)}\lesssim \delta$ we also have
\[
\| (1 - \Pi_{\leq l, \leq -\Log N+l}) E_2 \|_{\ell^{p'}(\Z)} \lesssim_{p,C_1} \langle l \rangle \delta.
\]
Using these two bounds, the triangle inequality and \eqref{eq:2} we may write
\begin{align}
\label{eq:5}
 \| (1 - \Pi_{\leq l, \leq -\Log N+l}) E_1 \|_{\ell^{p'}(\Z)} \lesssim_{p,C_1} \langle l \rangle N^{d/p'}.
\end{align}
Interpolating, \eqref{eq:3} and \eqref{eq:5} we obtain \eqref{eq:1}, and the proof is completed.
\end{proof}

We now address the second issue, namely that of relaxing the $\ell^\infty(\Z)$ control on $g$
to $\ell^2(\Z)$ control. 
The main tool for this is the following recent
$\ell^p(\Z)$ improving estimate for linear polynomial averages.

\begin{proposition}[$\ell^p(\Z)$-improving]\label{lp-improv}
Let $Q(\nn) \in \Z[\nn]$ be of degree $d \geq 2$.  Then for
every 
\[ 2 \geq p > 
\begin{cases} 2 - \frac{4}{d^2+d+3} &\mbox{if } d \geq 3 \\
2 - \frac{2}{3} & \mbox{if } d=2, \end{cases}\]
one has the bound
\[
\|A_N^{Q(\nn)} f\|_{\ell^2(\Z)} \lesssim_{p,Q} N^{d(\frac{1}{2} - \frac{1}{p})} \|f\|_{\ell^p(\Z)}
\]
for all $N \geq 1$ and $f \in \ell^p(\Z)$.
\end{proposition}

\begin{proof}  This follows from the work of Han--Kova{\v c}--Lacey--Madrid--Yang \cite{HKLMY}.  Indeed, the $d=2$ case is contained\footnote{Strictly speaking, this theorem requires all the coefficients of the quadratic polynomial $Q$ to be non-negative.  However, by applying a reflection $x \mapsto -x$ one can assume without loss of generality that the quadratic coefficient of $Q$ is positive, and then applying a translation $n \mapsto n+c$ for some large positive integer $c$ (noting the pointwise bound $A_N^{Q(\nn)} f \lesssim_c A_N^{Q(\nn+c)} |f|$) one can then deduce the case of general $Q$ from the non-negative coefficient case (perhaps at the risk of worsening the dependence of constants on $Q$).  See also \cite{DDL} for another treatment of the (monomial) quadratic case and an extension to higher dimensions.} in \cite[Theorem 1.6]{HKLMY}, and the $d \geq 3$ case is contained in \cite[Theorem 1.9]{HKLMY}, after specializing these theorems to the $p=2$ case and performing some routine algebra.  Note that \cite[Conjecture 1.5]{HKLMY} predicts that the range of $p$ can be lowered to $p > 2 - \frac{2}{d+1}$ for any value of $d$, but this is currently only known for $d=2$. For our purposes, any exponent $p$ less than $2$ would be sufficient for applications.
\end{proof}

We can now relax the $\ell^\infty(\Z)$ control on $g$ to $\ell^2(\Z)$ control:

\begin{corollary}[Structure of dual function, III]\label{struct-iii}  Under the notation and hypotheses of Proposition \ref{propdual}, one has
\begin{equation}\label{anhg-bound}
\| (1 - \Pi_{\leq l, \leq -\Log N+l}) \tilde A_N^*(h,g) \|_{\ell^2(\Z)} \lesssim_{C_1} 2^{-cl} \|h\|_{\ell^\infty(\Z)}  \|g\|_{\ell^2(\Z)},
\end{equation}
whenever $g,h \in\Schwartz(\Z)$ are supported on $[-N_0,N_0]$.
\end{corollary}

\begin{proof}
From Proposition \ref{propdual} we already have the bound
\[
\| (1 - \Pi_{\leq l, \leq -\Log N+l}) \tilde A_N^*(h,g) \|_{\ell^2(\Z)} \lesssim_{C_1} N^{d/2} 2^{-cl} \|h\|_{\ell^\infty(\Z)}  \|g\|_{\ell^\infty(\Z)}.
\]
On the other hand from \eqref{bkhg} and the triangle inequality, we have the pointwise bound
$$ \tilde A_N^*(h,g)(x) \lesssim \|h\|_{\ell^\infty(\Z)} A^{\nn-P(\nn)}_N |g|(x)$$
and hence by Lemma \ref{iw-prop}(i) and Proposition \ref{lp-improv} applied with $Q(\nn)=\nn-P(\nn)$ we have
\[
\| (1 - \Pi_{\leq l, \leq -\Log N+l}) \tilde A_N^*(h,g) \|_{\ell^2(\Z)} \lesssim_{C_1,p} N^{d(\frac{1}{2}-\frac{1}{p})} \|h\|_{\ell^\infty(\Z)}  \|g\|_{\ell^p(\Z)}
\]
for any  $2-\frac{4}{d^2+d+3}<p\le 2$.  The claim now follows from interpolation.
\end{proof}

Now we use the off-diagonal decay estimate \eqref{off-decay} to remove the support condition:

\begin{corollary}[Structure of dual function, IV]\label{siv} Under the notation and hypotheses of Proposition \ref{propdual}, one has \eqref{anhg-bound}
whenever $g \in \ell^2(\Z)$ and $h \in \ell^\infty(\Z)$.
\end{corollary}

\begin{proof}  If $g$ is supported on an interval $I$ of length $N^d$, then we may restrict $h$ to an $O(N^d)$-neighborhood of $I$ without affecting the average $\tilde A_N^*(h,g)$.  From Corollary \ref{struct-iii} and translation invariance we then conclude that \eqref{anhg-bound} holds in this case.

Now we handle the case when $g$ is not supported in such an interval.  We may normalize $\|h\|_{\ell^\infty(\Z)}=1$. We can split $g = \sum_{I\in\mathcal I} g \ind{I}$ where $I$ ranges over a partition $\mathcal I$ of $\R$ into intervals $I$ of length $N^d$.  Then by the preceding discussion the local dual function $D_I \coloneqq \tilde A_N^*(h,g \ind{I})$ obeys the bound
\begin{equation}\label{anghi}
\| (1 - \Pi_{\leq l, \leq -\Log N+l}) D_I \|_{\ell^2(\Z)} \lesssim_{C_1} 2^{-cl} \| g \|_{\ell^2(I)}
\end{equation}
for each interval $I$, and we wish to establish
$$ \Big\| \sum_{I\in\mathcal I} (1 - \Pi_{\leq l, \leq -\Log N+l}) D_I \Big\|_{\ell^2(\Z)} \lesssim_{C_1} 2^{-cl} \| g \|_{\ell^2(\Z)}$$
(recall $c$ is allowed to vary from line to line).  By squaring and applying Schur's test, it suffices to obtain the decay bound
$$ \langle (1 - \Pi_{\leq l, \leq -\Log N+l}) D_I, (1 - \Pi_{\leq l, \leq -\Log N+l}) D_J \rangle \lesssim_{C_1}
2^{-cl} \left\langle \frac{\mathrm{dist}(I,J)}{N^d} \right\rangle^{-2} \| g \|_{\ell^2(I)} \| g \|_{\ell^2(J)}$$
for all intervals $I,J$ of length $N^d$.  From Cauchy--Schwarz and \eqref{anghi} we already have
$$ \langle (1 - \Pi_{\leq l, \leq -\Log N+l}) D_I, (1 - \Pi_{\leq l, \leq -\Log N+l}) D_J \rangle \lesssim_{C_1}
2^{-cl} \| g \|_{\ell^2(I)} \| g \|_{\ell^2(J)}.$$
On the other hand, $\tilde A_N^*(h,g\ind{I})$ is supported in a $O(N^d)$-neighborhood of $I$, and similarly for $\tilde A_N^*(h,g\ind{J})$; also, $2^{\Log N-l} \lesssim N^d$.  From Lemma \ref{iw-prop}(i) and Cauchy--Schwarz followed by \eqref{anf} we thus have
\begin{multline*}
\langle (1 - \Pi_{\leq l, \leq -\Log N+l}) D_I, (1 - \Pi_{\leq l, \leq -\Log N+l}) D_J \rangle \\
\lesssim_{C_1}
    \langle l \rangle^{O(1)} \left\langle \frac{\mathrm{dist}(I,J)}{N^d} \right\rangle^{-10} \|D_I\|_{\ell^2(\Z)} \|D_J\|_{\ell^2(\Z)}\\
    \lesssim_{C_1} \langle l \rangle^{O(1)} \left\langle \frac{\mathrm{dist}(I,J)}{N^d} \right\rangle^{-10} \|g\|_{\ell^2(I)} \| g\|_{\ell^2(J)}.    
\end{multline*}
Taking geometric means of the two estimates, we obtain the claim.
\end{proof}

We may now prove Theorem \ref{improv}(i).  We may assume that $l,N$ are sufficiently large depending on $C_1$, since the claim follows from \eqref{anf} otherwise. It suffices to prove this claim under the additional hypothesis \eqref{2l} (which one can view as an upper bound on $l$ in terms of $N$), since for larger values of $l$ the hypothesis (i) becomes stronger and the conclusion \eqref{single-decay} is essentially unchanged.  By duality, it now suffices to establish the bound
$$ \langle \tilde A_N(f,g), h \rangle \lesssim_{C_1} 2^{-cl} \|f\|_{\ell^2(\Z)} \|g\|_{\ell^2(\Z)} \|h\|_{\ell^\infty(\Z)}$$
for any $f \in \ell^2(\Z), g \in \ell^2(\Z), h \in \ell^\infty(\Z)$ obeying the hypothesis in Theorem \ref{improv}(i).  From \eqref{transpose} and Lemma \ref{iw-prop} we can write the left-hand side as
$$ \langle (1 - \Pi_{\leq l, \leq -\Log N+l}) \tilde A^*_N(h,g), f \rangle$$
and the claim now follows from Corollary \ref{siv} and Cauchy--Schwarz.

\subsection{Proof of Theorem \ref{improv}(ii)}

Now we turn to the proof of Theorem \ref{improv}(ii).  This will
follow from a similar argument used to prove Theorem \ref{improv}(i), once we establish an analogue of
Proposition \ref{lip} for the function $g$ (with the denominator $N$
in the intervals replaced with $N^d$).  Such a result was obtained very recently in the quadratic case $P = \nn^2$ by Peluse and Prendiville \cite[Corollary 1.4]{PP2}, and the arguments there likely extend to cover all nonlinear polynomials $P$.  We give a derivation here that is self-contained (except for Theorem \ref{pp}, which is used as a ``black box''), inspired by some earlier unpublished notes in this direction by Peluse and Prendiville (private communication).

\begin{proposition}[Alternate inverse theorem for $g$]\label{lip-2}
 Under the hypotheses and notation of Theorem \ref{pp}, there exists a function $G \in \ell^2(\Z)$ with 
    \begin{equation}\label{ooh}
    \|G\|_{\ell^\infty(\Z)} \lesssim 1; \quad \|G\|_{\ell^1(\Z)} \lesssim N^d
    \end{equation}
    and with $\F_{\Z}G$  supported in the $O(\delta^{-O(1)}/N^d)$-neighborhood of some $\alpha \in \Q/\Z$ of naive height $O(\delta^{-O(1)})$ such that
    \begin{equation}\label{gG}
    |\langle g, G \rangle| \gtrsim \delta^{O(1)} N^d.
    \end{equation}
\end{proposition}

\begin{proof} As in the proof of Proposition \ref{lip} we may assume that $N \geq C \delta^{-C}$ for some large constant $C$, as the claim is trivial otherwise.  From \eqref{hyp} and \eqref{transpose}, we have
\begin{align}
\label{eq:14}
 |\langle f, \tilde A^*_N(h,g) \rangle| \geq \delta N^d.
\end{align}
Since $\|f\|_{\ell^2(\Z)}\lesssim N^{d/2}$, we conclude using the Cauchy--Schwarz inequality that
\[
|\langle \tilde A^*_N(h,g), \tilde A^*_N(h,g) \rangle| \gtrsim \delta^2 N^{d}.
\]
We apply Corollary \ref{struct} to the second factor $\tilde A^*_N(h,g)$, with $\delta$ replaced by $c_0 \delta^2$ for some small constant $c_0>0$, to obtain a decomposition
\begin{equation*}
\tilde A_N^*(h,g) = \sum_{\alpha \in \Q/\Z: \Height_\naive(\alpha) \lesssim_{c_0} \delta^{-O(1)}} F_{\alpha} + E_1+E_2,
\end{equation*} 
where each $F_{\alpha} \in \ell^2(\Z)$ has Fourier support in the $1/M$-neighborhood of $\alpha$ with $M \sim_{c_0} \delta^{O(1)} N$ and obeys the bounds
\begin{align}
\label{faq}
\| F_{\alpha}\|_{\ell^\infty(\Z)} \lesssim_{c_0} \delta^{-O(1)};
\quad\text{ and }\quad
\|F_{\alpha}\|_{\ell^1(\Z)} \lesssim_{c_0} \delta^{-O(1)}N^d,
\end{align}
and the error terms $E_1 \in \ell^1(\Z)$ and $E_2 \in \ell^2(\Z)$ obey the bounds
\begin{equation}
\label{eq:23}
\|E_1\|_{\ell^1(\Z)} \leq c_0\delta^2 N^d; \quad\text{ and }\quad
\|E_2\|_{\ell^2(\Z)} \leq c_0\delta^2.
\end{equation}
From \eqref{anf} and H\"older's inequality one has
$$|\langle \tilde A^*_N(h,g), E_1 \rangle|+|\langle \tilde A^*_N(h,g), E_2 \rangle| \lesssim c_0 \delta^2 N^{d}$$
hence if $c_0$ is small enough we conclude from the triangle inequality and pigeonhole principle that
\[
|\langle \tilde A^*_N(h,g), F_{\alpha} \rangle| \gtrsim \delta^{O(1)} N^d
\]
for some $\alpha \in \Q/\Z$ of naive height $O_{c_0}(\delta^{-O(1)})$. Henceforth we suppress the dependence of constants on $c_0$.  By \eqref{transpose} again, we conclude that
\[
\Big|\sum_{x\in\Z} \E_{n\in [N]} h(x) F_{\alpha}(x-n) g(x-P(n))\Big|
\gtrsim \delta^{O(1)} N^d.
\]
From the Fourier support of $F_{\alpha}$, we have the reproducing formula
\[
F_{\alpha}(x) = \frac{2}{M} \sum_{m\in\Z} F_{\alpha}(x-m) e(-\alpha m) \F_\R^{-1} \eta( 2m/M)
\]
where $\eta$ was defined in Section \ref{cutoff-sec}.
Thus
\[
\Big|\sum_{x\in\Z} \E_{n\in [N]} \sum_{m\in\Z} h(x) F_{\alpha}(x-m-n) e(-\alpha m) g(x-P(n)) \F_\R^{-1} \eta(2m/M)\Big|
\gtrsim \delta^{O(1)} N^{d+1}.
\]
Making the change of variables $s=m+n$, the left-hand side can be rewritten as
\[
\Big|\sum_{x\in\Z} \sum_{s\in \Z} h(x) F_{\alpha}(x-s) e(-\alpha s)  \E_{n\in [N]} e(\alpha n) g(x-P(n)) \F_\R^{-1} \eta(2(s-n)/M)\Big|.
\]
By the rapid decay of $\F_\R^{-1} \eta$ the inner sum can be restricted to $s = O(N)$.  Thus by the pigeonhole principle there exists $s = O(N)$ such that
\[
\Big|\sum_{x\in\Z} h(x) F_{\alpha}(x-s) e(-\alpha s) \E_{n \in [N]} e(\alpha n) g(x-P(n)) \F_\R^{-1} \eta(2(s-n)/M)\Big| \gtrsim \delta^{O(1)} N^{d}.
\]
From \eqref{faq} and the boundedness of $h$ one has
$$ \sum_{x \in \Z} |h(x) F_{\alpha}(x-s) e(-\alpha s)|^2 \lesssim \delta^{-O(1)} N^d$$
hence by the
Cauchy--Schwarz inequality
\[
\sum_{x\in\Z}  \left|\E_{n \in [N]} e(\alpha n) g(x-P(n)) \F_\R^{-1} \eta(2(s-n)/M)\right|^2\gtrsim \delta^{O(1)} N^{d}.
\]
By Plancherel's theorem, we can write the left-hand side as
\begin{align}
\label{eq:24}
\int_{\TT} |\F_\Z g(\xi)|^2 |S_N(\xi)|^2d\xi,
\end{align}
where $S_N$ is the normalized exponential sum
\[
S_N(\xi) \coloneqq \E_{n \in [N]} e(\alpha n) e(\xi P(n)) \F_\R^{-1} \eta(2(s-n)/M).
\]
By another appeal to Plancherel's theorem, one has
$$ \int_{\TT} |\F_\Z g(\xi)|^2d\xi=\|g\|^2_{\ell^2(\Z)} \lesssim N^d,$$
thus one must have
\[
\int_{\Omega} |\F_\Z g(\xi)|^2 |S_N(\xi)|^2d\xi
\gtrsim \delta^{O(1)} N^{d}
\]
for a set $\Omega \subseteq \TT$ of the form
\[
\Omega \coloneqq \{\xi\in\TT\colon |S_N(\xi)|\gtrsim \delta^{O(1)}\}.
\]

By the inverse form of Weyl's exponential sum estimate, see the argument as in \cite[Lemma A.11, pp. 1922]{GT}, we obtain
\[
\Omega\subseteq \pi( [-1/M', 1/M'] \times \{ \alpha' \in \Q/\Z: \Height_\naive(\alpha') \lesssim \delta^{-O(1)} \} )
\]
for some $M' \sim \delta^{O(1)} N^d$.  By the pigeonhole principle, we may therefore find $\alpha' \in \Q/\Z$ of naive height $O(\delta^{-O(1)})$ such that
\[
 \int_{\alpha'-1/M'}^{\alpha'+1/M'} |\F_\Z g(\xi \mod 1)|^2d\xi \gtrsim \delta^{O(1)} N^{d}.
\]
By Plancherel's theorem this implies that
\[
 \sum_{x\in\Z} \bigg|\frac{1}{M'} \sum_{m\in\Z} g(x-m) e(-\alpha'm) \F_\R^{-1} \eta\left(\frac{2m}{M'}\right)\bigg|^2 \gtrsim \delta^{O(1)} N^{d}
\]
so that \eqref{gG} holds with
\[
 G(x) \coloneqq \frac{1}{(M')^2} \sum_{m\in\Z}\sum_{m'\in\Z} g(x-m+m') e(-\alpha'(m-m')) \F_\R^{-1} \eta\left(\frac{2m}{M'}\right) \F_\R^{-1} \eta\left(\frac{2m'}{M'}\right).
\]
A routine calculation reveals that $G$ has Fourier support in the $2/M'$-neighborhood of $\alpha'$ and obeys the bounds
\[
\| G\|_{\ell^\infty(\Z)} \lesssim 1;
\quad \text{ and } \quad
\|G\|_{\ell^1(\Z)} \lesssim N^d,
\]
and the claim follows.
\end{proof}

We can now repeat all of the previous arguments with the role of $f$
now played by $g$, and with the spatial scale $N$ replaced by
$N^d$.  For the convenience of the reader we state the analogous
key propositions. Repeating the Hahn--Banach proof of Corollary
\ref{struct}, but using Proposition \ref{lip-2} in place of Proposition 
\ref{lip}, we conclude:

\begin{corollary}[Structure of second dual function, I]\label{struct-dual}
Let the notation and hypotheses be as in Corollary \ref{struct}.  Then there exists a decomposition
\begin{equation}\label{decomp-dual}
\tilde A_N^{**}(f,h) = \sum_{\alpha \in \Q/\Z: \Height_\naive(\alpha) \lesssim \delta^{-O(1)}} F_{\alpha} + E_1+E_2,
\end{equation} 
where each $F_{\alpha} \in \ell^2(\Z)$ has Fourier transform supported in the $O(\delta^{-O(1)}/N^d)$-neighborhood of $\alpha$
and obeys the bounds from \eqref{faq-bound}, and the error terms $E_1 \in \ell^1(\Z)$ and $E_2 \in \ell^2(\Z)$ obey the bounds from  \eqref{e-bound}.
\end{corollary}

Repeating the proof of Proposition \ref{propdual}, we conclude:

\begin{proposition}[Structure of second dual function, II]\label{propdual-second}
Let the notation and hypotheses be as in Proposition \ref{propdual}. Then
\[
\|  (1 - \Pi_{\leq l, \leq -d\Log N+dl}) \tilde A_N^{**}(f,h) \|_{\ell^2(\Z)} \lesssim_{C_1} 2^{-cl}
 N^{d/2}  \|f\|_{\ell^\infty(\Z)} \|h\|_{\ell^\infty(\Z)},
\]
whenever $f,h \in\Schwartz(\Z)$ are supported on $[-N_0,N_0]$.
\end{proposition}

Repeating the $L^p$-improving argument used to prove Corollary \ref{struct-iii}, we conclude:

\begin{corollary}[Structure of second dual function, III]\label{struct-iii-second}
Under the notation and hypotheses of Proposition \ref{propdual}, one has
\begin{equation}\label{anhg-bound-dual}
\| (1 - \Pi_{\leq l, \leq -d\Log N+dl}) \tilde A_N^{**}(f,h) \|_{\ell^2(\Z)} \lesssim_{C_1} 2^{-cl} \|f\|_{\ell^2(\Z)} \|h\|_{\ell^\infty(\Z)},
\end{equation}
whenever $f,h \in\Schwartz(\Z)$ are supported on $[-N_0,N_0]$.
\end{corollary}

Finally, we repeat the off-diagonal estimate argument used to prove Corollary \ref{siv} to conclude:

\begin{corollary}[Structure of second dual function, IV]\label{siv-second}
Under the notation and hypotheses of Proposition \ref{propdual}, one has \eqref{anhg-bound-dual}
whenever $f \in \ell^2(\Z)$ and $h \in \ell^\infty(\Z)$.
\end{corollary}

Theorem \ref{improv}(ii) now follows by repeating the proof of Theorem \ref{improv}(i).

\section{Approximation by model operators}\label{approx-sec}

To conclude the proof of Theorem \ref{main}, we need to establish Theorem \ref{varp}.  
Let $l_1,l_2 \in\N$, and define $l,u$ by \eqref{l-def}, \eqref{u-def} respectively.  Fix $s_1,s_2 \geq -u$.  In view of Proposition \ref{high-high-l2} we may assume that at least one of $s_1=-u$, $s_2=-u$, $(p_1,p_2) \neq (2,2)$ holds.  It will be convenient to adopt the following definition.  If $\G=\Z$ or $\G = \A_\Z$, we declare a tuple $(H_N)_{N \in \I'}$ of functions $H_N \in L^p(\G)$ to be \emph{acceptable} if one has the estimate
$$ \| ( H_N )_{N \in \I'} \|_{L^p(\G;\V^r)} \lesssim_{C_3} 
\langle \max(l,s_1,s_2) \rangle^{O(1)} 2^{O(\rho l)-c \max(l,s_1,s_2) \ind{p_1=p_2=2}} \|f\|_{\ell^{p_1}(\Z)} \|g\|_{\ell^{p_2}(\Z)}.$$
Our task is thus to show that the tuple
$$ ( \tilde A_N(F_N,G_N) )_{N \in \I}$$
is acceptable.

The main difficulty here is that the scale parameter $N$ affects the average $\tilde A_N(F_N,G_N)$ in three different ways, as the functions $F_N,G_N$ both separately depend on $N$, and the averaging operator $\tilde A_N$ also depends on $N$.  The strategy will be to perform Fourier-analytic manipulations (on the adelic frequency space $\R \times \Q/\Z$) to approximate this expression $\tilde A_N(F_N,G_N)$ by linear combinations of simpler ``model expressions'' $A(\tilde F_N, \tilde G_N)$, where the functions $\tilde F_N, \tilde G_N$ still depend on $N$, but the bilinear averaging operator $A$ is independent of $N$.  In such a setting we will be able to use general arguments (e.g., Rademacher--Menshov type inequalities) to control the variational norms of the bilinear expressions $A(\tilde F_N, \tilde G_N)$ by variational norms of the two linear expressions $\tilde F_N, \tilde G_N$ separately.  These in turn can be controlled by a number of tools, such as the vector-valued Ionescu--Wainger multiplier theorem, Theorem \ref{thm:iw}.

We return to the rigorous arguments.  For any $N \in \I$, we have
\begin{equation}\label{nmax}
N \geq \max( 2^{2^{\max(l,s_1,s_2)/C_0}}, C_3 ),
\end{equation}
which implies in particular that
\begin{equation}\label{N-big}
N \geq 2^{10 du}.
\end{equation}

In contrast, by Lemma \ref{mag-lem}(ii), $(\Q/\Z)_{\leq l}$ is the union of dual cyclic groups $\frac{1}{q}\Z/\Z$ with
\begin{equation}\label{q-small}
 q \leq 2^{u/10}.
 \end{equation}
Thus $N$ is going to be far larger than any single denominator $q$ arising in the major arcs.  If one wishes to contain $(\Q/\Z)_{\leq l}$ in a single dual cyclic group
$\frac{1}{Q}\Z/\Z$, Lemma \ref{mag-lem}(ii) permits one to do this with
\begin{equation}\label{Q-weak}
 Q = Q_{\leq l} \leq 2^{2^{u/10}}.
\end{equation}
Thus $N$ may or may not be significantly larger than this $Q$.  We will later separate $N$ into large and small scales in order to exploit this containment in the large scale case.

From \eqref{N-big} we also have
$$ -\Log N + l_{(N)} < -10u.$$
From \eqref{m-small} we see that the pair $(l,-u)$ has good major arcs.
This lets us factor the expressions $F_N, G_N$ using the symbol calculus \eqref{func-calc-l}.  Indeed, if we set
$$ F \coloneqq \Pi_{l_1, \leq -u} f; \quad G \coloneqq \Pi_{l_2, \leq -u} g$$
then from \eqref{iw}, \eqref{FN-def}, \eqref{GN-def} we have the identities
$$ F_N = \T^{l_1}_{\varphi_N} F; \quad G_N = \T^{l_2}_{\tilde \varphi_N} G$$
where $\varphi_N, \tilde \varphi_N \in \Schwartz(\R)$ are the bump functions
\begin{equation}\label{varphi-def}
 \varphi_N(\xi) \coloneqq \begin{cases} \eta( 2^{\Log N-s_1} \xi) - \eta( 2^{\Log N-s_1+1} \xi)  & s_1 > -u \\
\eta( 2^{\Log N+u} \xi)  & s_1 = -u
\end{cases}
\end{equation}
and
\begin{equation}\label{varphip-def}
\tilde \varphi_N(\xi) \coloneqq \begin{cases} \eta( 2^{d(\Log N-s_2)} \xi) - \eta( 2^{d(\Log N-s_2+1)} \xi) & s_2 > -u \\
\eta( 2^{d(\Log N+u)} \xi)  & s_2 = -u.
\end{cases}
\end{equation}
From Lemma \ref{iw-prop} we have
\begin{equation}\label{FG-bound}
\| F\|_{\ell^p(\Z)} \lesssim \langle l \rangle \|f\|_{\ell^p(\Z)}; \quad \|G\|_{\ell^{p'}(\Z)} \lesssim \langle l \rangle \|g\|_{\ell^{p'}(\Z)}
\end{equation}
hence we may replace $f,g$ by $F,G$ respectively in the definition of acceptability.  It will now suffice to show that the tuple
\begin{equation}\label{all-all-2}
( \tilde A_N( \T^{l_1}_{\varphi_N} F, \T^{l_2}_{\tilde \varphi_N} G ))_{N \in \I}
\end{equation}
is acceptable.

The dependence on $N$ has not yet materially improved, as the quantity $\tilde A_N( \T^{l_1}_{\varphi_N} F, \T^{l_2}_{\tilde \varphi_N} G )$ still depends on $N$ in three different ways.  However, we can clarify the dependence on $N$ by (adelic) Fourier analysis.  From Example \ref{avg-mult} and \eqref{func-calc}, we see that
$$
\tilde A_N( \T^{l_1}_{\varphi_N} F, \T^{l_2}_{\tilde \varphi_N} G ) = \B_{\Proj^{\otimes 2} m^{l_1,l_2}_N}(F,G),$$
where the symbol $m^{l_1,l_2}_N \colon (\R \times \Q/\Z)^2 \to \C$ is defined by the formula
\begin{multline*}
    m^{l_1,l_2}_N( (\xi_1,\alpha_1), (\xi_2,\alpha_2) ) \\
    \coloneqq \ind{\Height(\alpha_1)=2^{l_1}} \ind{\Height(\alpha_2)=2^{l_2}} \varphi_N(\xi_1)\tilde \varphi_N(\xi_2)
\E_{n \in [N]} e((\alpha_1+\xi_1) n + (\alpha_2+\xi_2) P(n)) \ind{n>N/2}.
\end{multline*}
From \eqref{N-big}, \eqref{q-small} we see that $N$ is large compared to the naive heights of $\alpha_1,\alpha_2$, while $\xi'_1,\xi'_2 = O(2^{-u})$ are small on the support of $m^{l_1,l_2}_N$.  This suggests that in the regimes of interest the symbol
\begin{align*}
\E_{n \in [N]} e((\alpha_1+\xi_1) n + (\alpha_2+\xi_2) P(n)) \ind{n>N/2}
\end{align*}
has an approximate factorization
\begin{equation}\label{mnsn}
m_{\hat \Z}(\alpha_1,\alpha_2) \tilde m_{N,\R}(\xi_1,\xi_2),
\end{equation}
where $m_{\hat \Z} \colon (\Q/\Z)^2 \to \C$ is the normalized exponential sum
$$ m_{\hat \Z}(\alpha_1,\alpha_2) \coloneqq \int_{\hat \Z} e( \alpha_1 x + \alpha_2 P(x))\ d\mu_{\hat \Z}(x),$$
where $\mu_{\hat \Z}$ is the probability Haar measure on the profinite integers $\hat \Z$,
or equivalently
$$ m_{\hat \Z}\left( \frac{a_1}{q} \mod 1, \frac{a_2}{q} \mod 1 \right) = \E_{n \in\Z/q\Z} e\left( \frac{a_1 n + a_2 P(n)}{q} \right)$$
for any $q \in \Z_+$ and $a_1,a_2 \in \Z$,  and $\tilde m_{N,\R} \colon \R^2 \to \C$ is the oscillatory integral
\begin{equation}\label{mn-def}
\tilde m_{N,\R}(\xi_1,\xi_2) \coloneqq \frac{1}{N} \int_{N/2}^N e(\xi_1 t + \xi_2 P(t))\ dt = \int_{1/2}^1 e(\xi_1 N t + \xi_2 P(Nt))\ dt.
\end{equation}
Note how the use of the upper averaging operators $\tilde A_N$ instead of $A_N$ allows us to keep $t$ bounded away from zero, which will be technically convenient later in the argument when we integrate by parts in $t$ (as we now avoid the stationary points of $P$).  The approximation \eqref{mnsn} can be compared with \eqref{approx-factor}.

The heuristic \eqref{mnsn} then suggests the adelic bilinear symbol $m^{l_1,l_2}_N \in \Schwartz( (\R \times \Q/\Z)^2 )$ approximately factors into the tensor product of a continuous bilinear symbol 
$$(\varphi_N \otimes \tilde \varphi_N) \tilde m_{N,\R} \in \Schwartz(\R^2)$$ 
and the arithmetic bilinear symbol
$$ m_{l_1,l_2,\hat \Z} \coloneqq ( \ind{(\Q/\Z)_{l_1}} \otimes \ind{(\Q/\Z)_{l_2}} ) m_{\hat \Z} \in \Schwartz((\Q/\Z)^2).$$
At the level of bilinear Fourier multipliers, this factorization suggests the approximation
$$ \tilde A_N( \T^{l_1}_{\varphi_N} F, \T^{l_2}_{\tilde \varphi_N} G ) \approx \B^{l_1,l_2,m_{\hat \Z}}_{(\varphi_N \otimes \tilde \varphi_N)\tilde m_{N,\R} }(F,G)$$
where we introduce the twisted bilinear Fourier multiplier operators
\begin{equation}\label{bil-twist}
\B^{l_1,l_2,m_{\hat \Z}}_{m} \coloneqq \B_{\Proj^{\otimes 2}( m \otimes m_{l_1,l_2,\hat \Z} )}
\end{equation}
for any $m \in \Schwartz( \R^2 )$.
More explicitly, one has    
\begin{multline*}
    \B^{l_1,l_2,m_{\hat \Z}}_{m}(f,g)(x) = \sum_{\alpha_1 \in (\Q/\Z)_{l_1}, \alpha_2 \in (\Q/\Z)_{l_2}} m_{\hat \Z}(\alpha_1,\alpha_2) \\
\quad \times \int_{\R^2} m(\xi_1,\xi_2) \F_\Z f(\alpha_1+\xi_1) \F_\Z g(\alpha_2+\xi_2) e(-x(\alpha_1+\alpha_2+\xi_1+\xi_2))\ d\xi_1 d\xi_2.
\end{multline*}

\begin{remark} Another way to think about the approximation \eqref{mnsn} is that it is approximating the discrete averaging operator $\tilde A_N \colon \Schwartz(\Z) \times \Schwartz(\Z) \to \Schwartz(\Z)$ by the adelic averaging operator $\tilde A_{N,\A_\Z} \colon \Schwartz(\A_\Z) \times \Schwartz(\A_\Z) \to \Schwartz(\A_\Z)$ defined by
\begin{equation}\label{adelic-average}
\tilde A_{N,\A_\Z}(f,g)(x) \coloneqq \frac{1}{N} \int_{[N/2,N] \times \hat \Z} f(x-y) g(x-P(y))\ d\mu_{\A_\Z}(y),
\end{equation}
which is in turn the tensor product of the continuous averaging operator $\tilde A_{N,\R} \colon \Schwartz(\R) \times \Schwartz(\R) \to \Schwartz(\R)$ defined by
$$ \tilde A_{N,\R}(f,g)(x) \coloneqq \frac{1}{N} \int_{N/2}^N f(x-t) g(x-P(t))\ dt,$$
and the arithmetic averaging operator $A_{\hat \Z} \colon \Schwartz(\hat \Z) \times \Schwartz(\hat \Z) \to \Schwartz(\hat \Z)$ defined in Example \ref{avg-mult}.  As we shall see, this approximation is particularly accurate in the large-scale regime when $N$ is large compared to the quantity $Q_{\leq l}$, see \eqref{def:Q_l}.  In fact the main estimate \eqref{var-poly-main-int} on the integers $\Z$ has a natural analogue on the adelic integers $\A_{\Z}$ which can be proven by the same methods (with several simplifications), and our proof of the integer estimate was discovered by first working with the adelic operator (or more precisely, a projection of this operator to $\R \times \Z/Q\Z$) as a model case.  This suggests that a natural route to prove other harmonic analysis estimates on the integers $\Z$ is to first study the analogous estimates on $\A_\Z$ or $\R \times \Z/Q\Z$ as model cases, in order to exploit the tensor product structure.
\end{remark}

We now make the above heuristic precise.  For future applications we make the approximation slightly more general than what is needed in the current step.

\begin{proposition}[Major arc approximation of $\tilde A_N$]\label{mod-approx}  For any $N \geq 1$ and $s \in \N$ with $-\Log N+s \leq -u$, we have
\begin{multline}\label{norma}
\left\| \tilde A_{N}\left( \Pi_{l_1, \leq -\Log N+s} \tilde F, \Pi_{l_2, \leq -d\Log N+ds} \tilde G \right) - \B^{l_1, l_2, m_{\hat \Z}}_{(\eta_{\leq -\Log N+s} \otimes \eta_{\leq -d\Log N+ds})\tilde m_{N,\R} }(\tilde F,\tilde G) \right\|_{\ell^p(\Z)} \\
\quad \lesssim_{C_3} 2^{O(\max(2^{\rho l},s))} N^{-1}
\| \tilde F \|_{\ell^{p_1}(\Z)} \|\tilde G\|_{\ell^{p_2}(\Z)}
\end{multline}
for all $\tilde F \in \ell^{p_1}(\Z), \tilde G \in \ell^{p_2}(\Z)$.
\end{proposition}

The key point here is the gain of $N^{-1}$ on the right-hand side, which in practice will make any expression estimated using this proposition acceptable (with room to spare).

\begin{proof}  From the same sort of calculations used in the  preceding heuristic discussion, we can expand the expression inside the norm of the left-hand side \eqref{norma} as
$$ \B_{\Proj^{\otimes 2} M}(\tilde F, \tilde G),$$
where the symbol $M \in \Schwartz( (\R \times \Q/\Z)^2)$ is defined by
\begin{multline*}
M( (\alpha_1,\xi_1), (\alpha_2,\xi_2)) \\
\coloneqq \ind{\Height(\alpha_1) = 2^{l_1}} \ind{\Height(\alpha_2) = 2^{l_2}} \eta_{\leq -\Log N+s}(\xi_1) \eta_{\leq -d\Log N+ds}(\xi_2) M_0( (\alpha_1,\xi_1), (\alpha_2,\xi_2) )    
\end{multline*}
with
\begin{multline*}
M_0( (\alpha_1,\xi_1), (\alpha_2,\xi_2)) \\
\coloneqq \E_{n \in [N]} e(\alpha_1 n + \alpha_2 P(n)) e(\xi_1 n + \xi_2 P(n)) \ind{n>N/2}  - m_{\hat \Z}(\alpha_1,\alpha_2) \tilde m_{N,\R}(\xi_1,\xi_2).    
\end{multline*}
Applying Lemma \ref{crude-mult} with $r_1 \coloneqq N^{-1}$ and $r_2 \coloneqq N^{-d}$, Lemma \ref{mag-lem}(iii), and the triangle inequality, as well as the Leibniz rule, it now suffices to establish the bounds
$$
\frac{\partial^{j_1}}{\partial \xi_1^{j_1}} \frac{\partial^{j_2}}{\partial \xi_2^{j_2}} M_0( (\alpha_1,\xi_1), (\alpha_2,\xi_2))
\lesssim_{C_3} 2^{O(\max(2^{\rho l},s))} N^{j_1+dj_2-1}$$
for $0 \leq j_1,j_2 \leq 2$, $\alpha_1 \in (\Q/\Z)_{l_1}$, $\alpha_2 \in (\Q/\Z)_{l_2}$, and $\xi_1 = O( 2^{s}/N)$, $\xi_2 = O( 2^{ds}/N^d)$.

By Lemma \ref{mag-lem}(ii), the sequence $n \mapsto e(\alpha_1 n + \alpha_2 P(n))$ is periodic with some period $q = O_\rho( 2^{O(2^{\rho l})})$.
Splitting into residue classes modulo $q$, and evaluating the derivatives, it suffices by the triangle inequality to show that
$$ \sum_{n \in [N] \backslash [N/2]} w(n)\ind{n=a \mod q}  - \frac{1}{q} \int_{N/2}^{N} w(t)\ dt
\lesssim_{C_3} 2^{O(\max(2^{\rho l},s))} N^{j_1+dj_2} $$
for all $a \in [q]$, where 
$$ w(t) \coloneqq e(\xi_1 t + \xi_2 P(t)) t^{j_1} P(t)^{j_2}.$$
It suffices to show that
$$ w(n) - \frac{1}{q} \int_n^{n+q}
w(t)\ dt \lesssim_{C_3} 2^{O(\max(2^{\rho l},s))} N^{j_1+dj_2-1}$$
for all $n \in [N] \backslash [N/2]$, since the claim then follows by summing over all $n \in [N] \backslash [N/2]$ with $n=a \mod q$ and using the triangle inequality to estimate the remainder.  By the fundamental theorem of calculus, it then suffices to establish the bound
$$ \frac{d}{dt} w(t) \lesssim_{C_3} 2^{O(\max(2^{\rho l},s))} N^{j_1+dj_2-1} $$
for $t \sim N$; but this follows from the hypotheses $\xi_1 = O(2^{s}/N)$, $\xi_2 = O(2^{ds}/N^d)$, and direct calculation.
\end{proof}

Applying this proposition with $\tilde F \coloneqq \T^{l_1}_{\varphi_N} F$, $\tilde G \coloneqq \T^{l_2}_{\tilde \varphi_N} G$, and $s \coloneqq \max(0,s_1,s_2)+1$, and using the functional calculus and Lemma \ref{iw-prop}, we conclude that
$$ \| \tilde A_N( \T^{l_1}_{\varphi_N} F, \T^{l_2}_{\tilde \varphi_N} G ) - \B^{l_1,l_2,m_{\hat \Z}}_{ (\varphi_N \otimes \tilde \varphi_N)\tilde m_{N,\R}}(F,G)  \|_{\ell^p(\Z)}
\lesssim_{C_3} 2^{O(\max(2^{\rho l},s_1,s_2))} N^{-1} \|F\|_{\ell^{p_1}(\Z)} \|G\|_{\ell^{p_2}(\Z)}.
$$
From \eqref{nmax} we certainly have
$$ 2^{O(\max(2^{\rho l},s_1,s_2))} \sum_{N \in \I} N^{-1} \lesssim_{C_3} \langle \max(l,s_1,s_2) \rangle^{O(1)} 2^{O(\rho l)-c \max(l,s_1,s_2) \ind{p_1=p_2=2}} $$
and thus by \eqref{varsum}, \eqref{FG-bound} we see that the tuple
$$
( \tilde A_N( \T^{l_1}_{\varphi_N} F, \T^{l_2}_{\tilde \varphi_N} G ) - \B^{l_1,l_2,m_{\hat \Z}}_{(\varphi_N \otimes \tilde \varphi_N)\tilde m_{N,\R} }(F,G) )_{N \in \I} 
$$
is acceptable.  Thus by the triangle inequality, the acceptability of \eqref{all-all-2} is equivalent to the acceptability of
\begin{equation}\label{all-all-3}
( \B^{l_1,l_2,m_{\hat \Z}}_{(\varphi_N \otimes \tilde \varphi_N)\tilde m_{N,\R} }(F,G))_{N \in \I}.
\end{equation}
From \eqref{simple} it suffices to prove the acceptability of the two subtuples
\begin{equation}\label{all-all-3-sub}
( \B^{l_1,l_2,m_{\hat \Z}}_{(\varphi_N \otimes \tilde \varphi_N)\tilde m_{N,\R} }(F,G))_{N \in \I_\sml},\quad  ( \B^{l_1,l_2,m_{\hat \Z}}_{(\varphi_N \otimes \tilde \varphi_N)\tilde m_{N,\R} }(F,G))_{N \in \I_\lrg},
\end{equation}
where 
\begin{equation}\label{small-def}
\I_{\sml} \coloneqq \{ N \in \I: N \leq 2^{2^{u}} \}
\end{equation}
is the set of ``small scales'', and
\begin{equation}\label{large-def}
\I_{\lrg} \coloneqq \{ N \in \I: N > 2^{2^{u}} \}.
\end{equation}
is the set of ``large scales''. As we shall see, for the small scales one will be able to tolerate the (doubly) logarithmic losses arising from Rademacher--Menshov arguments, and for the large scales one will be able to exploit \eqref{Q-weak} to replace the integers $\Z$ by the adelic integers $\A_\Z$.

At this stage the bilinear operator $\B^{l_1,l_2,m_{\hat \Z}}_{(\varphi_N \otimes \tilde \varphi_N)\tilde m_{N,\R} }$ still has a symbol that depends on $N$, although at least the dependence is now confined to the continuous frequency variables and not the arithmetic ones.  To simplify the dependence further, we observe from \eqref{func-calc-2} that we have the functional calculus
\begin{equation}\label{func-calc-3}
    \B^{l_1,l_2,m_{\hat \Z}}_{ (\varphi_1 \otimes \varphi_2)m}(f,g) = \B^{l_1, l_2, m_{\hat \Z}}_{m}( \T^{l_1}_{\varphi_1} f, \T^{l_2}_{\varphi_2} g)
\end{equation}
whenever $\varphi_1, \varphi_2 \in \Schwartz(\R_{\leq -u})$ and $m \in \Schwartz(\R_{\leq -u}^2)$.  From this calculus and the definition \eqref{mn-def} of $\tilde m_{N,\R}$, we can factor $\B^{l_1,l_2,m_{\hat \Z}}_{(\varphi_N \otimes \tilde \varphi_N)\tilde m_{N,\R} }(F,G)$ as
\begin{equation}\label{bn-ident}
\B^{l_1,l_2,m_{\hat \Z}}_{(\varphi_N \otimes \tilde \varphi_N)\tilde m_{N,\R} }(F,G) = \int_{1/2}^1 \B^{l_1,l_2,m_{\hat \Z}}_{m_*}( \T^{l_1}_{\varphi_{N,t}} F, \T^{l_2}_{\tilde \varphi_{N,t}} G)\ dt
\end{equation}
where $\varphi_{N,t}, \tilde \varphi_{N,t} \in \Schwartz(\R)$ are modulated variants of $\varphi_N, \tilde \varphi_N$ defined by the formulae
\begin{align}
    \varphi_{N,t}(\xi) &\coloneqq \varphi_N(\xi) e( N t \xi ) \label{vnt-def}\\
    \tilde \varphi_{N,t}(\xi) &\coloneqq \tilde \varphi_N(\xi) e( P(N t) \xi ) \label{tvnt-def}
\end{align}
and $m_* \in \Schwartz(\R^2)$ is the symbol
$$ m_* \coloneqq \eta_{\leq -2u} \otimes \eta_{\leq -2du}.$$
The advantage of this formulation \eqref{bn-ident} is that the bilinear operator $\B^{l_1, l_2, m_{\hat \Z}}_{m_*}$ is independent of $N$.  This is particularly useful in the small-scale case $N \in \I_{\sml}$, as it will let us control variational norms of bilinear expressions in terms of linear quantities via a two-parameter version of the Rademacher--Menshov inequality.

In the large-scale case $N \in \I_{\lrg}$ we can express \eqref{bn-ident} in another useful way. Introduce the adelic model functions $F_\A \in L^{p_1}(\A_\Z)$, $G_\A \in L^{p_2}(\A_\Z)$
by the formulae
\begin{equation}\label{vecf-def}
F_\A(x,y) \coloneqq \sum_{\alpha_1 \in (\Q/\Z)_{l_1}} \int_\R \eta_{\leq -2^{u-1}}(\xi_1) \F_\Z F(\alpha_1 + \xi_1) e( -(\xi_1,\alpha_1) \cdot (x,y) ) \ d\xi_1
\end{equation}
and
\begin{equation}\label{vecg-def}
G_\A(x,y) \coloneqq \sum_{\alpha_2 \in (\Q/\Z)_{l_2}} \int_\R \eta_{\leq -2^{u-1}}(\xi_2) \F_\Z G(\alpha_2 + \xi_2) e( -(\xi_2,\alpha_2) \cdot (x,y) ) \ d\xi_2
\end{equation}
for $x \in \R, y \in \hat \Z$, or equivalently on the Fourier side
\begin{align*}
    \F_{\A_{\Z}} F_\A(\xi_1, \alpha_1) &= \ind{\Height(\alpha_1)=2^{l_1}} \eta_{\leq -2^{u-1}}(\xi_1) \F_\Z F(\alpha_1 + \xi_1) \\
    \F_{\A_{\Z}} G_\A(\xi_2, \alpha_2) &= \ind{\Height(\alpha_2)=2^{l_2}} \eta_{\leq -2^{u-1}}(\xi_2) \F_\Z G(\alpha_2 + \xi_2)
\end{align*}
for $\xi_1,\xi_2 \in \R$ and $\alpha_1,\alpha_2 \in \Q/\Z$.  (One can use Lemma \ref{crude-mult} to verify that $F_\A$ does indeed lie in $L^p(\A_\Z)$, and similarly for $G_\A$.)  One can also interpret $F_\A, G_\A$ as the interpolated functions 
$$F_\A = \Sample_{\R_{\leq -2^{u-1}} \times (\Q/\Z)_{l_1}}^{-1} \Pi_{l_1,\leq -2^{u-1}} F, \quad G_\A = \Sample_{\R_{\leq -2^{u-1}} \times (\Q/\Z)_{l_2}}^{-1} \Pi_{l_2,\leq -2^{u-1}} G.$$  

In the large-scale case, $\eta_{\leq -2^{u-1}}$ equals $1$ on the support of $\varphi_{N,t}$, $\tilde \varphi_{N,t}$, and $m_*$ equals $1$ on the support of $\eta_{\leq -2^{u-1}}\otimes \eta_{\leq -2^{u-1}}$, and one can then describe various combinations of $F,G$ as applications of the sampling operator $\Sample$ to various combinations of $F_\A,G_\A$.  More precisely, one observes the identities
\begin{align}
\Pi_{l_1,\leq -2^{u-1}} F &= \Sample F_\A, \label{fa}\\
\Pi_{l_2,\leq -2^{u-1}} G &= \Sample G_\A, \label{ga}\\
\T^{l_1}_{\varphi_{N,t}} F &= \Sample \T_{\varphi_{N,t} \otimes 1} F_\A, \nonumber \\
\T^{l_2}_{\tilde \varphi_{N,t}} G &= \Sample \T_{\tilde \varphi_{N,t} \otimes 1} G_\A, \nonumber\\
\B^{l_1,l_2,m_{\hat \Z}}_{m_*}( \T^{l_1}_{\varphi_{N,t}} F, \T^{l_2}_{\tilde \varphi_{N,t}} G) &= \Sample \B_{1 \otimes m_{l_1,l_2,\hat \Z}}(\T_{\varphi_{N,t} \otimes 1} F_\A,\T_{\tilde \varphi_{N,t} \otimes 1} G_\A)\nonumber
\end{align} 
so that \eqref{bn-ident} can now be written as
$$ \Sample \int_{1/2}^1\B_{(\varphi_{N,t} \otimes \tilde \varphi_{N,t}) \otimes m_{l_1,l_2,\hat \Z}}(F_\A,G_\A)\ dt.$$
All functions on $\A_\Z$ here have Fourier support in the region $(\R_{\leq -2^{u-1}} \times (\Q/\Z)_{l_1}) \times (\R_{\leq -2^{u-1}} \times (\Q/\Z)_{l_2})$, which by Lemma \ref{mag-lem}(ii) is contained in $(\R_{\leq -2^{u-1}} \times (\frac{1}{Q_{\leq l}} \Z/\Z)) \times (\R_{\leq -2^{u-1}} \times (\frac{1}{Q_{\leq l}}\Z/\Z))$.  In this large-scale regime, this is a regime in which Theorem \ref{Sampling} applies, thanks to \eqref{Q-weak}.  In particular, from Theorem \ref{Sampling} (using the normed vector space $\V^r$) we have
\begin{multline}\label{large-var}
\| ( \B^{l_1, l_2, m_{\hat \Z}}_{(\varphi_N \otimes \tilde \varphi_N)\tilde m_{N,\R} }(F,G))_{N \in \I_{\lrg}} \|_{\ell^p(\Z;\V^r)}\\
\sim 
\Big\| \Big(\int_{1/2}^1 \B_{1 \otimes m_{\hat \Z}}(\T_{\varphi_{N,t} \otimes 1} F_\A,\T_{\tilde \varphi_{N,t} \otimes 1} G_\A)\ dt\Big)_{N \in \I_{\lrg}} \Big\|_{L^p(\A_\Z; \V^r))};
\end{multline}
similarly from \eqref{fa}, \eqref{ga}, \eqref{FG-bound}, Theorem \ref{Sampling}, and Lemma \ref{iw-prop} one has
\begin{equation}\label{FGA-bound}
\| F_\A \|_{L^{p_1}(\A_\Z)} \lesssim \langle l \rangle^{O(1)} \|f\|_{\ell^{p_1}(\Z)}; \quad \| G_\A \|_{L^{p_2}(\A_\Z)} \lesssim \langle l \rangle^{O(1)} \|g\|_{\ell^{p_2}(\Z)}.
\end{equation}

In view of the above discussion (and Proposition \ref{high-high-l2}), Theorem \ref{varp} (and hence Theorem \ref{main}) now reduces to establishing the following estimates.

\begin{theorem}[Model operator estimates, I]\label{model-est} Suppose that at least one of $s_1=-u$, $s_2=-u$, or $p \neq 2$ holds.  Then the small-scale model tuple
    \begin{equation}\label{other-small}
     \Big( \int_{1/2}^1 \B^{l_1,l_2,m_{\hat \Z}}_{m_*}( \T^{l_1}_{\varphi_{N,t}} F, \T^{l_2}_{\tilde \varphi_{N,t}} G)\ dt \Big)_{N \in \I_{\sml}} 
    \end{equation}
and the large-scale model tuple
\begin{equation}\label{all-all-large}
 \Big(\int_{1/2}^1 \B_{1 \otimes m_{\hat \Z}}(\T_{\varphi_{N,t} \otimes 1} F_\A,\T_{\tilde \varphi_{N,t} \otimes 1} G_\A) \Big)_{N \in \I_{\lrg}} 
\end{equation}
are both acceptable.
\end{theorem}

It remains to establish Theorem \ref{model-est}.  One difficulty in this theorem is the need to obtain some decay in $s_1,s_2$ when they are large.  Our main tool for doing this will be the following integration by parts identity. For $j_1,j_2=-1,0,+1$ with $(s_1,j_1),(s_2,j_2) \neq (-u,-1)$, we define the modified bump functions
\begin{equation}\label{pntj}
\varphi_{N,t,j_1}(\xi_1) \coloneqq (2^{-s_1} N \xi_1)^{j_1} \varphi_{N,t}(\xi_1) =  (2^{-s_1} N \xi_1)^{j_1} e(Nt \xi_1) \varphi_N(\xi_1)
\end{equation}
and
\begin{equation}\label{tpntj}
\tilde \varphi_{N,t,j_2}(\xi_2) \coloneqq (2^{-ds_2} N^d \xi_2)^{j_2} \tilde \varphi_{N,t}(\xi_2) =  (2^{-ds_2} N^d \xi_2)^{j_2} e(P(Nt) \xi_2) \tilde \varphi_N(\xi_2).
\end{equation}
Note it is necessary to exclude the cases $(s_1,j_1),(s_2,j_2) = (-u,-1)$ to prevent these functions from developing a singularity at the frequency origin.  

\begin{lemma}[Integration by parts identity] \label{integ-ident}\ 
\begin{itemize}
    \item[(i)] If $s_1 > -u$ then we have
$$ \int_{1/2}^1 \varphi_{N,t} \otimes \tilde \varphi_{N,t}\ dt = 
\frac{2^{-s_1}}{2\pi i} \varphi_{N,t,-1} \otimes \tilde \varphi_{N,t}\Big|_{t=1/2}^{t=1} - 2^{ds_2-s_1} \int_{1/2}^1 \varphi_{N,t,-1} \otimes \tilde \varphi_{N,t,1} \frac{P'(Nt)}{N^{d-1}}\ dt.$$
    \item[(ii)]  If $s_2 > -u$ then we have 
\begin{align*}
    \int_{1/2}^1 \varphi_{N,t} \otimes \tilde \varphi_{N,t}\ dt &=
\frac{2^{-ds_2}}{2\pi i} \varphi_{N,t} \otimes \tilde \varphi_{N,t,-1} \frac{N^{d-1}}{P'(Nt)}\Big|_{t=1/2}^{t=1}
- 2^{s_1-ds_2} \int_{1/2}^1 \varphi_{N,t,1} \otimes \tilde \varphi_{N,t,-1} \frac{N^{d-1}}{P'(Nt)}\ dt \\
& + \frac{2^{-ds_2}}{2\pi i} \int_{1/2}^1 \varphi_{N,t} \otimes \tilde \varphi_{N,t,-1} \frac{N^d P''(Nt)}{P'(Nt)^2}\ dt.
\end{align*}    
\end{itemize}
\end{lemma}

Note that the quantity $P'(Nt)$ that appears in some of the denominators here is non-vanishing thanks to the lower bounds $N \geq C_3$ and $t \geq 1/2$; indeed the tuples
\begin{equation}\label{weight-bound}
\left(\frac{P'(Nt)}{N^{d-1}}\right)_{N \in \I}, \left(\frac{N^{d-1}}{P'(Nt)}\right)_{N \in \I}, \left(\frac{N^d P''(Nt)}{P'(Nt)^2}\right)_{N \in \I}
\end{equation}
can all be easily verified to have a $\V^r$ norm of $O(1)$ for all $1/2 \leq t \leq 1$.  This is the main reason why we work with $\tilde A_N$ instead of $A_N$ in most of this paper.

\begin{proof} To prove (i) it suffices to show that
$$
\int_{1/2}^1 e(\xi_1 Nt + \xi_2 P(Nt))\ dt = \frac{e(\xi_1 Nt + \xi_2 P(Nt))}{2\pi i N \xi_1} \Big|_{t=1/2}^{t=1} - \int_{1/2}^1 e(\xi_1 Nt + \xi_2 P(Nt)) \frac{P'(Nt) \xi_2}{\xi_1} \ dt$$
whenever $\xi_1 \neq 0$ and $N \geq C_3$, but this follows by writing $e(\xi_1 Nt) = \frac{1}{2\pi i N \xi_1} \frac{d}{dt} e(\xi_1 Nt)$ and integrating by parts.  Similarly, to prove (ii) it suffices to show that
\begin{align*}
\int_{1/2}^1 e(\xi_1 Nt + \xi_2 P(Nt))\ dt &= \frac{e(\xi_1 Nt + \xi_2 P(Nt))}{2\pi i N \xi_2 P'(Nt)} \Big|_{t=1/2}^{t=1} - \int_{1/2}^1 e(\xi_1 Nt + \xi_2 P(Nt)) \frac{\xi_1}{P'(Nt) \xi_2} \ dt \\
& + \frac{1}{2\pi i} \int_{1/2}^1 e(\xi_1 Nt + \xi_2 P(Nt)) \frac{P''(Nt)}{\xi_2 P'(Nt)^2}\ dt
\end{align*}
whenever $\xi_2 \neq 0$ and $N \geq C_3$, but this follows by writing $e(\xi_2 P(Nt)) = \frac{1}{2\pi iN \xi_2 P'(Nt)} \frac{d}{dt} e(\xi_2 P(Nt))$ and integrating by parts.
\end{proof}

We will now show how Theorem \ref{model-est} is a consequence of Lemma \ref{integ-ident} and the following variant, which works with a fixed choice of $t$ but does not require any decay in the $s_1,s_2$ parameters.

\begin{theorem}[Model operator estimates, II]\label{model-est-2} Let $j_1,j_2 \in \{-1,0,+1\}$ be such that
\begin{equation}\label{no-sing}
(s_1,j_1), (s_2,j_2) \neq (-u,-1).
\end{equation}
Then for every $1/2 \leq t \leq 1$, one has  the small-scale model estimate
    \begin{equation}\label{other-small-2}
    \begin{split}
&   \left \|\left  ( \B^{l_1,l_2,m_{\hat \Z}}_{m_*}( \T^{l_1}_{\varphi_{N,t,j_1}} F, \T^{l_2}_{\tilde \varphi_{N,t,j_2}} G) \right)_{N \in \I_{\sml}} \right\|_{\ell^p(\Z;\V^r)}\\ &\quad \lesssim_{C_3} \langle \max(l,s_1,s_2) \rangle^{O(1)} 2^{O(\rho l)-c l \ind{p_1=p_2=2}}  \|F\|_{\ell^{p_1}(\Z)} \|G\|_{\ell^{p_2}(\Z)}.
\end{split}
    \end{equation}
and the large-scale model estimate    
\begin{equation}\label{all-all-large-2}
\begin{split}
&\left\| \left(\B_{1 \otimes m_{\hat \Z}}(\T_{\varphi_{N,t,j_1} \otimes 1} F_\A,\T_{\tilde \varphi_{N,t,j_2} \otimes 1} G_\A) \right)_{N \in \I_{\lrg}}\right \|_{L^p(\A_\Z; \V^r)} \\
&\quad \lesssim_{C_3} \langle \max(l,s_1,s_2) \rangle^{O(1)} 2^{O(\rho l)-c l \ind{p_1=p_2=2}} \|F_\A\|_{L^{p_1}(\A_\Z)} \|G_\A\|_{L^{p_2}(\A_\Z)}.
\end{split}
\end{equation}
\end{theorem}

We assume Theorem \ref{model-est-2} for now and show how it implies Theorem \ref{model-est}.  We give the argument for the large-scale tuple \eqref{all-all-large}, as the treatment of the small-scale tuple \eqref{other-small} is completely analogous.  From Theorem \ref{model-est-2} (with $j_1=j_2=0$), \eqref{FG-bound}, \eqref{FGA-bound} and Minkowski's integral inequality we already obtain the acceptability bound for \eqref{all-all-large} but with the factor $2^{-c \max(l,s_1,s_2) \ind{p_1=p_2=2}}$ replaced by $2^{-cl \ind{p_1=p_2=2}}$.  This gives the claim unless $p_1=p_2=2$ and $\max(s_1,s_2) > l$, so in particular $p=1$.  Since the high-high case $s_1,s_2 > -u$, $p_1=p_2=2$ has already been excluded, this only leaves us with the high-low case $s_1>l$, $s_2=-u$, $p_1=p_2=2$ and the low-high case $s_2>l$, $s_1=-u$, $p_1=p_2=2$.  In the low-high case one applies Lemma \ref{integ-ident}(ii), \eqref{weight-bound}, \eqref{var-alg}, and Minkowski's integral inequality to bound the left-hand side of \eqref{all-all-large} (where the integrand can be viewed as a linear functional applied to $\varphi_{N,t} \otimes \tilde \varphi_{N,t}$) by
$$ \lesssim 2^{-ds_2} \sup_{j_1, j_2=0,\pm 1} \sup_{1/2 \leq t \leq 1}
\| (\B_{1 \otimes m_{\hat \Z}}(\T_{\varphi_{N,t,j_1} \otimes 1} F_\A,\T_{\tilde \varphi_{N,t,j_2} \otimes 1} G_\A) )_{N \in \I_{\lrg}} \|_{L^1(\A_\Z; \V^r)} ,$$
and the acceptability of \eqref{all-all-large} in this case now follows from Theorem \ref{model-est-2} and \eqref{FG-bound}, \eqref{FGA-bound} (noting that the hypothesis \eqref{no-sing} is verified).  In the high-low case one argues similarly using Lemma \ref{integ-ident}(i) instead of Lemma \ref{integ-ident}(ii).

It remains to establish Theorem \ref{model-est-2}.  This will be the purpose of the next three sections of this paper.

\section{The small-scale estimate: applying the Rademacher--Menshov inequality}\label{small-sec}

In this section we establish \eqref{other-small-2}.  A key tool in the small-scale case will be the following two-dimensional version of the Rademacher--Menshov inequality. 

\begin{lemma}[Two-dimensional Rademacher--Menshov]\label{2drm}  Let $K \in \Z_+$, and for any $k_1,k_2 \in [K]$ let $a_{k_1,k_2}$ be a complex number, with the convention that $a_{k_1,k_2}=0$ if $k_1=0$ or $k_2=0$.  Then for any $1 < r < \infty$, one has
\begin{align*}
    &
\| (a_{k,k})_{k \in [K]} \|_{\V^r} \lesssim_r \sum_{M_1,M_2 \in 2^\N \cap [K]} 
\big\| (\Delta a_{M_1 j_1,M_2j_2})_{(j_1,j_2) \in [K/M_1] \times [K/M_2]} \big\|_{\ell^r},
\end{align*} 
where $\Delta a_{M_1 j_1,M_2j_2}:= a_{M_1 j_1,M_2j_2}-a_{M_1(j_1-1),M_2j_2}-a_{M_1j_1,M_2(j_2-1)}+a_{M_1(j_1-1),M_2(j_2-1)}$.
\end{lemma}

The one-dimensional analogue of this inequality is well known; see e.g., \cite[Lemma 2.5, pp. 534]{MSZ2}.

\begin{proof}  By definition \eqref{vardef} of the $\V^r$ norm, one has
$$  \| (a_{k,k})_{k \in [K]} \|_{\V^r} \lesssim_r \| (a_{k_j,k_j} - a_{k_{j-1},k_{j-1}})_{j \in [J]} \|_{\ell^r}$$
for some sequence $1 \leq k_1 < \dots < k_J \leq K$, with the convention $k_0=0$.  

Let $\mu$ be the discrete complex measure on $[K]^2$ with masses
$$ 
\mu(\{(l_1,l_2)\}) \coloneqq a_{l_1,l_2} - a_{l_1-1,l_2} - a_{l_1,l_2-1} + a_{l_1-1,l_2-1}.$$
By telescoping series we may write
$$
a_{k_j,k_j} - a_{k_{j-1},k_{j-1}} = \mu( [k_j]^2 \backslash [k_{j-1}]^2 ).$$
Observe that the $L$-shaped region $[k_j]^2 \backslash [k_{j-1}]^2$ can be partitioned into the union of two rectangles:
$$ [k_j]^2 \backslash [k_{j-1}]^2 = [k_j] \times ([k_j] \backslash [k_{j-1}]) \uplus ([k_j] \backslash [k_{j-1}]) \times [k_{j-1}].$$
We partition these rectangles further into dyadic subrectangles as follows.  For each $M \in 2^\N \cap [K]$, let ${\mathcal I}_M$ be the collection of all discrete dyadic intervals $I$ in $[K]$ of length $M$, thus $I = [M] + (j-1)M = \{ jM-M+1,\dots,M\}$ for some $j \in [K/M]$.  Every interval $J$ in $[K]$ can then be written as the union of disjoint dyadic intervals $I \in \bigcup_{M \in 2^\N \cap [K]} {\mathcal I}_M$, in such a manner that at most two intervals are used from each collection ${\mathcal I}_M$.  Indeed, one can take the $I$ to be the maximal dyadic intervals contained in $J$: for each scale $M$, the intervals in ${\mathcal I}_M$ that lie in $J$ are consecutive, and all but the two extreme intervals in this sequence will fail to be maximal. Taking Cartesian products, we conclude that the region
$[k_j]^2 \backslash [k_{j-1}]^2$ can be written as the union of dyadic rectangles $I_1 \times I_2$ with $I_1 \in {\mathcal I}_{M_1}, I_2 \in {\mathcal I}_{M_2}$ for some $M_1,M_2 \in 2^\N \cap [K]$, in such a way that each pair $(M_1,M_2)$ is associated to $O(1)$ rectangles $I_1 \times I_2$.  From the triangle inequality, we thus have
$$ \mu([k_j]^2 \backslash [k_{j-1}]^2) \lesssim 
\sum_{M_1,M_2 \in 2^\N \cap [K]} \sup_{I_1 \in {\mathcal I}_{M_1}, I_2 \in {\mathcal I}_{M_2}: I_1 \times I_2 \subset [k_j]^2 \backslash [k_{j-1}]^2} |\mu(I_1 \times I_2)|$$
and hence on taking $\ell^r$ norms
$$  \| (a_{k,k})_{k \in [K]} \|_{\V^r} \lesssim_r 
\sum_{M_1,M_2 \in 2^\N \cap [K]} \bigg\|\Big (\sup_{I_1 \in {\mathcal I}_{M_1}, I_2 \in {\mathcal I}_{M_2}: I_1 \times I_2 \subset [k_j]^2 \backslash [k_{j-1}]^2} |\mu(I_1 \times I_2)|\Big)_{j \in [J]} \bigg\|_{\ell^r};$$
since the rectangles $I_1 \times I_2$ associated to a given region $[k_j]^2 \backslash [k_{j-1}]^2$ are disjoint, we conclude that
$$  \| (a_{k,k})_{k \in [K]} \|_{\V^r} \lesssim_r 
\sum_{M_1,M_2 \in 2^\N \cap [K]} \left\| (\mu(I_1 \times I_2))_{I_1 \in {\mathcal I}_{M_1}, I_2 \in {\mathcal I}_{M_2}} \right\|_{\ell^r}.$$
If $I_1 = [M_1] + (j_1-1)M_1$ and $I_2 = [M_2] + (j_2-1)M_2$ then
$$ \mu(I_1 \times I_2) = a_{M_1 j_1,M_2j_2}-a_{M_1(j_1-1),M_2j_2}-a_{M_1j_1,M_2(j_2-1)}+a_{M_1(j_1-1),M_2(j_2-1)}$$
and the claim follows.
\end{proof}

We can combine this with Khintchine's inequality to conclude:

\begin{corollary}[Rademacher--Menshov for bilinear forms]\label{rm-bil}  Let $K \in \Z_+$, and for any $k \in [K]$ let $f_k \in V, g_k \in W$ be elements of some vector spaces $V,W$.  Let $0 < q < \infty$, and let $B \colon V \times W \to L^q(X)$ be a bilinear map for some measure space $X$.  Then
\begin{equation}\label{bfkgk}
\begin{split}
&\| (B(f_k,g_k))_{k \in [K]} \|_{L^q(X;\V^2)} \\
&\lesssim_q \langle \log K \rangle^{\max(2,\frac{2}{q})} \sup_{\epsilon_1,\epsilon'_1,\dots,\epsilon_K,\epsilon'_K, \in \{-1,+1\}} 
 \bigg\| B\Big( \sum_{k \in [K]} \epsilon_k (f_k-f_{k-1}), \sum_{k \in [K]} \epsilon'_k (g_k-g_{k-1}) \Big) \bigg\|_{L^q(X)}
\end{split}
\end{equation}
with the conventions $f_0=g_0=0$.
\end{corollary}

In our applications, the set $[K]$ will index a lacunary set of scales, so the $\log K$ type losses are in fact doubly logarithmic in the scale parameters.  This will allow us to profitably use this corollary for scales as large as $2^{2^u}$.  Note in this corollary that the bilinear operator $B$ is not permitted to depend on $k$, but fortunately the Fourier-analytic manipulations of the preceding section have achieved such an independence of $k$ for the bilinear operator appearing in \eqref{other-small-2}.

\begin{proof}  We may normalize
\begin{equation}\label{supe}
 \sup_{\epsilon_1,\epsilon'_1,\dots,\epsilon_K,\epsilon'_K, \in \{-1,+1\}} \bigg\| B\Big( \sum_{k \in [K]} \epsilon_k (f_k-f_{k-1}), \sum_{k \in [K]} \epsilon'_k (g_k-g_{k-1}) \Big) \bigg\|_{L^q(X)} = 1.
 \end{equation}
 For each $x \in X$, we apply Lemma \ref{2drm} with $a_{k_1,k_2} = B(f_{k_1},g_{k_2})(x)$ and $r=2$ to bound the left-hand side of \eqref{bfkgk} by
$$ \lesssim \bigg\| \sum_{M_1,M_2 \in 2^\N \cap [K]} \big\| (B(\tilde f_{M_1 j_1}, \tilde g_{M_2 j_2}))_{(j_1,j_2) \in [ K/M_1] \times [K/M_2]} \big\|_{\ell^2} \bigg\|_{L^q(X)},$$
where $\tilde f_{M_1 j_1}:=f_{M_1 j_1} - f_{M_1(j_1-1)}$ and $\tilde g_{M_2 j_2} := g_{M_2 j_2} - g_{M_2(j_2-1)}$.
The last norm by the triangle or quasi--triangle inequality \eqref{quasi} is bounded by
$$ \lesssim_q \langle K \rangle^{\max(2,\frac{2}{q})} \sup_{M_1,M_2 \in 2^\N \cap [K]} \left\| (B(\tilde f_{M_1 j_1}, \tilde g_{M_2 j_2}))_{(j_1,j_2) \in [K/M_1] \times [K/M_2]} \right\|_{L^q(X; \ell^2)}.$$
Thus it suffices to show for each $M_1,M_2 \in 2^\N \cap [K]$ that
$$
\Big\| (B(\tilde f_{M_1 j_1}, \tilde g_{M_2 j_2}))_{(j_1,j_2) \in [K/M_1] \times [K/M_2]} \Big\|_{L^q(X; \ell^2)}^q \lesssim_q 1.$$
But by two applications of Khintchine's inequality, one can bound the left-hand side by the expected value of
$$
\bigg\| \sum_{j_1 \in [K/M_1]} \sum_{j_2 \in [K/M_2]} \epsilon_{j_1} \epsilon'_{j_2} B(\tilde f_{M_1 j_1}, \tilde g_{M_2 j_2}) \bigg\|_{L^q(X)}^q,$$
where $\epsilon_{j_1}, \epsilon'_{j_2}$ are independent random Bernoulli signs.  But every instance of this random expression can be factored (after relabeling the signs) in the form of one of the norms in \eqref{supe}, raised to the power $q$, and the claim follows.
\end{proof}

We now apply this estimate to \eqref{other-small-2}.  We enumerate the elements of ${\I}_{\sml}$ in order as $N_1 < \dots < N_K$; we may assume that $K \geq 1$ since otherwise there is nothing to prove.  From \eqref{small-def} we have $K = O( 2^u )$.  Thus by Lemma \ref{rm-bil} we may bound the left-hand side of \eqref{other-small-2} by
$$ u^{O(1)} \| \B^{l_1,l_2,m_{\hat \Z}}_{m_*}( \T^{l_1}_{\varphi_*} F, \T^{l_2}_{\tilde \varphi_*} G ) \|_{\ell^p(\Z)}$$
for some cutoffs $\varphi_*, \tilde\varphi_*$ of the form
\begin{align}
    \varphi_* &= \sum_{k \in [K]} \epsilon_k (\varphi_{N_k,t,j_1} - \varphi_{N_{k-1},t,j_1}) \label{varphistar}\\
    \tilde \varphi_* &= \sum_{k \in [K]} \tilde \epsilon_k (\tilde \varphi_{N_k,t,j_2} - \tilde \varphi_{N_{k-1},t,j_2})\label{varphistar-2} 
\end{align}
for some signs $\epsilon_k, \tilde \epsilon_k \in \{-1,+1\}$, where we adopt the convention $\varphi_{N_0,t,j_1} = \tilde \varphi_{N_0,t,j_2}=0$.  Note from \eqref{u-def} that $u^{O(1)} \lesssim_{C_3} 2^{O(\rho l)}$, so the loss of $u^{O(1)}$ will be acceptable for us.  It now suffices to show that
\begin{equation}\label{e:small scales}
\| \B^{l_1,l_2,m_{\hat \Z}}_{m_*}( \T^{l_1}_{\varphi_*} F, \T^{l_2}_{\tilde \varphi_*} G ) \|_{\ell^p(\Z)}
 \lesssim_{C_3} \langle \max(l,s_1,s_2) \rangle^{O(1)} 2^{-c l \ind{p_1=p_2=2}}  \|F\|_{\ell^{p_1}(\Z)} \|G\|_{\ell^{p_2}(\Z)}.\end{equation}

We now use

\begin{lemma}[Single-scale estimate]\label{ssu}  If $\tilde F \in \ell^{p_1}(\Z), \tilde G \in \ell^{p_2}(\Z)$ have Fourier support on ${\mathcal M}_{l_1, \leq -3u}$ and
${\mathcal M}_{l_2, \leq -3du}$ respectively, then
$$    \| \B^{l_1,l_2,m_{\hat \Z}}_{m_*} (\tilde F, \tilde G) \|_{\ell^p(\Z)} \lesssim_{C_3} 2^{-cl \ind{p_1=p_2=2}} \|\tilde F\|_{\ell^{p_1}(\Z)} \|\tilde G\|_{\ell^{p_2}(\Z)}.$$
\end{lemma}

\begin{proof} The strategy is to apply Proposition \ref{mod-approx} in reverse, so that Theorem \ref{improv} may be applied.
We may normalize $\|\tilde F\|_{\ell^{p_1}(\Z)} = \|\tilde G \|_{\ell^{p_2}(\Z)} = 1$.  From Proposition \ref{mod-approx} with $N = 2^{u}$ and $s=0$, we see that
$$
\| \tilde A_{2^{u}}( \tilde F, \tilde G ) - \B^{l_1,l_2,m_{\hat \Z}}_{\tilde m_{2^{u}, \R} m_*}(\tilde F,\tilde G) \|_{\ell^p(\Z)} \lesssim_{C_3} 2^{O(2^{\rho l})-u} \lesssim_{C_3} 2^{-cl \ind{p_1=p_2=2}},$$
noting that on the Fourier support of $\tilde F, \tilde G$ the multipliers $m_*$ and $\eta_{\leq -u} \otimes \eta_{\leq -du}$ are both equal to $1$.
Since $\F_\Z \tilde F$ vanishes on ${\mathcal M}_{\leq l_1-1, \leq -\Log N+l_1-1}$ and $\F_\Z \tilde G$ vanishes on the major arcs ${\mathcal M}_{\leq l_2-1, \leq -d\Log N+dl_2-d}$, we see from Theorem \ref{improv} (and \eqref{anf}) that
$$
\| \tilde A_{2^u}( \tilde F, \tilde G )  \|_{\ell^p(\Z)} \lesssim_{C_3} 2^{-cl \ind{p_1=p_2=2}}.$$
By the triangle inequality, it thus suffices to show that
$$ \| \B^{l_1,l_2,m_{\hat \Z}}_{(1-2\tilde m_{2^u, \R}) m_*}(\tilde F,\tilde G) \|_{\ell^p(\Z)} \lesssim_{C_3} 2^{O(2^{\rho l})-u}.$$
Applying Lemma \ref{crude-mult}(ii) (and Lemma \ref{mag-lem}(iii)) with $r_1 = 2^{-2u}$ and $r_2 = 2^{-2du}$, it suffices to show that
$$ \frac{\partial^{j_1}}{\partial \xi_1^{j_1}} \frac{\partial^{j_2}}{\partial \xi_1^{j_2}} ((1-2\tilde m_{2^u, \R}) m_*)(\xi_1,\xi_2)
\lesssim 2^{(2j_1 + 2dj_2-1)u}$$
for all $\xi_1,\xi_2 \in \R$ and $0 \leq j_1,j_2 \leq 2$.  By the product rule and definition of $m_*$ it suffices to show that
$$ \frac{\partial^{j_1}}{\partial \xi_1^{j_1}} \frac{\partial^{j_2}}{\partial \xi_1^{j_2}} (1-2\tilde m_{2^u, \R})(\xi_1,\xi_2)
\lesssim 2^{(2j_1 + 2dj_2-1)u}$$
when $\xi_1 = O( 2^{-2u})$, $\xi_2 = O (2^{-2du})$, and $0 \leq j_1,j_2 \leq 2$. But from \eqref{mn-def} one has
\begin{align*}
1 - 2\tilde m_{2^u, \R}(\xi_1,\xi_2) &= 2 \int_{1/2}^1 1 - e(2^u t \xi_1 + P(2^u t) \xi_2)\ dt \\
&=-4\pi i \int_0^1 \int_{1/2}^1 (2^u \xi_1 + 2^u P'(2^u tt') \xi_2) e(2^u t \xi_1 + P(2^u t) \xi_2)\ dt dt'
\end{align*} 
so by differentiation under the integral sign and the triangle inequality it suffices to show that
$$
 \frac{\partial^{j_1}}{\partial \xi_1^{j_1}} \frac{\partial^{j_2}}{\partial \xi_1^{j_2}} (2^u \xi_1 + 2^u P'(2^u tt') \xi_2) e(2^u t \xi_1 + P(2^u t) \xi_2)
\lesssim 2^{(2j_1 + 2dj_2-1)u}
$$
uniformly for $t \in [0,1]$, $t' \in [1/2,1]$. But this follows from direct calculation (in fact one obtains a slightly stronger bound of $O(2^{(j_1+dj_2-1)u})$ when $j_1=j_2=0$ and $O(2^{(j_1+dj_2)u})$ when $j_1+j_2>0$).
\end{proof}

In view of this lemma, it now suffices to establish the bounds
$$ \| \T^{l_1}_{\varphi_*} \|_{\ell^q(\Z) \to \ell^q(\Z)}, \| \T^{l_2}_{\tilde \varphi_*} \|_{\ell^q(\Z) \to \ell^q(\Z)}
\lesssim_{C_3,q} \langle \max(l,s_1,s_2) \rangle^{O(1)}$$
for any $1 < q < \infty$.  By interpolation, it suffices to achieve this when $q$ is an even integer or the dual of an even integer.
Using Theorem \ref{thm:iw}, it suffices to show that
$$ \| \T_{\varphi_*} \|_{L^q(\R) \to L^q(\R)}, \| \T_{\tilde \varphi_*} \|_{L^q(\R) \to L^q(\R)}
\lesssim_{C_3,q} \max(1,s_1,s_2)^{O(1)}$$
for all $1 < q < \infty$.

By expanding out \eqref{varphistar-2}, \eqref{tvnt-def}, \eqref{varphip-def} (and treating the $s_2 > -u$, $s_2=-u$ cases separately), we see that $\tilde \varphi_*$ is a shifted Calder\'on--Zygmund multiplier of the form treated in Theorem \ref{cz-shift}, with $A = 2^{-ds_2}$, $\lambda_N = 2^{ds_2} P(Nt) / N^d$, $K = O(\max(1,s_2))$, and $C=O(1)$.  (Note that the hypothesis \eqref{no-sing} is needed to avoid a divergence at the frequency origin.)  The claim for $\T_{\tilde \varphi_*}$ then follows from that theorem.  The treatment of $\T_{\varphi_*}$ is similar (with $s_2$ replaced by $s_1$, $P(Nt)$ replaced by $Nt$, and $d$ replaced by $1$).
This concludes the proof of \eqref{other-small-2}.

\section{The large-scale estimate: exploiting tensor product structure}

In this section we establish \eqref{all-all-large-2}.  Note from Examples \ref{avg-mult}, \ref{tensor-mult} that one can factor the bilinear operator $\B_{1 \otimes m_{\hat \Z}}$ as the tensor product of the identity and the arithmetic averaging operator $A_{\hat \Z}$.  Thus on the one hand we can write
$$ \B_{1 \otimes m_{\hat \Z}}(\T_{\varphi_{N,t,j_1} \otimes 1} F_\A,\T_{\tilde \varphi_{N,t,j_2} \otimes 1} G_\A) $$
as
\begin{equation}\label{form-1}
 \int_{\hat \Z} (\T_{\varphi_{N,t,j_1} \otimes 1} \tau_{(0,y)} F_\A) (\T_{\tilde \varphi_{N,t,j_2} \otimes 1} \tau_{(0,P(y))} G_\A)\ d\mu_{\hat \Z}(y)
 \end{equation}
where we define the translation operators $\tau_h F(x) \coloneqq F(x-h)$ for any $F \in L^0(\G)$ and $h \in \G$.   On the other hand, if we use $F_x \colon y \mapsto F(x,y)$ to denote the slice $F_x \colon \hat \Z \to \C$ of a function $F \colon \A_\Z \to \C$ at a real number $x$, we can write the slice
$$ \B_{1 \otimes m_{\hat \Z}}(\T_{\varphi_{N,t,j_1} \otimes 1} F_\A,\T_{\tilde \varphi_{N,t,j_2} \otimes 1} G_\A)_x$$
as
\begin{equation}\label{form-2}
A_{\hat\Z}( (\T_{\varphi_{N,t,j_1} \otimes 1} F_\A)_x, (\T_{\tilde \varphi_{N,t,j_2} \otimes 1} G_\A)_x).
\end{equation}

We now establish the easier case $(p_1,p_2) \neq (2,2)$, in which we do not need to obtain a gain of the form $2^{-cl}$; we will also not need to lose factors of $2^{O(\rho l)}$.  As such we will not need to exploit any cancellation in the averaging operator $A_{\hat \Z}$, and can use the formulation \eqref{form-1}.
By the triangle inequality, it thus suffices to show that
\begin{multline*}
\| ( (\T_{\varphi_{N,t,j_1} \otimes 1} \tilde F_\A) (\T_{\tilde \varphi_{N,t,j_2} \otimes 1} \tilde G_\A) )_{N \in \I_{\lrg}} \|_{L^p(\A_\Z;\V^r)}\\
\lesssim_{C_3} \langle \max(l,s_1,s_2) \rangle^{O(1)} \|\tilde F_\A\|_{L^{p_1}(\A_\Z)} \|\tilde G_\A\|_{L^{p_2}(\A_\Z)}    
\end{multline*}
for all $\tilde F_\A \in L^{p_1}(\A_\Z)$ and $\tilde G_\A \in L^{p_2}(\A_\Z)$.
There are now no interactions between the different fibers $\R\times \{y\}$, $y \in \hat \Z$ of $\A_\Z$, and so by H\"older's inequality and the Fubini--Tonelli theorem  (or \eqref{ttensor-bil}) it suffices to prove the continuous bilinear estimate
$$
\| ( (\T_{\varphi_{N,t,j_1}} \tilde F) (\T_{\tilde \varphi_{N,t,j_2}} \tilde G ))_{N \in \I_{\lrg}} \|_{L^p(\R;\V^r)}
\lesssim_{C_3} \langle \max(l,s_1,s_2) \rangle^{O(1)} \| \tilde F\|_{L^{p_1}(\R)} \|\tilde G\|_{L^{p_2}(\R)}$$
for any $\tilde F \in L^{p_1}(\R), \tilde G \in L^{p_2}(\R)$.  By \eqref{var-alg} and H\"older's inequality it suffices to establish the linear bounds
\begin{equation}\label{lin-1}
 \| (\T_{\varphi_{N,t,j_1}} \tilde F)_{N \in \I_{\lrg}} \|_{L^{p_1}(\R; \V^r)} \lesssim_{C_3} \max(1,s_1)^{O(1)}
\| \tilde F\|_{L^{p_1}(\R)} 
\end{equation}
and
\begin{equation}\label{lin-2}
\| (\T_{\tilde \varphi_{N,t,j_2}} \tilde G)_{N \in \I_{\lrg}} \|_{L^{p_2}(\R; \V^r)} \lesssim_{C_3} \max(1,s_2)^{O(1)}
\| \tilde G\|_{L^{p_2}(\R)}.
\end{equation}

We just establish the latter estimate, as the former is similar.  First suppose that we are in the high-frequency case $s_2 > -u$.  In this case we use \eqref{varsum} to replace the $\V^r$ norm by an $\ell^2$ norm, thus we now wish to show
$$ \| (\T_{\tilde \varphi_{N,t,j_2}} \tilde G)_{N \in \I_{\lrg}} \|_{L^{p_2}(\R; \ell^2)} \lesssim_{C_3} \max(1,s_2)^{O(1)}
\| \tilde G\|_{L^{p_2}(\R)}.$$
But as with the arguments at the end of Section \ref{small-sec}, the $\tilde \varphi_{N,t,j_2}$ form a family of the type considered in Theorem \ref{cz-shift}, with $A = 2^{-ds_2}$, $\lambda_N = 2^{ds_2} P(Nt) / N^d$, $K = O(\max(1,s_2))$, and $C=O(1)$, and the claim now follows from the shifted square function estimate proven in that theorem.

Now suppose we are in the low-frequency case $s_2=-u$, which means that $j_2=0,1$ by the hypothesis \eqref{no-sing}.  If $j_2=1$ then $\tilde \varphi_{N,t,j_2}$ vanishes at the origin and we can repeat the arguments from the high-frequency case.  If $j_2=0$ then $\tilde \varphi_{N,t,j_2} = \tilde \varphi_{N,t}$ no longer vanishes at the origin, but the difference $\tilde \varphi_{N,t} - \tilde \varphi_N$ does, and we can again use the high-frequency arguments to conclude.  By the triangle inequality, it now suffices to show that
$$
\| (\T_{\tilde \varphi_{N}} \tilde G)_{N \in \I_{\lrg}} \|_{L^{p_2}(\R; \V^r)} \lesssim_{C_3} 
\| \tilde G\|_{L^{p_2}(\R)}.
$$
But this follows from L\'epingle's inequality and a standard square function argument (see \cite[Theorem 1.1]{JSW}, with the square function argument contained in \cite[Lemma 3.2]{JSW}).

This completes the proof of the $(p_1,p_2) \neq (2,2)$ case of \eqref{all-all-large-2}.  Now we turn to the $(p_1,p_2)=(2,2)$ case, so that $p=1$. We begin with a general variational inequality:

\begin{lemma}[Interchanging variational and Lebesgue norms]\label{fkr}  Let $X$ be a measure space, and let $1 \leq R < r \leq \infty$.  Then for any $f_1,\dots,f_K \in L^r(X)$ one has
$$ \| (f_k)_{k \in [K]} \|_{L^r(X; \V^r)} \lesssim_{r,R} \| (f_k)_{k \in [K]} \|_{\V^R([K]; L^r(X))}.$$
\end{lemma}

\begin{proof}  We allow implied constants to depend on $r,R$. Since
$$ \| (f_k)_{k \in [K]} \|_{L^r(X; \V^r)} \lesssim \| (f_k)_{k \in [K]} \|_{L^r(X; V^r)} + \|f_1\|_{L^r(X)},$$
it suffices to establish the seminorm version
$$ \| (f_k)_{k \in [K]} \|_{L^r(X; V^r)} \lesssim_{r,R} \| (f_k)_{k \in [K]} \|_{V^R([K]; L^r(X))}$$
of the inequality.

We can assume that $f_k$ is not almost everywhere equal to $f_{k-1}$ for any $1 < k \leq K$, since otherwise we could concatenate the two indices $k,k-1$ together.  We normalize
$$ \| (f_k)_{k \in [K]} \|_{V^R([K]; L^r(X))}^R = 1$$
and then we can define a non-decreasing function $a \colon [K] \to [0,1]$ by the formula
$$ a(K') \coloneqq \| (f_k)_{k \in [K']} \|_{V^R([K']; L^r(X))}^R$$
for any $K' \in [K]$.  From \eqref{var-seminorm} we have the H\"older type bound
\begin{equation}\label{ho}
 \| f_{K_1} - f_{K_2} \|_{L^r(X)} \leq (a(K_1) - a(K_2))^{1/R}
 \end{equation}
whenever $1 \leq K_2 \leq K_1 \leq K$.  In particular, because we assumed $f_k$ not equal almost everywhere to $f_{k-1}$, we see that $a$ is strictly increasing.

For any $x \in X$, let $\mu_x$ be the absolutely continuous complex measure on $[0,1]$ defined by
$$ \mu_x(E) \coloneqq \sum_{2 \leq k \leq K: a(k) \in E} \frac{|E \cap [a(k-1),a(k)]|}{|[a(k-1),a(k)]|} (f_k(x) - f_{k-1}(x)).$$
Then we have
$$ f_{K_1}(x) - f_{K_2}(x) = \mu_x( [a(K_2), a(K_1)] )$$
whenever $1 \leq K_2 \leq K_1 \leq K$.  Also from \eqref{ho} and telescoping series (and the hypothesis $R \geq 1$) we observe the H\"older bound
\begin{equation}\label{must}
 \| \mu_x([s,t]) \|_{L^r(X)} \lesssim (t-s)^{1/R}
 \end{equation}
for any $0 \leq s \leq t \leq 1$.

Using dyadic decomposition as in the proof of Lemma \ref{2drm} (or \cite[Lemma 2.5, pp. 534]{MSZ2}), we have
$$ \| (f_k(x))_{k \in [K]} \|_{V^r} \lesssim \sum_{m=0}^\infty \big\| (\mu_x( [(j-1)2^{-m}, j2^{-m} ))_{j \in [2^m]} \big \|_{\ell^r}$$
and hence by the Fubini--Tonelli theorem and the triangle inequality
$$ \| (f_k)_{k \in [K]} \|_{L^r(X;V^r)} \lesssim \sum_{m=0}^\infty \Big\|  (\| \mu_x( [(j-1)2^{-m}, j2^{-m} )) \|_{L^r(X)})_{j \in [2^m]} \Big \|_{\ell^r}.$$
Applying \eqref{must}, the right-hand side is
$$ \lesssim \sum_{m=0}^\infty 2^{m/r} 2^{-m/R};$$
since $R<r$, this quantity is $O(1)$, and the claim follows.
\end{proof}

We can apply this lemma to bilinear operators:

\begin{corollary}[Interchanging variational and Lebesgue norms, II]\label{bilinear-fkr}  Let $V,W$ be normed vector spaces, let $K \in \Z_+$, and for each $k \in [K]$ let $f_k \in V, g_k \in W$.  Let $1 \leq R < r \leq \infty$, and let $B \colon V \times W \to L^r(X)$ be a bilinear map to $L^r(X)$ for some measure space $X$. Then
$$ \| (B(f_k,g_k))_{k \in [K]} \|_{L^r(X; \V^r)} \lesssim_{r,R} \|B\|_{V \times W \to L^r(X)}
\| (f_k)_{k \in [K]} \|_{\V^R([K]; V)} \| (g_k)_{k \in [K]} \|_{\V^R([K]; W)}.$$
\end{corollary}

\begin{proof}  We allow all implied constants to depend on $r,R$.    We may normalize 
$$ \|B\|_{V \times W \to L^r(X)} = \| (f_k)_{k \in [K]} \|_{\V^R([K]; V)} = \| (g_k)_{k \in [K]} \|_{\V^R([K]; W)}=1.$$
In particular the product sequence $(f_k,g_k) \in V \times W$, $k \in [K]$ obeys the variational norm bound
$$ \| (f_k,g_k)_{k \in [K]} \|_{\V^R([K]; V \times W)} \lesssim 1.$$
By Lemma \ref{fkr}, it suffices to show that
$$ \| B(f_k,g_k))_{k \in [K]} \|_{\V^R([K]; L^r(X))} \lesssim 1.$$
On the ball of radius $O(1)$ in $V \times W$, the (nonlinear) map $(f,g) \mapsto B(f,g)$ is Lipschitz continuous into $L^r(X)$ with Lipschitz constant $O(1)$, and the claim follows from \eqref{vardef}.
\end{proof}

We apply this lemma to the problem of establishing \eqref{all-all-large-2} in the $p_1=p_2=2$ case.  In the next section we establish the following arithmetic variant of Theorem \ref{improv}:

\begin{theorem}[Arithmetic bilinear estimate]\label{bile}  Let $l \in \N$, and let $f, g \in L^2(\hat \Z)$ obey one of the following hypotheses:
\begin{itemize}
    \item[(i)] $\F_{\hat \Z} f$ vanishes on $(\Q/\Z)_{\leq l}$;
    \item[(ii)] $\F_{\hat \Z} g$ vanishes on $(\Q/\Z)_{\leq l}$.
\end{itemize}
Then for any $1 \leq r < \frac{2d}{d-1}$ one has
$$ \| A_{\hat \Z}(f,g) \|_{L^r(\hat \Z)} \lesssim_{C_3,r} 2^{-c_r l} \| f\|_{L^2(\hat \Z)} \| g \|_{L^2(\hat \Z)}$$
(recall our conventions that $c_r>0$ denotes a constant that can depend on $d,r$).
\end{theorem}

The key point here is that the exponent $r$ in Theorem \ref{bile} is allowed to be slightly larger than $2$.

To prove \eqref{all-all-large-2} for $r>2$, we use the slice formulation \eqref{form-2}.  It suffices by monotonicity of $\V^r$ norms to work in the range $2 < r < \frac{2d}{d-1}$.    From \eqref{vecf-def}, \eqref{vecg-def} we see that every slice $(F_\A)_x$, $(G_\A)_x$ of $F_\A, G_\A$ take values in the finite-dimensional vector spaces $L^2(\hat \Z)^{(\Q/\Z)_{l_1}}$,$L^2(\hat \Z)^{(\Q/\Z)_{l_2}}$ respectively, and hence so do $\T_{\varphi_{N,t,j_1}\otimes 1} \vec F, \T_{\tilde \varphi_{N,t,j_2}\otimes 1} \vec G$ for any $N$.  By Theorem \ref{bile}, the operator norm of $A_{\hat \Z} \colon L^2(\hat \Z)^{(\Q/\Z)_{l_1}} \times L^2(\hat \Z)^{(\Q/\Z)_{l_2}} \to L^r(\Z/Q\Z)$ is $O_{C_3}(2^{-cl})$.  Applying H\"older's inequality to bound the $L^1(\A_\Z)$ norm by the $L^r(\A_\Z)$ norm, followed Corollary \ref{bilinear-fkr} for some $2 < R < r$, then Cauchy--Schwarz, we conclude\footnote{Strictly speaking, our definitions and arguments are not justified here because the vector spaces $L^2(\hat \Z)$, $L^1(\hat \Z; \V^r)$ are infinite-dimensional.  However, one can approximate $\hat \Z$ by finite cyclic groups $\Z/Q\Z$ to make these spaces finite-dimensional and then take limits to avoid this difficulty; indeed, given the definitions of $F_\A,G_\A$ we can just work with a single large but fixed $Q$.  Alternatively one can extend many of the previous vector-valued definitions to separable Banach spaces.  We leave the details to the interested reader.} that
\begin{multline*}
 \| (A_{\hat \Z}( \T_{\varphi_{N,t,j_1}\otimes 1} (F_\A)_x, \T_{\tilde \varphi_{N,t,j_2} \otimes 1} (G_\A)_x ) )_{N \in \I_{\lrg}} \|_{L^1(\R; L^1(\hat \Z; \V^r))}\\
 \lesssim_{C_3} 2^{-cl}  
\| (\T_{\varphi_{N,t,j_1} \otimes 1} (F_\A)_x)_{N \in \I_{\lrg}} \|_{L^2(\R; \V^R(\I_{\lrg}; L^2(\hat \Z)))}
\| (\T_{\tilde \varphi_{N,t,j_2} \otimes 1} (G_\A)_x)_{N \in \I_{\lrg}} \|_{L^2(\R; \V^R(\I_{\lrg}; L^2(\hat \Z)))},   
\end{multline*}
where we view $x$ as a variable of integration in $\R$. It thus suffices to establish the bounds
$$
\| (\T_{\varphi_{N,t,j_1}\otimes 1} \vec F)_{N \in \I_{\lrg}} \|_{L^2(\R; \V^R(\I_{\lrg}; L^2(\hat \Z)))}
\lesssim_{C_3} \max(1, s_1)^{O(1)} \|\vec F\|_{L^2(\R; L^2(\hat \Z))}$$
and
$$
\| (\T_{\tilde \varphi_{N,t,j_2} \otimes 1} \vec G)_{N \in \I_{\lrg}} \|_{L^2(\R; \V^R(\I_{\lrg}; L^2(\hat \Z)))}
\lesssim_{C_3} \max(1, s_2)^{O(1)} \|\vec G\|_{L^2(\R; L^2(\hat \Z))}$$
for any vector-valued functions $\vec F, \vec G \in L^2(\R; L^2(\hat \Z))$.
But these are simply vector-valued versions of \eqref{lin-1}, \eqref{lin-2}, and are proven in exactly the same fashion (since all of the tools used in the proof extend to the vector-valued setting); in particular, the vector-valued version of L\'epingle's inequality was established in \cite[Theorem 3.1, pp. 810]{MSZ1}, and all linear $L^p$ estimates extend to the vector-valued setting by the Marcinkiewicz--Zygmund inequality.  One may first wish to approximate $L^2(\hat \Z)$ by a finite dimensional Hilbert space to avoid technicalities.  This will conclude the proof of \eqref{all-all-large} (and thus Theorem \ref{main}), once we establish Theorem \ref{bile}.  This is the purpose of the next section.

\section{Arithmetic bilinear estimates}

We now prove Theorem \ref{bile}. It may be worth mentioning that the adelic viewpoint is not strictly necessary here and one could replace the profinite integers $\hat{\Z}$ here with $\Z/Q\Z$.  But then one needs to check that none of the bounds lose any factor of $Q$ (or even $\log Q$) as this would be fatal to the argument. From this point of view the adelic formalism is cleaner and automatically handles uniformity in the $Q$ parameter. 
We begin with the $r=1$ case, which is a limiting case of Theorem \ref{improv} in which the continuous aspect of that theorem degenerates completely, leaving only the arithmetic aspect:

\begin{proposition}\label{bile-2}  Theorem \ref{bile} holds when $r=1$.
\end{proposition}

We remark that when $q$ is a prime this result is essentially contained in \cite{BC} (when $P(\nn)=\nn^2$) and \cite{P1} (in the general case); see \cite{DLS} for the strongest current values for the constant $c$.

\begin{proof} For sake of exposition we assume that hypothesis (i) of Theorem \ref{bile}
holds; the case when hypothesis (ii) is assumed one  proceeds similarly.  By a limiting argument we may assume that the functions $F, G$ on $\hat \Z$ factor through a finite quotient $\Z/Q\Z$, in which case the task is to show that
$$ \| A_{\Z/Q\Z}(f,g) \|_{L^1(\Z/Q\Z)} \lesssim_{C_3} 2^{-cl} \| f\|_{L^2(\Z/Q\Z)} \| g \|_{L^2(\Z/Q\Z)}$$
assuming that $\F_{\Z/Q\Z} f$ vanishes on $(\Q/\Z)_{\leq l} \cap (\frac{1}{Q}\Z / \Z)$.

Let $N$ be a large natural number (which we will eventually send to infinity), and
let $R$ be an extremely large real number (which we will also send to
infinity, before sending $N$ to infinity).  In particular, one should
think of $N,R$ as being large compared to $l, Q$.  We define the
functions $f_{R}, g_{R} \in \Schwartz(\Z)$ by the formulae
\begin{align*}
f_R(n) &\coloneqq \frac{1}{\sqrt{R}} \psi( n/R ) f(n \hbox{ mod } Q) \\
g_R(n) &\coloneqq \frac{1}{\sqrt{R}} \psi( n/R ) g(n \hbox{ mod } Q) 
\end{align*}
where $\psi \in \Schwartz(\R)$ is a real even function with
$\| \psi \|_{L^2(\RR)}=1$ whose Fourier transform is supported on
$[-1,1]$.  Clearly $f_R \in L^2(\Z)$ has Fourier transform supported
on the set $\pi( [-1/R,1/R] \times \{ \alpha \in \frac{1}{Q} \Z / \Z \colon \hat f(\alpha) \neq 0 \} )$.
 From the hypothesis (i), we see that if $N, R$ is sufficiently large
(depending on $Q,l$), this union of arcs is disjoint from all of the
arcs in ${\mathcal M}_{\leq l, \leq -\Log N+l}$ (because the frequencies
$\alpha$ with $\hat f(\alpha) \neq 0$ have a non-zero separation from the
frequencies $(\Q/\Z)_{\leq l}$).  By Theorem
\ref{improv}, we conclude for $N,R$ sufficiently large that
\begin{align*}
\| A_{N,\Z}( f_R, g_R ) \|_{L^1(\Z)} \lesssim_{C_1} (2^{-cl} + \Log N^{-cC_1}) \|f_R\|_{L^2(\Z)} \|g_R\|_{L^2(\Z)}.
\end{align*}
From the Riemann integrability of $|\psi|^2$ it is
easy to see that
\begin{align*}
\lim_{R \to \infty} \|f_R\|_{L^2(\Z)} = \|f\|_{L^2(\Z/Q\Z)}
\end{align*}
and similarly
\begin{align*}
\lim_{R \to \infty} \|g_R\|_{L^2(\Z)} = \|g\|_{L^2(\Z/Q\Z)}
\end{align*}
and hence
\begin{align*}
\limsup_{N \to \infty} \limsup_{R \to \infty}
\| A_{N,\Z} ( f_R, g_R ) \|_{L^1(\Z)}
\lesssim 2^{-cl} \|f\|_{L^2(\Z/Q\Z)} \|g\|_{L^2(\Z/Q\Z)}.
\end{align*}
For any $N,R$, the Schwartz function nature of $\psi$ readily gives the asymptotic
$$
A_{N,\Z}( f_R, g_R)(n)
= \frac{1}{R} |\psi(n/R)|^2 A_{\Z/Q\Z}(f,g)(n \mod Q) + O_{N,Q}\big( R^{-2} \langle n/R \rangle^{-10} \big)
$$
and thus by the Riemann integrability of $|\psi|^2$ we obtain
\begin{align*}
\limsup_{R \to \infty}
\| A_{N,\Z}( f_R, g_R ) \|_{L^1(\Z)}
= \| A_{\Z/Q\Z}(f,g) \|_{L^1(\Z/Q\Z)}.
\end{align*}
Taking limits as $N \to \infty$, we then have
\begin{align*}
\limsup_{N \to \infty} \limsup_{R \to \infty}
\| A_{N, \Z}( f_R, g_R ) \|_{L^1(\Z)}
= \| A_{\Z/Q\Z}(f,g) \|_{L^1(\Z/Q\Z)}
\end{align*}
and the claim follows.
\end{proof}

By interpolation with Proposition \ref{bile-2}, we see that to establish the remaining cases of Theorem \ref{bile}, it will suffice to establish the bound
\begin{equation}\label{Avg-prof}
\| A_{\hat \Z} \|_{L^2(\hat \Z) \times L^2(\hat \Z) \to L^q(\hat \Z)} \lesssim_q 1
\end{equation}
for all $1 \leq q < \frac{2d}{d-1}$.   Approximating $\hat \Z$ by the product of finitely many of the $p$-adic groups $\Z_p$, it suffices by limiting arguments to show that
$$
\| A_{\prod_{p \in S} \Z_p} \|_{L^2(\prod_{p \in S} \Z_p) \times L^2(\prod_{p \in S} \Z_p) \to L^q(\prod_{p \in S} \Z_p)} \lesssim_q 1
$$
whenever $S$ is a finite set of primes.  From Examples \ref{avg-mult}, \ref{tensor-mult} we see that the bilinear operator $A_{\prod_{p \in S} \Z_p}$ is the tensor product of the individual operators $A_{\Z_p}$, so by \eqref{ttensor-bil} we may factor the operator norm as
$$
\| A_{\prod_{p \in S} \Z_p} \|_{L^2(\prod_{p \in S} \Z_p) \times L^2(\prod_{p \in S} \Z_p) \to L^q(\prod_{p \in S} \Z_p)} =
\prod_{p \in S} \| A_{\Z_p} \|_{L^2(\Z_p) \times L^2(\Z_p) \to L^q( \Z_p)}.$$

Thus it will suffice to establish the bound
\begin{equation}\label{avg-1}
\| A_{\Z_p} \|_{L^2(\Z_p) \times L^2(\Z_p) \to L^q(\Z_p)} \lesssim_q 1
\end{equation}
for all primes $p$, together with the improvement
\begin{equation}\label{avg-2}
\| A_{\Z_p} \|_{L^2(\Z_p) \times L^2(\Z_p) \to L^q(\Z_p)} \leq 1
\end{equation}
whenever $p$ is sufficiently large depending on $q$.

We begin with \eqref{avg-1}.  By bilinear interpolation it suffices to establish the bounds
\begin{equation}\label{avg-3}
\| A_{\Z_p} \|_{L^1(\Z_p) \times L^\infty(\Z_p) \to L^\infty(\Z_p)} \leq 1
\end{equation}
and
\begin{equation}\label{avg-4}
\| A_{\Z_p} \|_{L^\infty(\Z_p) \times L^1(\Z_p) \to L^s(\Z_p)} \lesssim_s 1
\end{equation}
for all $1 \leq s < \frac{d}{d-1}$.  The estimate \eqref{avg-3} is immediate from the pointwise inequality
$$ |A_{\Z_p}(f,g)| \leq A_{\Z_p}^{\nn}(|f|) \|g\|_{L^\infty(\Z_p)}.$$
To prove \eqref{avg-4}, we similarly use the pointwise inequality
$$ |A_{\Z_p}(f,g)| \leq A_{\Z_p}^{P(\nn)}(|g|) \|f\|_{L^\infty(\Z_p)}$$
so it suffices to show the linear $L^p$ improving bound
$$ \|A_{\Z_p}^{P(\nn)} \|_{L^1(\Z_p) \to L^s(\Z_p)} \lesssim_s 1$$
for $1 \leq s < \frac{d}{d-1}$.  By a limiting argument, it suffices to show that
$$ \|A_{\Z/p^j \Z}^{P(\nn)} \|_{L^1(\Z/p^j \Z) \to L^s(\Z/p^j\Z)} \lesssim_s 1$$
for all $j \in \N$.  By Minkowski's inequality, it suffices to show that the counting function $h \colon \Z/p^j\Z \to \N$ defined by
$$ h(m) \coloneqq \# \{ n \in \Z/p^j\Z: P(n) = m \}$$
has an $L^s(\Z/p^j\Z)$ norm of $O_s(1)$.  But this follows from Corollary \ref{cor:1} in the appendix.  This concludes the proof of \eqref{avg-1}. We remark that this argument in fact yields a weak-type endpoint for \eqref{avg-1}, but it is not clear to us how to use this to obtain a corresponding weak-type endpoint for \eqref{Avg-prof} as the weak $L^p$ spaces do not interact well with tensor products. In any event, for our application any exponent $q$ greater than $2$ would suffice, so endpoint estimates are not needed.

Now we prove \eqref{avg-2}.  By H\"older's inequality we may take $2 < q < \frac{2d}{d-1}$.  We let $l$ be a large number (depending on $q,P$) to be chosen later, and then assume that $p \in \PP$ is a prime that is sufficiently large depending on $l,q,P$. 
From Proposition \ref{bile-2} we then see that
\begin{align}
\label{eq:39}
\| A_{\Z_p}(f,g_0) \|_{L^1(\Z_p)} \lesssim_{C_3} 2^{-cl} \|f\|_{L^2(\Z_p)} \|g_0\|_{L^2(\Z_p)},
\end{align}
whenever $f,g_0 \in L^2(\Z_p)$ with $g_0$ of mean zero, since for $p$ large enough, the only element of $\Z_p^*$ of height at most $2^l$ is the origin.

Interpolating this bound with \eqref{avg-1} (for a slightly larger choice of $q$), we conclude that
\begin{equation}\label{eq:42}
 \| A_{\Z_p}(f,g_0) \|_{L^q(\Z_p)} \lesssim_{q,C_3} 2^{-c_q l} \|f\|_{L^2(\Z_p)} \|g_0\|_{L^2(\Z_p)}
 \end{equation}
(recall our conventions that $c_q>0$ denotes a constant that can depend on $d,q$ and varies from line to line).

Let $f,g \in L^2(\Z_p)$ with $\|f\|_{L^2(\Z_p)}=\|g\|_{L^2(\Z_p)}=1$.  It will suffice to
show that 
\begin{align*}
 \E_{n \in \Z_p} |A_{\Z_p}(f,g)(n)|^q \leq 1.
\end{align*}
Since $|A_{\Z_p}(f,g)| \leq A_{\Z_p}(|f|,|g|)$, we may assume without loss of generality that $f,g$ are non-negative.
We split $f = a + f_0$ and $g = b + g_0$, where
\begin{align*}
a \coloneqq \E_{n \in \Z_p} f(n)
\quad \text{ and } \quad
b \coloneqq \E_{n \in \Z_p} g(n)
\end{align*}
are the means of $f,g$, and $f_0 \coloneqq f - a$, $g_0 \coloneqq g-b$ are the mean zero components.  If we define the ``energies''
\begin{align*}
E_f \coloneqq \| f_0\|_{L^2(\Z_p)}^2
\quad \text{ and } \quad
E_g \coloneqq \| g_0 \|_{L^2(\Z_p)}^2
\end{align*}
then from Pythagoras' theorem we have $0 \leq E_f,E_g \leq 1$ and
\begin{equation}\label{abb}
|a| = (1-E_f)^{1/2}
\quad \text{ and } \quad
|b| = (1-E_g)^{1/2}.
\end{equation}
A short calculation shows that
\begin{align*}
A_{\Z_p}(a,b) = ab
\quad \text{ and } \quad
A_{\Z_p}(f_0,b) = 0
\end{align*}
and hence
\begin{align*}
A_{\Z_p}(f,g) = ab + A_{\Z_p}(f,g_0).
\end{align*}
Since the function $x \mapsto |x|^q$ is continuously twice differentiable, Taylor
expansion yields the pointwise bound
$$
|A_{\Z_p}(f,g)|^q = |ab|^q + q |ab|^{q-1} A_{\Z_p}(f, g_0) + O_q( |A_{\Z_p}(f,g_0)|^2 + |A_{\Z_p}(f,g_0)|^q ).
$$
Since $A_{\Z_p}(a,g_0)$ has mean zero, we have
\begin{align*}
\E_{n \in \Z_p} A_{\Z_p}(f, g_0)(n) &= 
 \E_{n \in \Z_p} A_{\Z_p}(f_0, g_0)(n) \\
 &\leq \|A_{\Z_p}(f_0,g_0)\|_{L^1(\Z_p)} 
\end{align*}  
and thus (since $|a|, |b| \leq 1$ and $q \geq 2$)
$$ \|A_{\Z_p}(f,g)\|_{L^q(\Z_p)}^q \leq |ab|^2 + O_q( \|A_{\Z_p}(f_0,g_0)\|_{L^1(\Z_p)} + \|A_{\Z_p}(f,g_0)\|_{L^2(\Z_p)}^2 + \|A_{\Z_p}(f,g_0)\|_{L^q(\Z_p)}^q ).$$
From \eqref{abb}, \eqref{eq:42}, the $L^2$ boundedness of $f,f_0,g_0$, and H\"older's inequality we conclude
$$ \|A_{\Z_p}(f,g)\|_{L^q(\Z_p)}^q \leq (1-E_f) (1-E_g) + O_{q,C_3}( 2^{-c_q l} (E_f^{1/2} E_g^{1/2} + E_g) ).$$
Since $E_f E_g \leq \min(E_f,E_g) \leq \frac{E_f+E_g}{2}$ one has
$$ (1-E_f) (1-E_g) \leq 1 -\frac{E_f+E_g}{2};$$
since $E_f^{1/2} E_g^{1/2} = O(E_f+E_g)$, the claim follows by choosing $l$ large enough depending on $q,C_3$.
This proves \eqref{avg-4}, and thus Theorem \ref{bile}.  

The proof of Theorem \ref{main} is (finally!) complete.

\section{Breaking duality}\label{pless}

In this section we extend Theorem \ref{main} to certain cases in which $p<1$.  Throughout this section $P \in \Z[\nn]$ is a polynomial of degree $d \geq 2$.

We begin with the following expansion of the range of applicability of \eqref{anf} for these averages.

\begin{lemma}[Single scale estimate below $\ell^1$]\label{single-below}  Let $1 < p_1,p_2 < \infty$ obey the constraints
\begin{equation}\label{dd0}
 \frac{1}{p_1} + \frac{2}{p_2}, \quad \frac{2}{p_1} + \frac{1}{p_2} < 2
 \end{equation}
when $d=2$, or
\begin{equation}\label{dd1}
 \frac{d^2+d-1}{p_1} + \frac{d^2+d+1}{p_2}, \quad 
 \frac{d^2+d+1}{p_1} + \frac{d^2+d-1}{p_2} < d^2+d+1
 \end{equation}
 when $d \geq 3$.  Then for any measure-preserving system $(X,\mu,T)$ one has
$$ \| A^{\nn,P(\nn)}_N(f,g) \|_{L^p(X)} \lesssim_{p_1,p_2,P} \|f\|_{L^{p_1}(X)} \|g\|_{L^{p_2}(X)}$$
 for all $N \geq 1$, $f \in L^{p_1}(X)$, $g \in L^{p_2}(X)$, where $\frac{1}{p} = \frac{1}{p_1}+\frac{1}{p_2}$.  Similarly with $A^{\nn,P(\nn)}_N$ replaced by $\tilde A^{\nn,P(\nn)}_N$.
\end{lemma}

We remark that if \cite[Conjecture 1.5]{HKLMY} holds, the condition should be able to be relaxed to
$$ \frac{d-1}{p_1} + \frac{d}{p_2}, \quad \frac{d}{p_1} + \frac{d-1}{p_2} < d,$$
bringing it in line with \eqref{dd0}.

\begin{proof}  From the pointwise bound $|\tilde A^{\nn,P(\nn)}_N(f,g)| \leq A^{\nn,P(\nn)}_N(|f|,|g|)$ it suffices to establish the claim for $A^{\nn,P(\nn)}_N$.  We may assume that $p<1$ since the claim follows from \eqref{anf} otherwise.  
By the Calder\'on transference principle it suffices to establish this bound for the case of the integer shift $(\Z,\mu_\Z,T_\Z)$.  Noting the pointwise bound
$$ |A^{\nn,P(\nn)}_N(f,g)(x)| \leq \sum_{I \in \mathcal I} \ind{I}(x) A^{\nn,P(\nn)}_N(\ind{I} |f|,\ind{I} |g|)(x),$$
where $I$ ranges over a collection $\mathcal I$ of intervals of length $O_P(N^d)$ and overlap $O_P(1)$, it suffices to establish the claimed bound when $f,g$ are supported in a single one of these intervals $I$, that is to say
$$ \| A^{\nn,P(\nn)}_N(f,g) \|_{\ell^p(\Z)} \lesssim_{p_1,p_2,P} \|f\|_{\ell^{p_1}(I)} \|g\|_{\ell^{p_2}(I)}.$$
As $A^{\nn,P(\nn)}_N(f,g)$ is supported in an interval of length $O_P(N^d)$, we have from H\"older's inequality and the hypothesis $p<1$ that
$$ \| A^{\nn,P(\nn)}_N(f,g) \|_{\ell^p(\Z)} \lesssim_{p_1,p_2,P}
N^{d(\frac{1}{p}-1)} \| A^{\nn,P(\nn)}_N(f,g) \|_{\ell^1(\Z)}.$$
From the triangle inequality and the Fubini--Tonelli theorem one has
$$ \| A^{\nn,P(\nn)}_N(f,g)\|_{\ell^1(\Z)} \leq \sum_{x \in \Z}
|f|(x) A^{P(\nn)-\nn} |g|(x)$$
(cf. \eqref{transpose}), so by H\"older's inequality it suffices to establish the bound
$$ \| A^{P(\nn)-\nn} g\|_{\ell^{p'_2}(\Z)} \lesssim_{p_1,p_2,P} N^{d(\frac{1}{p'_2}-\frac{1}{p_1})} \|g\|_{\ell^{p_1}(\Z)}$$
for any $g \in \ell^{p_1}(\Z)$.  But this follows from the results of \cite{HKLMY} (cf. Proposition \ref{lp-improv}).
\end{proof}

As remarked in the proof of Proposition \ref{lp-improv}, one expects the range of $p_1,p_2$ to be improvable here, at least in the case $d \geq 3$. We remark that the same argument allows one to break duality in \eqref{anf} (that is to say, obtain \eqref{anf} for at least some ranges of exponents $p_1,\dots,p_k$ with $\frac{1}{p_1}+\dots+\frac{1}{p_k} > 1$) for any average $A_N^{P_1(\nn),\dots,P_k(\nn)}$ (or $\tilde A_N^{P_1(\nn),\dots,P_k(\nn)}$) in which all the $P_i$ have degree at most $d$, with at least one of the differences $P_i-P_j$ having degree exactly $d$, for some $d \geq 2$; we leave the details to the interested reader.

Now we can obtain norm convergence results with an explicit range of $p_1,p_2$.

\begin{corollary}[Breaking duality for the mean ergodic theorem]\label{main-ext}  Let $(X,\mu,T)$ be a measure-preserving system with $X$ of finite measure, and let $P(\nn) \in \Z[\nn]$ have degree $d \geq 2$.  If $p_1,p_2,p$ obey the hypotheses in Lemma \ref{single-below}, then the averages $A^{\nn,P(\nn)}_N(f,g)$ converge in $L^p(X)$ norm for all $f \in L^{p_1}(X)$, $g \in L^{p_2}(X)$.
\end{corollary}

\begin{proof}  By Theorem \ref{main}(i) and H\"older's inequality (using the finite measure hypothesis) the claim already holds for (say) $f,g \in L^\infty(X)$.  The claim now follows from Lemma \ref{single-below} and the usual limiting argument (which is still valid in the quasinormed space $L^p(X)$).
\end{proof}

For the remaining components of Theorem \ref{main}, we can similarly break duality, albeit with a much poorer range of exponents:

\begin{proposition}[Breaking duality for all the ergodic theorems]\label{main-ext-2}  Let $P(\nn) \in \Z[\nn]$ have degree $d \geq 2$, and let $\eps>0$.  If $(\frac{1}{p_1},\frac{1}{p_2})$ is in a sufficiently small neighborhood of $(\frac{1}{2},\frac{1}{2})$ (where the neighborhood depends only on $d$,$\eps$), and $\frac{1}{p} \coloneqq \frac{1}{p_1} + \frac{1}{p_2}$, then the conclusions (i)-(iv) of Theorem \ref{main} hold for this choice of $p_1,p_2,p$, where in (iv) we replace the requirement $r>2$ with $r>2+\eps$.
\end{proposition}

It may be possible to refine the range of $p_1,p_2$ here to match that in Corollary \ref{main-ext} or Lemma \ref{single-below} by a more careful argument, but we will not attempt to do so here.

\begin{proof} (Sketch)  We repeat the proof of Theorem \ref{main}.  By the arguments in Section \ref{transfer-sec}, it suffices to show that Theorem \ref{end-var} holds for the indicated choice of $p_1,p_2,p$.  We then repeat the reductions in Section \ref{iw-decomp-sec} that were used to reduce Theorem \ref{end-var} to Theorem \ref{varp}.  The only differences are that \eqref{anf} is replaced by the more general Lemma \ref{single-below} (which in particular is applicable for $(\frac{1}{p_1}, \frac{1}{p_2})$ sufficiently close to $(\frac{1}{2},\frac{1}{2})$), and uses the quasi--triangle inequality in place of the triangle inequality when $p<1$  (adjusting the exponent $10$ appearing in the argument if necessary).  It then suffices to establish Theorem \eqref{varp} for $(\frac{1}{p_1},\frac{1}{p_2})$ in a neighborhood of $(\frac{1}{2},\frac{1}{2})$.  In fact it suffice to establish the cruder estimate
$$ \| ( \tilde A_N( F_N, G_N ))_{N \in \I} \|_{\ell^p(\Z;\ell^\infty)} \lesssim_{C_3} 2^{O(\max(l,s_1,s_2))}
\|f\|_{\ell^{p_1}(\Z)} \|g\|_{\ell^{p_2}(\Z)}.
$$
for $(\frac{1}{p_1},\frac{1}{p_2})$ in a neighbourhood of $(\frac{1}{2},\frac{1}{2})$, since the claim then follows by interpolation with the $p_1=p_2=2$ case of Theorem \ref{varp} and reducing the size of the neighborhood in an $\eps$-dependent fashion (here we use the interpolation theory\footnote{See for instance \cite{MSZ1} for an overview of this interpolation theory.} of variational norms, as well as the equivalence $\V^\infty \equiv \ell^\infty$).

The contribution of the small scales $\I_\sml$ can now be crudely handled by Lemma \ref{single-below} and the quasi-triangle inequality \eqref{quasi} (since we are now willing to concede factors of $2^{O(l)}$).  Hence we may work entirely with large scales $\I_\lrg$.  It is not difficult to verify that Proposition \ref{mod-approx} extends to the non-Banach regime $p<1$ (basically because Lemma \ref{crude-mult} does, and because one can freely lose powers of $q$ in that proposition).  Applying the arguments in Section \ref{approx-sec} with suitable changes, we reduce to showing that the $\ell^p(\Z; \ell^\infty)$ norm of \eqref{all-all-large} is bounded by
$$ \lesssim_{C_3} 2^{O(\max(l,s_1,s_2))}
\|f\|_{\ell^{p_1}(\Z)} \|g\|_{\ell^{p_2}(\Z)}$$
for $(\frac{1}{p_1},\frac{1}{p_2})$ in a neighborhood of $(\frac{1}{2},\frac{1}{2})$.

In the non-Banach regime we are no longer able to remove the integration in $t$; instead we crudely replace it by a supremum norm.  In lieu of Theorem \ref{model-est-2}, it will now suffice to show that
\begin{equation}\label{all-all-large-3}
\begin{split}
\left\| H\right \|_{L^p(\A_\Z; \V^r)}  \lesssim_{C_3} 2^{O( \max(l,s_1,s_2) )} \|F_\A\|_{L^{p_1}(\A_\Z)} \|G_\A\|_{L^{p_2}(\A_\Z)}
\end{split}
\end{equation}
where $H$ is the maximal operator
$$ H \coloneqq \sup_{t \in [1/2,1]} |\B_{1 \otimes m_{\hat \Z}}(\T_{\varphi_{N,t} \otimes 1} F_\A,\T_{\tilde \varphi_{N,t} \otimes 1} G_\A)|.
$$
(Here we implicitly use the fact that Theorem \ref{Sampling} continues to hold in the range $p<1$.) By a variant of \eqref{form-2}, each slice $H_x$ of $H$ at some $x \in \R$ is given by
$$ H_x = \sup_{t \in [1/2,1]} |A_{\hat \Z}( (T_{\varphi_{N,t} \otimes 1} F_\A)_x, (T_{\tilde \varphi_{N,t} \otimes 1} G_\A)_x )|.$$
We crudely bound
$$ T_{\varphi_{N,t} \otimes 1} F_\A \lesssim_{C_3} 2^{\max(0,s_1)} \mathrm{M}_{\mathrm{HL}} F_\A,$$
$$ T_{\tilde \varphi_{N,t} \otimes 1} G_\A \lesssim_{C_3} 2^{\max(0,s_2)} \mathrm{M}_{\mathrm{HL}} G_\A,$$
where $\mathrm{M}_{\mathrm{HL}}$ denotes the Hardy--Littlewood maximal operator in the $\R$ variable, so that
$$ H_x \lesssim_{C_3} 2^{O( \max(l,s_1,s_2) )} A_{\hat \Z}((\mathrm{M}_{\mathrm{HL}} F_\A)_x, (\mathrm{M}_{\mathrm{HL}} G_\A)_x ).$$
From the Hardy--Littlewood inequality and the Fubini--Tonelli theorem, it now suffices to establish the estimate
$$ \| A_{\hat \Z}(F,G) \|_{L^p(\A_\Z)} \lesssim
\|F\|_{L^{p_1}(\A_\Z)} \|G\|_{L^{p_2}(\A_\Z)}$$
for any $F \in L^{p_1}(\A_\Z)$, $G \in L^{p_2}(\A_\Z)$.  When $p \geq 1$ this follows from H\"older's inequality and the triangle inequality.  For $p < 1$ we can interpolate the $p \geq 1$ estimate with \eqref{Avg-prof} and conclude that
$$ \| A_{\hat \Z}(F,G) \|_{L^1(\A_\Z)} \lesssim
\|F\|_{L^{p_1}(\A_\Z)} \|G\|_{L^{p_2}(\A_\Z)}$$
for all $(\frac{1}{p_1},\frac{1}{p_2})$ sufficiently close to $(\frac{1}{2},\frac{1}{2})$, and the claim now follows from H\"older's inequality.
\end{proof}

\section{Unboundedness of quadratic variation}\label{quad-var}

In this section we show that the quadratic variation of polynomoial averages is unbounded in any Lebesgue space norm.  The counterexample already applies in the linear setting:

\begin{proposition}[Unboundedness of $V^2$]  Let $P(\nn) \in \Z[\nn]$ be a non-constant polynomial, and let $0 < p \leq \infty$.  Let $\I \subseteq \Z_+$ be an infinite set. Then for every $C>0$ there exists a measure-preserving system $(X,\mu,T)$ of total measure $1$ and $f \in L^\infty(X)$ with $\|f\|_{L^\infty(X)} \leq 1$ such that
$$ \| (A_{N,X}^{P(\nn)}(f))_{N \in \I} \|_{L^p(X; V^2)} > C.$$
\end{proposition}

We remark that the case $p=2$ of this proposition (with $f$ controlled in $L^2$ rather than $L^\infty$) was established by the first author in \cite{Kr} (the argument there is given for $P(\nn)=\nn^2$, but extends easily to more general polynomials).  This result relied on a previous result of Lewko and Lewko \cite{LL} who in turn invoked a result of Jones and Wang \cite{jw}.  It turns out that by appealing to the latter results directly we can handle all values of $p$, answering \cite[Conjecture 1]{Kr} in the affirmative.

\begin{proof}  Suppose for contradiction that this were not the case, then we would have the variational inequality
\begin{equation}\label{anx-bound}
 \| (A_{N,X}^{P(\nn)}(f))_{N \in \I} \|_{L^p(X; V^2)} \leq C \|f\|_{L^\infty(X)}
 \end{equation}
for every measure-preserving system $(X,\mu,T)$ and every $f \in L^\infty(X)$.

We apply this inequality to the following multidimensional system in which the different components of the shift have radically different mixing times (so that the averages $A_{N,X}^{P(\nn)}$ behave like martingale expectation operators).  Set $X = \TT^K$ for some $K \in \Z_+$ with Haar probability measure $\mu$, and let $f \colon X \to \C$ be a smooth function.  Fix a sequence $\alpha_1,\alpha_2,\dots,\alpha_K$ of real numbers that are linearly independent over $\Q$ (e.g., one could take $\alpha_i \coloneqq \log p_i$ where $p_i$ is the $i^{\mathrm{th}}$ prime). Let $N_1 < \dots < N_K$ be distinct elements of $\I$, and consider the shift map
$$ T(x_1,\dots,x_K) \coloneqq \left(x_1 + \frac{\alpha_1}{N_1^{d+1}}, \dots, x_K + \frac{\alpha_K}{N_K^{d+1}}\right),$$
where $d$ is the degree of $P$.  Then for any $k \in [K]$, we have
$$ A_{N_k,\TT^K}^{P(\nn)} f(x_1,\dots,x_K) = \E_{n \in [N_k]} f\left( x_1 + \frac{\alpha_1 P(n)}{N_1^{d+1}}, \dots, x_K + \frac{\alpha_K P(n)}{N_K^{d+1}}\right).$$
Let $\eps>0$.  If we assume for each $k \in [K]$ that $N_k$ is sufficiently large depending on $\eps, N_1,\dots,N_{k-1}, d, P, f$, then we have
\begin{align*}
f\left( x_1 + \frac{\alpha_1 P(n)}{N_1^{d+1}}, \dots, x_K + \frac{\alpha_K P(n)}{N_K^{d+1}}\right)
= &f\left(x_1 + \frac{\alpha_1 P(n)}{N_1^{d+1}}, \dots, x_{k-1} + \frac{\alpha_{k-1} P(n)}{N_{k-1}^{d+1}}, x_k, \dots, x_K \right)\\
&+ O(\eps)    
\end{align*}
for all $n \in [N_k]$ and $(x_1,\dots,x_K) \in X$, and thus
$$ A_{N_k,\TT^K}^{P(\nn)} f(x_1,\dots,x_K) = 
\E_{n \in [N_k]} f\left(x_1 + \frac{\alpha_1 P(n)}{N_1^{d+1}}, \dots, x_{k-1} + \frac{\alpha_{k-1} P(n)}{N_{k-1}^{d+1}}, x_k, \dots, x_K \right) + O(\eps).
$$
Because $\alpha_1,\dots,\alpha_{k-1}$ are linearly independent, a standard application of the Weyl equidistribution theorem shows that the sequence
$$ \left( \frac{\alpha_1 P(n)}{N_1^{d+1}} \mod 1, \dots, \frac{\alpha_{k-1} P(n)}{N_{k-1}^{d+1}} \mod 1 \right)$$
is equidistributed over the torus $\TT^{k-1}$.  Thus, if $N_k$ is chosen large enough, we have
$$ A_{N_k,\TT^K}^{P(\nn)} f(x) = \E_k f(x) + O(\eps)$$
for all $k \in [K]$ and $x \in X$, where $\E_k f$ is the conditional expectation
$$ \E_k f(x_1,\dots,x_K) \coloneqq \int_{\TT^{k-1}} f(y_1,\dots,y_{k-1},x_k,\dots,x_K)\ dy_1 \dots dy_{k-1}.$$
Taking variations, we conclude that
$$ \| (A_{N_k,\TT^K}^{P(\nn)} - \E_k f)_{k \in [K]} \|_{L^p(\TT^K;V^2)} \lesssim_K \eps $$
which from \eqref{anx-bound} and the triangle inequality (or quasi--triangle inequality \eqref{quasi}) gives
$$ \| (\E_k f)_{k \in [K]} \|_{L^p(\TT^K;V^2)} \lesssim_{C,p} \|f\|_{L^\infty(\TT^K)} + O_K(\eps).$$
Sending $\eps \to 0$ (noting that the left-hand side does not depend on $\eps$-dependent quantities such as $N_1,\dots,N_K$), we conclude that
$$ \| (\E_k f)_{k \in [K]} \|_{L^p(\TT^K;V^2)} \lesssim_{C,p} \|f\|_{L^\infty(\TT^K)}.$$
for any smooth $f \in L^\infty(\TT^K)$.  Taking limits, we see that we can drop the hypothesis that $f$ is smooth.  

We now define a map $\pi \colon \TT^K \to [0,1)$ by the formula
$$ \pi( x_1 \mod 1, \dots, x_K \mod 1 ) \coloneqq
\sum_{k \in [K]} \frac{\lfloor 2x_k\rfloor}{2^{K-k+1}}$$
for $x_1,\dots,x_K \in [0,1)$.
It is not difficult to see that $\pi$ pushes forward Haar measure on $\TT^K$ to Lebesgue measure on $[0,1)$, and furthermore if $\tilde f\in L^\infty([0,1))$ then
$$ \E_k(\tilde f \circ \pi) = (\tilde \E_k \tilde f) \circ \pi$$
almost everywhere on $\TT^K$, where $\tilde \E_k$ are the martingale projections
$$ \tilde \E_k \tilde f(x) \coloneqq 2^k \int_{(j-1)/2^k}^{j/2^k} f(y)\ dy$$
whenever $j \in [2^k]$ and $x \in [(j-1)/2^k, j/2^k)$.  From this we conclude that
$$ \| (\tilde \E_k \tilde f)_{k \in [K]} \|_{L^p([0,1);V^2)} \lesssim_{C,p} \|\tilde f\|_{L^\infty([0,1))}$$
for all $K \in \N$ and $f \in L^\infty([0,1))$.  Taking $K \to \infty$ and using monotone convergence, we conclude that
$$ \| (\tilde \E_k \tilde f)_{k \in \N} \|_{L^p([0,1);V^2)} \lesssim_{C,p} \|\tilde f\|_{L^\infty([0,1))}.$$
But this contradicts \cite[Proposition 8.1]{jw}.
\end{proof}

\begin{remark}  By considering a suitable product system, one can then construct a single measure-preserving system $(X,\mu,T)$ of total measure $1$ such that the vector-valued operator $f \mapsto (A_N^{P(\nn)}(f))_{N \in \I}$ is unbounded from $L^2(X)$ to $L^p(X;V^2)$.  It is likely that one can sharpen the construction further to find a single $f \in L^2(X)$ for which $\| (A_N^{P(\nn)}(f))_{N \in \I} \|_{V^2}=+\infty$ almost everywhere, but we will not do so here.
\end{remark}

By setting all but one function equal to the constant function $1$, and using the monotonicity of variational norms and $L^p$ norms, we obtain

\begin{corollary}[Failure of variational estimate for $r\leq 2$]\label{r-counter}    Let $P_1,\dots,P_k \in \Z[\nn]$ be polynomials, not all constant, let $0 < p_1,\dots,p_k,p \leq \infty$ and $0 < r \leq 2$.  Let $\I \subseteq \Z_+$ be an infinite set.  Then there does not exist any constant $C>0$ for which one has the estimate
$$ \| (A_{N,X}^{P_1(\nn),\dots,P_k(\nn)}(f_1,\dots,f_k))_{N \in I} \|_{L^p(X; V^r)} \leq C \|f_1\|_{L^{p_1}(X)} \dots \|f_k\|_{L^{p_k}(X)}$$
for all measure-preserving systems $X = (X,\mu,T)$ of total mass one, and all $f_1 \in L^{p_1}(X), \dots, f_k \in L^{p_k}(X)$.
\end{corollary}

Applying Proposition \ref{transf} in the contrapositive, we see that we similarly obtain a counterexample for the integer shift system in the H\"older exponent case $\frac{1}{p_1} + \dots + \frac{1}{p_k} = \frac{1}{p}$, and we can replace $A_N$ by $\tilde A_N$ in the Banach exponent case $p_1, \dots, p_k \geq 1$.

\appendix

\section{Ionescu--Wainger theory}\label{iw-app}

In this appendix we review some number-theoretic and Fourier-analytic constructions of Ionescu and Wainger \cite{IW} that allow one to apply Fourier projections to ``major arcs'' with good multiplier estimates.  See also \cite{M1}, \cite{MSZ3} for further development of the Ionescu--Wainger theory, and \cite{Pierce} for a recent discussion of the role of superorthogonality in that theory.  We will loosely follow the presentation in \cite{MSZ3}.  A new notational innovation is the introduction of the notion of the height $\Height(\alpha)$ of a profinite frequency $\alpha \in \Q/\Z$.

Throughout this appendix we fix a small quantity $\rho>0$ (in the main paper it is set by the formula \eqref{rho-def}). Let $C^0_\rho$ be a sufficiently large quantity depending on $\rho$.  If $l \leq C^0_\rho$, we define
$$ P_{\leq l} \coloneqq [2^l].$$
For $l > C^0_\rho$, we define $P_{\leq l}$ differently.  We first  define the natural
number
\[
D = D_{\rho} \coloneqq \lfloor 2/\rho \rfloor + 1,
\]
and for any natural number $l\in\N$, set
\begin{align*}
N_0 = N_0^{(l)} \coloneqq \lfloor 2^{\rho l/2} \rfloor + 1,
\quad \text{ and } \quad
Q_0 = Q_0^{(l)} \coloneqq (N_0!)^D.
\end{align*}
Then for $l > C_\rho^0$, we define the set
\begin{align*}
P_{\leq l} \coloneqq \big\{q = Qw: Q|Q_0 \text{ and } w\in W_{\le l}\cup\{1\}\big\},
\end{align*}
where
\begin{align*}
W_{\leq l} \coloneqq \bigcup_{k\in[D]}\bigcup_{(\gamma_1,\dots,\gamma_k)\in[D]^k}\big\{p_1^{\gamma_1} \cdots p_k^{\gamma_k}\colon
 p_1,\ldots, p_k\in(N_0^{(l)}, 2^l]\cap\PP \text{ are distinct}\big\}.
\end{align*}
In other words $W_{\leq l}$ is the set of all products of prime
factors from $(N_0^{(l)}, 2^l]\cap\PP$ of length at most $D$, with exponents between $1$ and $D$.

We observe that (for $C^0_\rho$ large enough) one has 
\begin{equation}\label{2l-include}
[2^l] \subset P_{\leq l}
\end{equation}
for all $l$.  This is trivial for $l \leq C^0_\rho$.  Now suppose that $l > C_\rho^0$ and $q \in [2^l]$.  Observe that there are at most $D$ primes larger than $N_0$ that can divide $q$, and each such prime can divide $q$ at most $D$ times, so the product of all these primes (with multiplicity) lies in $W_{\leq l} \cup \{1\}$.  By the fundamental theorem of arithmetic, the claim will now follow if one can show that $p^j | Q_0$ whenever $p \leq N_0$ and $p^j | q$. Since $j \leq \frac{\log q}{\log p} \leq \frac{l}{\log p}$ (recall our convention that $\log$ is to base $2$), and $p$ divides $N_0!$ at least $\lfloor \frac{N_0}{p}\rfloor$ times, it suffices to establish the inequality
$$ \frac{l}{\log p} \leq D \left\lfloor \frac{N_0}{p}\right\rfloor.$$
Since $D > \frac{2}{\rho} \geq \frac{l}{\log N_0}$, it suffices to show that
$$ \frac{\log N_0}{\log p} \leq (1+\eps_\rho) \left\lfloor \frac{N_0}{p} \right\rfloor$$
for $2 \leq p \leq N_0$, where $\eps_\rho$ is the positive quantity $\eps_\rho \coloneqq \frac{\rho D}{2}-1$. If we set $n \coloneqq \lfloor \frac{N_0}{p} \rfloor$, then $n \in [N_0/2]$ and $\frac{\log N_0}{\log p} \leq \frac{\log N_0}{\log N_0 - \log(n+1)}$, so after some rearranging we reduce to showing that
$$ \log(n+1) \leq \left(1 - \frac{1}{(1+\eps_\rho) n}\right) \log N_0$$
for all $n \in [N_0/2]$.  But this can be easily checked if $C^0_\rho$ (and hence $N_0$) is sufficiently large depending on $\rho$ (one can for instance check the cases $1 \leq n \leq N_0^{1/2}$ and $N_0^{1/2} < n \leq N_0/2$ separately).

We now see that the $P_{\leq l}$ are non-decreasing in $l$ with $\bigcup_{l \in \N} P_{\leq l} = \Z_+$.  We can therefore define the Ionescu--Wainger height $\Height(\alpha) = \Height_\rho(\alpha)$ of an arithmetic frequency $\frac{a}{q} \mod 1$, with $q \in \Z_+$ and $a \in [q]^\times$, by the formula
$$ \Height\left(\frac{a}{q} \mod 1\right) \coloneqq \inf \{ 2^l: l \in \N, q \in P_{\leq l} \}.$$

Now we prove Lemma \ref{mag-lem}.  The claim (i) is immediate from \eqref{2l-include}, with the final claim concerning $\frac{1}{p}\Z/\Z$ following from direct inspection of definitions.  For the first part of (ii) we observe that
\begin{equation}\label{qq}
(\Q/\Z)_{\leq l} = \bigcup_{q \in P_{\leq l}} \frac{1}{q}\Z/\Z
\end{equation}
so it suffices to show that $q \lesssim_\rho 2^{2^{\rho l}}$ for all $q \in P_{\leq l}$.  For $l > C_\rho^0$, we have from definition that
$$ q \leq Q_0 (2^l)^{D^2} \leq N_0^{DN_0} 2^{D^2 l}
\lesssim_\rho 2^{\rho^{-1} 2^{\rho l/2} + \rho^{-2} l}$$
giving the claim; in fact we obtain the slightly sharper bound
\begin{equation}\label{slightly-sharper}
q \lesssim_\rho 2^{O_\rho(2^{\rho l/2})}.
\end{equation}
For the second claim, we need to show that
\begin{align}\label{def:Q_l}
Q_{\leq l} \coloneqq \mathrm{lcm}(q \in \Z_+: q \in P_{\leq l} ) \lesssim_\rho 2^{O(2^l)}.    
\end{align}
The claim is trivial for $l \leq C^0_\rho$.  For $l>C^0_\rho$ we have
$$  \mathrm{lcm}(q \in \Z_+:  q \in P_{\leq l} ) = Q_0 \prod_{p \in (N_0^{(l)}, 2^l]\cap\PP} p^D.$$
From Mertens' theorem we have
$$  \prod_{p \in (N_0^{(l)}, 2^l]\cap\PP} p \lesssim 2^{O(2^l)}$$
and 
$$ Q_0 \leq N_0^{DN_0} \lesssim 2^{O_\rho(2^{\rho l})}$$
giving the claim. The claim (iii) follows from \eqref{qq} and \eqref{slightly-sharper}. This proves Lemma \ref{mag-lem}.

To establish  Theorem \ref{thm:iw}, we observe from Lemma \ref{mag-lem}(ii) and \eqref{slightly-sharper} that the elements of $(\Q/\Z)_{\leq l}$ are separated from each other by $\gtrsim_\rho 2^{-O_\rho(2^{\rho l/2})}$, giving the non-aliasing claim.  The claim \eqref{iw-mult} follows\footnote{The factor $\langle l\rangle$ in this theorem was recently removed in \cite{T}.} from \cite[Theorem 2.1]{MSZ3} (specialized to the one-dimensional case); various special cases of this theorem were previously established in \cite{IW}, see also Remark \ref{rem:1}(i). Note that on the right-hand side one can use the scalar norm rather than the vector-valued norm thanks to the Marcinkiewicz--Zygmund inequality (or Khintchine's inequality).  Finally, the claim for the multipliers \eqref{tlphi-def-2} follows from \eqref{iw-mult} and the triangle inequality.

Now we prove Lemma \ref{iw-prop}.  The Fourier support properties are clear from inspection and the disjointness of the individual major arcs.  The contraction property on $\ell^2$ follows from Plancherel's theorem because the symbol $\Proj \eta_{\leq k}$ is bounded pointwise by $1$.  To obtain the bound \eqref{pilk}, by interpolation we may assume that $q$ is either an even integer or the dual of an even integer.  Then it suffices from Theorem \ref{thm:iw} to establish the bound
\begin{equation}\label{easy}
\| \T_{\eta_{\leq k}} \|_{L^q(\R) \to L^q(\R)} \lesssim_q 1.
\end{equation}
But this follows from Lemma \ref{crude-mult} (with $r = 2^k$).

Finally we establish \eqref{off-decay}.  It suffices to establish the bound 
$$ \| \Pi_{\leq l, \leq k} f \|_{\ell^q(\{ n \in \Z: \mathrm{dist}(n,I) > 2^{m-k} \})} \lesssim_M 2^{-Mm} \|f\|_{\ell^q(I)}$$
for any $m \in \Z_+$.  By interpolation we may assume $q$ is either an even integer or the dual of an even integer.  By adjusting constants in the definition \eqref{m-small} of good major arcs if necessary we may assume that
$$ k \leq -2 v,$$
where
$$ v \coloneqq \lfloor C_\rho 2^{\rho l}\rfloor.$$
We split
$$ \eta_{\leq k} := \eta_{\leq k}^{(1)} + \eta_{\leq k}^{(2)} + \eta_{\leq k}^{(3)},$$
where $\eta_{\leq k}^{(1)}, \eta_{\leq k}^{(2)}, \eta_{\leq k}^{(3)} \in \Schwartz(\R)$ are the functions
\begin{align*}
    \eta_{\leq k}^{(1)} &:= \F_\R( \eta_{\leq m-k} \F_\R^{-1} \eta_{\leq k} ) \\
    \eta_{\leq k}^{(2)} &:= \F_\R( (1-\eta_{\leq m-k}) \F_\R^{-1} \eta_{\leq k} ) \eta_{\leq -v}\\
    \eta_{\leq k}^{(3)} &:= \F_\R( (1-\eta_{\leq m-k}) \F_\R^{-1} \eta_{\leq k} ) (1-\eta_{\leq -v}) \\
    &= - \F_\R( \eta_{\leq m-k} \F_\R^{-1} \eta_{\leq k} ) (1-\eta_{\leq -v})
\end{align*}
We can then decompose
$$ \Pi_{\leq l,\leq k} f = \T^{\leq l}_{\eta_{\leq k}} f = \T^{\leq l}_{\eta_{\leq k}^{(1)}} f + \T^{\leq l}_{\eta_{\leq k}^{(2)}} f + \T^{\leq l}_{\eta_{\leq k}^{(3)}} f.$$
Observe that the inverse Fourier transform of $\eta_{\leq k}^{(1)}$ is supported in $[-2^{m-k}, 2^{m-k}]$, and hence  $\T^{\leq l}_{\eta_{\leq k}^{(1)}} f$ vanishes on the region $\{ n \in \Z: \mathrm{dist}(n,I) > 2^{m-k}\}$.  For $\eta_{\leq k}^{(2)}$, we use Theorem \ref{thm:iw} (and the fact that $(k,-v)$ has good major arcs), \eqref{easy}, Young's inequality, and a rescaling to bound
\begin{align*}
    \| \T^{\leq l}_{\eta_{\leq k}^{(2)}} f \|_{\ell^q(\Z)} &\lesssim_q \langle l \rangle \| \T_{\eta_{\leq k}^{(2)}} \|_{L^q(\R) \to L^q(\R)} \|f\|_{\ell^q(I)} \\
    &\lesssim_q \langle l \rangle \| \T_{\F_\R( (1-\eta_{\leq m-k}) \F_\R^{-1} \eta_{\leq k} )} \|_{L^q(\R) \to L^q(\R)} 
 \|f\|_{\ell^q(I)} \\
    &\lesssim_q \langle l \rangle \|(1-\eta_{\leq m-k}) \F_\R^{-1} \eta_{\leq k} \|_{L^1(\R)} 
 \|f\|_{\ell^q(I)} \\
    &\lesssim_q \langle l \rangle \| \F_\R^{-1} \eta \|_{L^1(\R \backslash [-2^{m-1}, 2^{m-1}])} 
 \|f\|_{\ell^q(I)} 
\end{align*}
and hence the contribution of this term is acceptable by the rapid decrease of $\F_\R^{-1} \eta$.

Finally, for $\eta_{\leq k}^{(3)}$ we use Lemma \ref{crude-mult}(i) (with $r = 2^{k}$) and Lemma \ref{mag-lem}(iii) to bound
$$
\Big\| \T^{\leq l}_{\eta_{\leq k}^{(3)}} f \Big\|_{\ell^q(\Z)} \lesssim_{C_1,q} 2^{O(2^{\rho l})} \|f\|_{\ell^q(I)} \sup_{0 \leq j \leq 2} \int_\R 2^{k(1-j)} \left|\frac{d^j}{d \xi^j} \eta_{\leq k}^{(3)}(\xi)\right|\ d\xi.
$$
Direct calculation using the rapid decay of $\F_\R \eta$ shows that
$$ \int_\R 2^{k(1-j)} \bigg|\frac{d^j}{d \xi^j} \eta_{\leq k}^{(3)}(\xi)\bigg|\ d\xi \lesssim_M 2^{M(k-m+v)} \lesssim 2^{-Mm} 2^{-MC_\rho 2^{\rho l}}$$
and hence the contribution of this term is also acceptable (taking $C_\rho$ large enough).  This concludes the proof of Lemma \ref{iw-prop}.

\section{Shifted Calder\'on--Zygmund theory}\label{shift-app}

In this appendix we review some standard shifted Calder\'on--Zygmund estimates, of the sort that appear for instance in \cite[Lemma 4.8, pp. 346]{Lie2}.  For our appications we will need a vector-valued version of these estimates.

\begin{theorem}[Shifted Calder\'on--Zygmund estimates]\label{cz-shift}  Let $\D$ be a finite $\lambda$-lacunary set for some $\lambda > 1$, and let $A > 0$, $C>0$, $d \geq 1$, and $K \geq 1$.  For each $N \in \D$, let
 $\varphi_N \in \Schwartz(\R)$ be a function of the form
$$ \varphi_N(\xi) \coloneqq \psi(AN^d \xi) e(\lambda_N AN^d \xi)$$
for some $\lambda_N \in [-2^K,2^K]$, where $\psi \in \Schwartz(\R)$ vanishes at the origin and is supported on $[-C,C]$ for some $C>0$, obeying the derivative estimates
$$ \left|\frac{d^j}{d\xi^j} \psi(\xi)\right| \leq C$$
for all $j=0,1,2$ and $\xi \in \R$.  Then for any $1 < p < \infty$ and any separable Hilbert space $(H, \|\cdot\|_H)$, one has
\begin{equation}\label{form}
\| \T_{\sum_{N \in \D} \epsilon_N \varphi_N} \|_{L^p(\R;H) \to L^p(\R; H)} \lesssim_{C,\lambda,d,p} K
\end{equation}
for any complex numbers $\epsilon_N, N \in \D$ with $|\epsilon_N| \leq 1$; in particular, by Khintchine's inequality
$$ \| (\T_{\varphi_N})_{N \in \D} \|_{L^p(\R;H) \to L^p(\R; \ell^2(\D;H))} \lesssim_{C,\lambda,d,p} K.$$
\end{theorem}

\begin{proof} (Sketch) Let $\varphi:=\sum_{N \in \D} \epsilon_N \varphi_N$. From the hypotheses on $\psi$ one has the bound
$$ |\psi(\xi)| \lesssim_C |\xi| \ind{|\xi| \leq C}$$
and hence from the triangle inequality one has $\| \varphi \|_{L^\infty(\R)} \lesssim_{C,\lambda} 1$.  The $p=2$ case of the theorem then follows from Plancherel's theorem.  By duality it then suffices to establish the $1 < p < 2$ case, and by Marcinkiewicz interpolation it suffices to prove the weak-type $(1,1)$ bound
$$\left|\left \{ x \in \R:  \|\T_{\varphi} f(x)\|_H \geq \alpha \right\}\right| \lesssim_{C,\lambda,d} \frac{K}{\alpha} \|f\|_{L^1(\R;H)}$$
for $f \in L^1(\R;H)$ and $\alpha>0$.  We perform a vector-valued Calder\'on--Zygmund decomposition $f=g+\sum_{I \in \mathcal D} b_I$, where $\|g\|_{L^2(\R;H)}^2 \lesssim \|f\|_{L^1(\R;H)} \alpha$, $I$ ranges over a collection of dyadic intervals $\mathcal D$ with
$$ \sum_{I\in \mathcal D}  |I| \lesssim_{C,\lambda} \frac{1}{\alpha} \|f\|_{L^1(\R;H)},$$
and $b_I \in L^1(\R;H)$ is supported on $I$ with mean zero and 
\begin{equation}\label{bio}
\| b_I \|_{L^1(\R;H)} \lesssim |I|.
\end{equation}
By the previous inequality it suffices to prove
$$\left|\left \{ x \in \Big(\bigcup_{I\in \mathcal D} 100I\Big)^c:  \|\T_{\varphi} f(x)\|_H \geq \alpha \right\}\right| \lesssim_{C,\lambda,d} \frac{K}{\alpha} \|f\|_{L^1(\R;H)},$$
where $aI$ is the interval centered at $I$ of $a>0$ times the length.
By the triangle inequality and Markov's inequality, it thus suffices to show that
\begin{equation}\label{romeg}
 \int_{(100I)^c} \|\T_\varphi b_I(x)\|_H\ dx \lesssim_{C,\lambda} K |I|
\end{equation}
for each $I\in \mathcal D$.  We may expand
$$ \T_\varphi b_I(x) = \sum_{N \in \D} \epsilon_N \int_{\R} (AN^d)^{-1}\F_\R^{-1} \psi\bigg(\frac{x-y-\lambda_NAN^d}{AN^d}\bigg) b_I(y)\ dy.$$
We may assume that $I\in\mathcal D$ is centered at the origin, and 
exploiting the fact that $b_I$ has mean zero we may dominate the left-hand side of \eqref{romeg} by
$$
\sum_{N \in \D} \int_{(100I)^c}  \int_I \frac{1}{AN^d} \bigg|\F_\R^{-1}\psi\bigg(\frac{x-\lambda_NAN^d-y}{AN^d}\bigg)-\F_\R^{-1}\psi\bigg(\frac{x-\lambda_NAN^d}{AN^d}\bigg)\bigg| \|b_I(y)\|_H\ dydx.
$$
So by \eqref{bio} it suffices to show that
$$
\sum_{N \in \D}  \int_{(100I)^c} \frac{1}{AN^d} \bigg|\F_\R^{-1}\psi\bigg(\frac{x-\lambda_NAN^d-y}{AN^d}\bigg)-\F_\R^{-1}\psi\bigg(\frac{x-\lambda_NAN^d}{AN^d}\bigg)\bigg| dx \lesssim_{C,\lambda} K
$$
for all $y \in I$.

Fix $y,I$.  We perform a partition 
$$ \D = \D_{\mathrm{low}} \cup \D_{\mathrm{medium}} \cup \D_{\mathrm{high}}$$
where $\D_{\mathrm{low}}$ consists of those spatial scales $N \in \D$ that are ``low frequency'' (or ``coarse scale'') in the sense that $|I| \leq A N^d$, $\D_{\mathrm{medium}}$ consists of those spatial scales $N \in \D$ that are ``medium frequency'' (or ``medium scale'') in the sense that $\lambda_N^{-1} |I| \leq A N^d < |I|$, and $\D_{\mathrm{high}}$ consists of those spatial scales $N \in \D$ that are ``high frequency'' (or ``fine scale'') in the sense that $A N^d < \lambda_N^{-1} |I|$.

The expression
  $$ \F_\R^{-1}\psi\bigg(\frac{x-\lambda_NAN^d-y}{AN^d}\bigg)-\F_\R^{-1}\psi\bigg(\frac{x-\lambda_NAN^d}{AN^d}\bigg) $$
 can be bounded by $O_C( \langle \frac{x}{AN^d} \rangle^{-2} )$ in the high-frequency case $N \in \D_{\mathrm{high}}$ from the triangle inequality and the hypotheses $y \in I$, $x \in (100I)^c$, by $O_C( \langle \frac{x}{AN^d} - \lambda_N \rangle^{-2} )$ in the medium-frequency case $N \in \D_{\mathrm{medium}}$ from the triangle inequality alone, and by $O_C( \frac{|I|}{A N^d} \langle \frac{x}{AN^d} - \lambda_N \rangle^{-2} )$ in the low-frequency case using the mean-value theorem.  The claim then follows from direct computation and the hypothesis $|\lambda_N| \leq 2^K$.
\end{proof}

\section{Concentration estimates on polynomials}
\label{sec:app1}

In this appendix we work in a $p$-adic field $\QQ_p = \bigcup_{n \in \N} p^{-j} \Z_p$ for $p\in\PP$, although much
of the discussion here would also extend with minor changes to the
real numbers $\R$ or (after adjusting some exponents by factors of two) the
complex numbers $\C$, and the reader may wish to work with the real case
first to build intuition.  We have a norm on the $p$-adics defined by
$|x| \coloneqq p^{-\nu_p(x)}$, where $\nu_p$ is the usual $p$-valuation (with the
usual convention $|0|=0$), as well as a Haar measure $\mu_{\QQ_p}$
on $\QQ_p$ with the following properties for any $x,y \in \QQ_p$ and
$r \in p^\Z \coloneqq \{p^n:n\in\Z\}$:
\begin{enumerate}[label*={(\roman*)}]
\item (ultratriangle inequality) $|x+y| \leq \max(|x|, |y|)$.
\item (multiplicativity) $|xy| = |x| |y|$.
\item (nondegeneracy) $|x| \geq 0$, with equality if and only if $x=0$.
\item (dimension one) $\mu_{\QQ_p}(B(x,r))= r$, where $B(x,r) \coloneqq  \{ y \in \QQ_p: |y-x| \leq r\}$ is the usual ball. 
\end{enumerate}

Note that if $P$ is a polynomial with coefficients in $\QQ_p$, thus
\[
 P(x) = a_d x^d + \dots + a_1x+a_0
\]
for some $a_d,\dots,a_0 \in \QQ_p$, one can define the derivative $P'$
algebraically by the usual formula
\[
 P'(x) \coloneqq d a_d x^{d-1} + \dots + 2a_2x+ a_1.
\]

We then have the following basic estimates on the distribution of $p$-adic polynomials.

\begin{proposition}[Distribution of $p$-adic polynomials]\label{prop:zero} Let
$P(x) = a_d x^d + \dots + a_0$ be a polynomial of degree $d\ge1$ with
coefficients in $\QQ_p$.  Let $r \in p^\Z$, and
let $\Omega$ be the level set
$$ \Omega \coloneqq \{ x \in \QQ_p: |P(x)| \leq r \}.$$
\begin{enumerate}[label*={(\roman*)}]
\item[(i)] (Bernstein inequality) One can cover $\Omega$ by $O_d(1)$ balls
$B$, such that on each ball $B$ one has
\[
 \sup_{x \in B} |P'(x)| \lesssim_d \frac{r}{\mu_{\QQ_p}(B)}.
\]
\item[(ii)]  (Van der Corput estimate) We have
\[
 \mu_{\QQ_p}(\Omega) \lesssim_d \bigg(\frac{r}{|a_d|}\bigg)^{1/d}.
\]
In fact $\Omega$ is covered by $O_d(1)$ balls of radius $\big(\frac{r}{|a_d|}\big)^{1/d}$.
\item[(iii)]  (Distributional estimate) If $d \geq 2$, and $f \colon \QQ_p \to [0,+\infty)$ is the function
\[
 f(y) \coloneqq \frac{1}{r} \mu_{\QQ_p}(\{ x \in \QQ_p: |P(x) - y| \leq r \}),
\]
then
\[
\mu_{\QQ_p}(\{ y \in \QQ_p: f(y) \geq \lambda \}) \lesssim_d \lambda^{-\frac{d}{d-1}} |a_d|^{-\frac{1}{d-1}}.
\]
\end{enumerate}
\end{proposition}

A model example to keep in mind here is when $P(x) = a_d x^d$ is a monomial, in which case $\Omega$ consists of a single ball of radius $(r/|a_d|)^{1/d}$, with $P' = O_d( |a_d| (r/|a_d|)^{\frac{d-1}{d}} )$ on this ball; also, one can verify that $f(y) = O_d( |a_d|^{-1/d} r^{\frac{1}{d}-1})$ when $|y| \leq r$ and $f(y) = O_d( |a_d|^{-1/d} |y|^{\frac{1}{d}-1})$ when $|y| > r$.  (The reader may wish to first verify these claims with $\QQ_p$ replaced by $\R$ in order to build geometric intuition.)  Note that this example also shows why all the exponents in the proposition are natural from a dimensional analysis (or scaling) perspective.  Taking limits in (ii) as $r \to 0$, we also conclude that
$$ \left\| \frac{d P_* \mu_{\QQ_p}}{d\mu_{\QQ_p}}\right \|_{L^{\frac{d}{d-1},\infty}(\QQ_p)} \lesssim_d |a_d|^{-\frac{1}{d}},$$
where $\frac{d P_* \mu_{\QQ_p}}{d\mu_{\QQ_p}}$ is the Radon--Nikodym derivative (relative to Haar measure $\mu_{\QQ_p}$) of the pushforward measure $P_* \mu_{\QQ_p}$ of $\mu_{\QQ_p}$ by $P$, and $L^{\frac{d}{d-1},\infty}$ is the weak $L^{\frac{d}{d-1}}$ norm; in the monomial case $P(x) = a_d x^d$ one can compute that this Radon--Nikodym derivative is proportional to the function $y \mapsto |a_d|^{-\frac{1}{d}} |y|^{\frac{1}{d}-1}$.

The van der Corput estimate in Proposition \ref{prop:zero}(ii) can be also deduced from \cite[Proposition 3.3. pp. 847]{KoW}, but for the convenience of the reader we provide a self-contained proof.  

\begin{proof}  To prove (i), we first work in the special case that $P$ completely factorizes:
\[
 P(x) = c (x-\alpha_1) \cdots (x-\alpha_d)
\]
for some $c, \alpha_1,\dots,\alpha_d \in \QQ_p$ with $c \neq 0$.  We
can cover $\Omega$ by $\Omega_1 \cup \ldots \cup\Omega_d$, where
\[
 \Omega_i \coloneqq \{ x \in \Omega: |x-\alpha_i| \leq |x-\alpha_j| \text{ for all } j\in[d]\}.
\]
It suffices to establish the claim (i) for a single $\Omega_i$.  Note from
the ultratriangle inequality that for $x \in \Omega_i$ and
$j\in[d]$ one has
\[
 |x-\alpha_j| = \max\{ |x-\alpha_i|, |\alpha_i - \alpha_j| \},
\]
and hence
\[
 |P(x)| = |c| \prod_{j=1}^d \max\{ |x-\alpha_i|, |\alpha_i - \alpha_j| \}.
\]
Thus we see that $\Omega_i \subseteq B(\alpha_i, R)$, where $R \in p^\Z$ is the maximal quantity for which
\[
 |c| \prod_{j=1}^d \max\{ R, |\alpha_i - \alpha_j| \} \leq r.
\]
On the other hand, we have from the product rule and triangle inequality for $x \in B(\alpha_i, R)$ that
\begin{align*}
 |P'(x)| &\lesssim_d |c| \sup_{j\in[d]} \prod_{k \neq j} |x-\alpha_k|\\
&\lesssim_d |c| \sup_{j\in[d]} \prod_{k \neq j} \max\{R, |\alpha_i-\alpha_k|\} \\
&\lesssim_d R^{-1} |c| \prod_{k=1}^d \max\{R, |\alpha_i-\alpha_k|\}\\
&\lesssim_d \frac{r}{R}
\end{align*}
giving the claim (i).

Now suppose that $P$ only partially factorizes, thus
\[
P(x) = (x-\alpha_1) \cdots (x-\alpha_j) Q(x)
\]
for some $0 \leq j \leq d$ and some polynomial $Q$ of degree $d-j$.
The case $j=d$ has already been handled; now suppose inductively that
$j < d$ and the claim (i) has already been proven for $j+1$.  We may
assume $\Omega$ is non-empty since the claim (i) is trivial otherwise.
Let $\alpha_{j+1}$ be an element of $\Omega$ which maximizes the
magnitude of the quantity
$\delta \coloneqq (\alpha_{j+1}-\alpha_1) \cdots (\alpha_{j+1}-\alpha_j)$;
such a quantity exists since $\Omega$ is compact, and $\delta$ is
non-zero by continuity.  Then
\[
 r \geq |P(\alpha_{j+1})| = |\delta| |Q(\alpha_{j+1})|,
\]
so $|Q(\alpha_{j+1})| \leq r/|\delta|$.  By the factor theorem we have
\[
 Q(x) = Q(\alpha_{j+1}) + (x-\alpha_{j+1}) R(x)
\]
for some polynomial $R$ of degree $d-j-1$, thus
\[
 P(x) = (x-\alpha_1) \cdots (x-\alpha_j) Q(\alpha_{j+1}) + (x-\alpha_1) \cdots (x-\alpha_{j+1}) R(x).
\]
By construction, for $x \in \Omega$ we have $|P(x)| \leq r$, and
\[
 |(x-\alpha_1) \cdots (x-\alpha_j) Q(\alpha_{j+1})| \leq |\delta| |Q(\alpha_{j+1})| \leq r,
\]
hence by the ultratriangle inequality we also have
\[
 |(x-\alpha_1) \cdots (x-\alpha_{j+1}) R(x)| \leq r.
\]
By the induction hypothesis we can cover $\Omega$ by $O_d(1)$ balls
$B$ on which the derivative of
$(x-\alpha_1) \dots (x-\alpha_{j+1}) R(x)$ is $O_d(r/\mu_{\QQ_p}(B))$; by the
$j=d$ case we can also say the same about
$(x-\alpha_1) \dots (x-\alpha_j) Q(\alpha_{j+1})$.  Intersecting the
balls together, we can say the same about $P$.  This closes the
induction and establishes the claim for any $0 \leq j \leq d$.
Setting $j=0$, we obtain (i).

Now we establish (ii).  By iterating (i) $d$ times and intersecting the balls together, we can cover
$\Omega$ by $O_d(1)$ balls $B$ on which
$P^{(d)}(x) \lesssim_d r/\mu_{\QQ_p}(B)^d$.  But since $P^{(d)}(x) = d! a_d$, we
have $\mu_{\QQ_p}(B) \lesssim_d (r/|a_d|)^{1/d}$, giving the claim.

Now we prove (iii).  Let $\lambda>0$, and define the set
\[
 E \coloneqq \{ y \in \QQ_p: f(y) \geq \lambda \}.
\]
Our task is to show that
\[
\mu_{\QQ_p}(E) \lesssim_d \lambda^{-\frac{d}{d-1}} |a_d|^{-\frac{1}{d-1}}.
\]
If $y \in E$, then by definition
\[
 \mu_{\QQ_p}(\{ x \in \QQ_p: |P(x) - y| \leq r \}) \geq \lambda r.
\]
By (i), the set in the left-hand side can be covered by $O_d(1)$ balls
$B$, on which $|P'| \lesssim_d r/\mu_{\QQ_p}(B)$.  By the pigeonhole principle,
one of these balls $B$ must intersect the set in a set of measure
$\gtrsim_d \lambda r$, thus $|P'| \lesssim_d r / (\lambda r) = 1/\lambda$ on this ball,
and thus
\[
 \mu_{\QQ_p}(\{ x \in \QQ_p: |P(x) - y| \leq r \text{ and } |P'(x)| \lesssim_d 1/\lambda \})  \gtrsim_d \lambda r.
\]
By the Fubini--Tonelli theorem we conclude that
\[
 \mu_{\QQ_p}\times \mu_{\QQ_p}(\{ (x,y) \in \QQ_p^2: |P(x) - y| \leq r \text{ and } |P'(x)| \lesssim_d 1/\lambda \})  \gtrsim_d \lambda r \mu_{\QQ_p}(E).
\]
But by the Fubini--Tonelli theorem again, the left-hand side is equal to
\[
 r \mu_{\QQ_p}(\{ x \in \QQ_p: |P'(x)| \lesssim_d 1/\lambda \})
\]
and hence by (ii) we obtain
\[
 \lambda r \mu_{\QQ_p}(E) \lesssim_d r \bigg(\frac{1}{\lambda |a_d|}\bigg)^{\frac{1}{d-1}},
\]
giving the claim.
\end{proof}

We can descend from the $p$-adics to a cyclic group of prime power order:

\begin{corollary}[Distribution of polynomials on a cyclic group of prime power order]
\label{cor:1}
Let $Q = p^j$ for some $j \in\Z_+$, and let
$P \in \Z[\nn]$ be a polynomial of degree $d \geq 2$, which we also view as a map from $\Z/Q\Z$ to itself.  Let
$h \colon \Z/Q\Z \to \N$ be the counting function
\[
 h(y) \coloneqq \# \{ x \in \Z/Q\Z: P(x) = y \}.
\]
Then for any $\lambda>0$ we have the weak-type bound
\[
 \# \{ y \in \Z/Q\Z: h(y) \geq \lambda \} \lesssim_{P} \lambda^{-\frac{d}{d-1}} Q.
 \]
In particular, one has
\begin{align}\label{h-est}
\| h \|_{L^s(\Z/Q\Z)} \lesssim_{s,P} 1 
\end{align}
for any $0 < s < \frac{d}{d-1}$.
\end{corollary}

As before, the example of a monomial $P(x) = x^d$ shows that the range of $s$ here is best possible. Interestingly, it seems difficult to establish this corollary without some version of the $p$-adic formalism, even though the statement of the corollary does not explicitly mention $p$-adics. Estimate \eqref{h-est} was previously obtained for monomials $P(x)=x^d$ in an unpublished work of Jim Wright on $L^p$-improving estimates for averaging operators on cyclic groups of the form $\Z/p^j\Z$ (private communication).

\begin{proof} We can write $P(x) = a_d x^d + \dots + a_0$, where
$a_0,\dots,a_d \in \Z_p$ are $p$-adic integers, thus they have norm at most
$1$.  Note that
\[
 h(y) = Q \mu_{\QQ_p}( \{ x \in \QQ_p: |x| \leq 1 \text{ and } |P(x) - y'| \leq Q^{-1} \}),
\]
for any $y\in\Z/Q\Z$ and $y' \in B(y,1/Q)$, thus
\[
\# \{ y \in \Z/Q\Z: h(y) \geq \lambda \}
\leq Q \mu_{\QQ_p}(\{ y' \in \QQ_p: Q \mu_{\QQ_p}(\{x\in\QQ_p: |P(x) - y'| \leq Q^{-1} \}) \geq \lambda \}),
\]
and the claim now follows from Proposition \ref{prop:zero}(iii).
\end{proof}

\end{document}